\documentclass[12pt,english]{amsart}
\usepackage[T1]{fontenc}
\usepackage[latin9]{inputenc}
\usepackage{color}
\usepackage{babel}

\usepackage{amsthm}
\usepackage{amstext}
\usepackage{graphicx}
\usepackage{amssymb}
\usepackage{esint}
\usepackage[unicode=true, 
 bookmarks=true,bookmarksnumbered=false,bookmarksopen=false,
 breaklinks=false,pdfborder={0 0 1},backref=false,colorlinks=true,pdfpagemode=FullScreen]
 {hyperref}
\hypersetup{pdftitle={Normal Functions},
 pdfauthor={Kerr and Pearlstein}}
 
\makeatletter
\numberwithin{equation}{section} 
\numberwithin{figure}{section} 
\theoremstyle{plain}
\theoremstyle{plain}
\newtheorem{thm}{Theorem}
  \theoremstyle{plain}
  \newtheorem{conjecture}[thm]{Conjecture}
  \theoremstyle{definition}
  \newtheorem{defn}[thm]{Definition}
  \theoremstyle{plain}
  \newtheorem{prop}[thm]{Proposition}
 \theoremstyle{definition}
  \newtheorem{example}[thm]{Example}
  \theoremstyle{remark}
  \newtheorem{rem}[thm]{Remark}
  \theoremstyle{remark}
  \newtheorem{notation}[thm]{Notation}
  \theoremstyle{plain}
  \newtheorem{lem}[thm]{Lemma}
  \theoremstyle{plain}
  \newtheorem{cor}[thm]{Corollary}

\usepackage{amsmath}
\usepackage{amsfonts}
\usepackage{amssymb}
\usepackage{babel}
\input xy
\xyoption{all}

\def\FF{\mathbb{F}}
\def\CC{\mathbb{C}}
\def\RR{\mathbb{R}}
\def\NN{\mathbb{N}}

\def\GG{\mathbb{G}}
\def\QQ{\mathbb{Q}}

\def\ZZ{\mathbb{Z}}
\def\PP{\mathbb{P}}

\def\VV{\mathbb{V}}
\def\HH{\mathbb{H}}

\def\A{\mathcal{A}}

\def\D{\mathcal{D}}

\def\E{\mathcal{E}}
\def\F{\mathcal{F}}

\def\H{\mathcal{H}}

\def\J{\mathcal{J}}
\def\K{\mathcal{K}}
\def\L{\mathcal{L}}
\def\M{\mathcal{M}}

\def\T{\mathcal{T}}

\def\V{\mathcal{V}}
\def\W{\mathcal{W}}
\def\X{\mathcal{X}}

\def\Z{\mathcal{Z}}

\def\e{\epsilon}

\def\sZ{\mathfrak{Z}}

\def\ANF{\text{ANF}}
\def\AJ{AJ}
\def\sing{\text{sing}}
\def\IH{\mathit{IH}}
\def\MHM{\text{MHM}}
\def\MHS{\text{MHS}}
\def\pd{\partial}
\def\AH{\A \H}
\def\DbMHM{D^b \MHM}
\def\Hom{Hom}

\def\DbMHS{D^b \MHS}
\def\zl{\Z_0}
\def\pd{\partial}
\def\Dmod{\mathcal{D}}
\def\Mmod{\mathcal{M}}
\def\Sbar{\bar{S}}
\def\shO{\mathcal{O}}
\def\shH{\mathcal{H}}
\def\shHO{\shH_{\shO}}

\def\TZ{T_{\ZZ}}
\DeclareMathOperator{\Spec}{Spec}
\DeclareMathOperator{\Sym}{Sym}
\def\Jbar{\bar{J}}

\makeatother

\begin{document}

\title[Normal Functions]{An Exponential History of Functions with 
Logarithmic Growth}

\author[Kerr and Pearlstein]{Matt Kerr and Gregory Pearlstein}
\begin{abstract}
We survey recent work on normal functions, including limits and singularities
of admissible normal functions, the Griffiths-Green approach to the
Hodge conjecture, algebraicity of the zero-locus of a normal function,
N\'eron models, and Mumford-Tate groups. Some of the material and
many of the examples, esp. in $\S\S5-6$, are original.
\end{abstract}
\maketitle
In a talk on the theory of motives, A. A. Beilinson remarked that
according to his time-line of results, advances in the (relatively
young) field were apparently a logarithmic function of $t$; hence,
one could expect to wait 100 years for the next significant milestone.
Here we allow ourselves to be more optimistic: following on a drawn-out
history which begins with Poincar\'e, Lefschetz, and Hodge, the theory
of \emph{Normal Functions} reached maturity in the programs of Bloch,
Griffiths, Zucker, and others. But the recent blizzard of results
and ideas, inspired by works of M. Saito on admissible normal functions,
and Green and Griffiths on the Hodge Conjecture, has been impressive
indeed. In addition to further papers of theirs, significant progress
has been made in work of P. Brosnan, F. Charles, H. Clemens, H. Fang,
J. Lewis, R. Thomas, Z. Nie, C. Schnell, C. Voisin, A. Young, and
the authors --- much of this in the last 4 years. This seems like
a good time to try to summarize the state of the art and speculate
about the future, barring (say) 100 more results between the time
of writing and the publication of this volume.

In the classical algebraic geometry of curves, Abel's theorem and
Jacobi inversion articulate the relationship (involving rational integrals)
between configurations of points with integer multiplicities, or zero-cycles,
and an abelian variety known as the Jacobian of the curve: the latter
algebraically parametrizes the cycles of degree 0 modulo the subgroup
arising as divisors of meromorphic functions. Given a family $\X$
of algebraic curves over a complete base curve $S$, with smooth fibers
over $S^{*}$ ($S$ minus a finite point set $\Sigma$ over which
fibers have double point singularities), Poincar\'e \cite{P1,P2}
defined \emph{normal functions} as holomorphic sections of the corresponding
family of Jacobians over $S$ which behave {}``normally'' (or {}``logarithmically'')
in some sense near the boundary. His main result, which says essentially
that they parametrize $1$-dimensional cycles on $\X$, was then used
by Lefschetz (in the context where $\X$ is a pencil of hyperplane
sections of a projective algebraic surface) to prove his famous $(1,1)$
theorem for algebraic surfaces \cite{L}. This later became the basis
for the Hodge conjecture, which says that certain \emph{topological-analytic}
invariants of an \emph{algebraic} variety must come from \emph{algebraic}
subvarieties:
\begin{conjecture}
\label{conj de Hodge}For a smooth projective complex algebraic variety
$X$, with $\mathit{Hg}^{m}(X)_{\QQ}$ the classes in $H_{sing}^{2m}(X_{\CC}^{an},\QQ)$
of type $(m,m)$, and $CH^{m}(X)$ the {}``Chow group'' of codimension-$m$
algebraic cycles modulo rational equivalence, the fundamental class
map $CH^{m}(X)\otimes\QQ\to\mathit{Hg}^{m}(X)_{\QQ}$ is surjective.
\end{conjecture}
Together with a desire to learn more about the structure of Chow groups
(the Bloch-Beilinson conjectures reviewed in $\S5$), this can be
seen as the primary motivation behind all the work described (as well
as the new results) in this paper. In particular, in $\S1$ (after
mathematically fleshing out the Poincare-Lefschetz story) we describe
the attempts to directly generalize Lefschetz's success to higher-codimension
cycles which led to Griffiths's Abel-Jacobi map (from the codimension
$m$ cycle group of a variety $X$ to its $m^{\text{th}}$ {}``intermediate''
Jacobian), horizontality and variations of mixed Hodge structure,
and S. Zucker's {}``Theorem on Normal Functions''. As is well-known,
the breakdown (beyond codimension 1) of the relationship between cycles
and (intermediate) Jacobians, and the failure of the Jacobians to
be algebraic, meant that the same game played in 1 parameter would
not work outside very special cases. 

It has taken nearly three decades to develop the technical underpinnings
for a study of normal functions over a \emph{higher} dimensional base
$S$: Kashiwara's work on admissible variations of mixed Hodge structure
\cite{K}, M. Saito's introduction of mixed Hodge modules \cite{S4},
multivariable nilpotent and $SL_{2}$ orbit theorems (\cite{KNU},\cite{Pe2}),
and so on. And then in 2006, Griffiths and Green had a fundamental
idea tying the Hodge conjecture to the presence of \emph{nontorsion
singularities} --- nontrivial invariants in local intersection cohomology
--- for multiparameter normal functions arising from Hodge classes
on algebraic varieties \cite{GG}. We describe their main result and
the follow-up work \cite{BFNP} in $\S3$. Prior to that the reader
will need some familiarity with the boundary behavior of {}``admissible''
normal functions arising from higher codimension algebraic cycles.
The two principal invariants of this behavior are called \emph{limits}
and \emph{singularities}, and we have tried in $\S2$ to give the
reader a geometric feel for these through several examples and an
explanation of the precise sense in which the limit of Abel-Jacobi
invariants (for a family of cycles) is again some kind of Abel-Jacobi
invariant. In general throughout $\S\S$1-2 (and $\S$4.5-6) normal
functions are {}``of geometric origin'' (arise from cycles), whereas
in the remainder the formal Hodge-theoretic point of view dominates
(though Conjecture \eqref{conj de Hodge} is always in the background).
We should emphasize that the first two sections are intended for a
broad audience, while the last four are of a more specialized nature;
one might say that the difficulty level increases exponentially.

The transcendental (non-algebraic) nature of intermediate Jacobians
means that even for a normal function of geometric origin, algebraicity
of its vanishing locus (as a subset of the base $S$), let alone its
sensitivity to the field of definition of the cycle, is not a foreordained
conclusion. Following a review of Schmid's nilpotent and $SL_{2}$
orbit theorems (which lie at the heart of the limit mixed Hodge structures
introduced in $\S2$), in $\S4$ we explain how generalizations of
those theorems to mixed Hodge structures (and multiple parameters)
have allowed complex algebraicity to be proved for the zero-loci of
{}``abstract'' admissible normal functions \cite{BP1,BP2,BP3,S5}.
We then address the field of definition in the geometric case, in
particular the recent result of Charles \cite{Ch} under a hypothesis
on the VHS underlying the zero-locus, the situation when the family
of cycles is algebraically equivalent to zero, and what all this means
for filtrations on Chow groups. Another reason one would want the
zero-locus to be algebraic is that the Griffiths-Green normal function
attached to a nontrivial Hodge class can then be shown, by an observation
of C. Schnell, to have a singularity in the intersection of the zero-locus
with the boundary $\Sigma\subset S$ (though this intersection could
very well be empty). 

Now,\emph{ a priori}, admissible normal functions (ANF's) are only
horizontal and holomorphic sections of a Jacobian bundle over $S\backslash\Sigma$
which are highly constrained along the boundary. Another route (besides
orbit theorems) that leads to algebraicity of their zero-loci is the
construction of a {}``N\'eron model'' --- a partial compactification
of the Jacobian bundle satisfying a Hausdorff property (though not
a complex analytic space in general) and graphing admissible normal
functions over all of $S$. N\'eron models are taken up in $\S5$;
as they are better understood they may become useful in defining global
invariants of (one or more) normal functions. However, unless the
underlying variation of Hodge structure (VHS) is a nilpotent orbit
the group of components of the N\'eron model (i.e., the possible
singularities of ANF's at that point) over a codimension$\geq2$ boundary
point remains mysterious. Recent examples of M. Saito \cite{S6} and
the second author \cite{Pe3} show that there are analytic obstructions
which prevent ANF's from surjecting onto (or even mapping nontrivially
to) the putative singularity group for ANF's (rational $(0,0)$ classes
in the local intersection cohomology). At first glance this appears
to throw the existence of singularities for Griffiths-Green normal
functions (and hence the Hodge conjecture) into serious doubt, but
in $\S5.5$ we show that this concern is probably ill-founded.

The last section is devoted to a discussion of Mumford-Tate groups
of mixed Hodge structures (introduced by Y. Andr\'e \cite{An}) and
variations thereof, in particular those attached to admissible normal
functions. The motivation for writing this section was again to attempt
to {}``force singularities to exist'' via conditions on the normal
function (e.g., involving the zero-locus) which maximize the monodromy
of the underlying local system inside the M-T group; we were able
to markedly improve Andr\'e's maximality result (but not to produce
singularities). Since the general notion of (non)singularity of a
VMHS at a boundary point is defined here (in $\S6.3$), which generalizes
the notion of singularity of a normal function, we should point out
that there is another sense in which the word {}``singularity''
is used in this paper. The \textquotedbl{}singularities'' of a \emph{period
mapping associated to a VHS or VMHS} are points where the connection
has poles or the local system has monodromy (i.e. $\Sigma$ in the
above notation), and at which one must compute a limit mixed Hodge
structure (LMHS). These contain the {}``singularities of the VMHS'',
nearly always as a \emph{proper} subset; indeed, pure VHS never have
singularities (in the sense of $\S6.3$), though their corresponding
period mappings do.

This paper has its roots in the first author's talk at a conference
in honor of Phillip Griffiths's 70th birthday at the IAS, and the
second author's talk at MSRI during the conference on the topology
of stratified spaces to which this volume is dedicated. The relationship
between normal functions and stratifications occurs in the context
of mixed Hodge modules and the Decomposition Theorem \cite{BBD},
and is most explicitly on display in the construction of the multivariable
N\'eron model in $\S5$ as a topological group whose restrictions
to the strata of a Whitney stratification are complex Lie groups.
We want to thank the conference organizers and Robert Bryant for doing
an excellent job at putting together and hosting a successful interdisciplinary
meeting blending (amongst other topics) singularities and topology
of complex varieties, $L^{2}$ and intersection cohomology, and mixed
Hodge theory, all of which play a role below. We are indebted to Patrick
Brosnan, Phillip Griffiths, and James Lewis for helpful conversations
and sharing their ideas. We also want to thank heartily both referees
as well as Chris Peters, whose comments and suggestions have made
this a better paper.

One observation on notation is in order, mainly for experts: in order
to clarify the distinction in some places between monodromy weight
filtrations arising in LMHS and weight filtrations postulated as part
of the data of an admissible variation of mixed Hodge structure (AVMHS),
the former are always denoted $M_{\bullet}$ (and the latter $W_{\bullet}$)
in this paper. In particular, for a degeneration of (pure) weight
$n$ HS with monodromy logarithm $N$, the weight filtration on the
LMHS is written $M(N)_{\bullet}$ (and centered at $n$). While perhaps
nontraditional, this is consistent with the notation $M(N,W)_{\bullet}$
for relative weight monodromy filtrations for (admissible) degenerations
of MHS. That is, when $W$ is {}``trivial'' ($W_{n}=\H$, $W_{n-1}=\{0\}$)
it is simply omitted.

Finally, we would like to draw attention to the interesting recent
article \cite{Gr4} of Griffiths which covers ground related to our
$\S\S2-5$, but in a complementary fashion that may also be useful
to the reader.

\tableofcontents{}

\section{Prehistory and Classical Results}

The present chapter is not meant to be heroic, but merely aims to
introduce a few concepts which shall be used throughout the paper.
We felt it would be convenient (whatever one's background) to have
an up-to-date, {}``algebraic'' summary of certain basic material
on normal functions and their invariants in one place. For background
or further (and better, but much lengthier) discussion of this material
the reader may consult the excellent books \cite{Le1} by Lewis and
\cite{Vo2} by Voisin, as well as the lectures of Green and Voisin
from the {}``Torino volume'' \cite{GMV} and the papers \cite{Gr1},
\cite{Gr2}, \cite{Gr3} of Griffiths.

Even experts may want to glance this section over since we have included
some bits of recent provenance: the relationship between log-infinitesimal
and topological invariants, which uses work of M. Saito; the result
on inhomogeneous Picard-Fuchs equations, which incorporates a theorem
of M\"uller-Stach and del Angel; the important example of Morrison
and Walcher related to open mirror symmetry; and the material on $K$-motivation
of normal functions (cf. $\S\S1.3,1.7$), which will be used in $\S\S$
2 and 4.

Before we begin, a word on the \emph{currents} that play a r\^ole
in the bullet-train proof of Abel's Theorem in $\S1.1$. These are
differential forms with distribution coefficients, and may be integrated
against $C^{\infty}$ forms, with exterior derivative $d$ defined
by {}``integration by parts''. They form a complex computing $\CC$-cohomology
(of the complex manifold on which they lie) and include $C^{\infty}$chains
and log-smooth forms. For example, for a $C^{\infty}$ chain $\Gamma$,
the delta current $\delta_{\Gamma}$ has the defining property $\int\delta_{\Gamma}\wedge\omega=\int_{\Gamma}\omega$
for any $C^{\infty}$ form $\omega$. (For more details, see Chap.
3 of \cite{GH}.)

\subsection{Abel's Theorem}

Our (historically incorrect) story begins with a divisor $D$ of degree
zero on a smooth projective algebraic curve $X/\CC$; the associated
analytic variety $X^{an}$ is a Riemann surface. (Except when explicitly
mentioned, we continue to work over $\CC$.) Writing $D=\sum_{_{\text{finite}}}n_{i}p_{i}\in Z^{1}(X)_{\text{hom}}$
($n_{i}\in\ZZ$ such that $\sum n_{i}=0$, $p_{i}\in X(\CC)$), by
Riemann's Existence Theorem one has a meromorphic $1$-form $\hat{\omega}$
with $Res_{p_{i}}(\hat{\omega})=n_{i}$ ($\forall i$). Denoting by
$\{\omega_{1},\ldots,\omega_{g}\}$ a basis for $\Omega^{1}(X)$,
consider the map \begin{equation}
\xymatrix{ Z^1(X)_{\text{hom}} \ar [r] \ar @/^2pc/ [rr]^{\widetilde{AJ}} & \frac{\Omega^1(X)^{\vee}}{\int_{H_1(X,\ZZ)}(\cdot)} \ar [r]^{ev_{\left\{ \omega_i \right\} } \mspace{20mu}}_{\cong \mspace{20mu}} & \frac{\CC^g}{\Lambda^{2g}} =: J^1(X) \\ D \ar @{|->} [r] & \int_{\Gamma} \ar @{|->} [r] & \left( \int_{\Gamma} \omega_1, \ldots, \int_{\Gamma} \omega_g \right) } \\
\end{equation} where $\Gamma\in C_{1}(X^{an})$ is any chain with $\partial\Gamma=D$

$\mspace{100mu}$\includegraphics[scale=0.5]{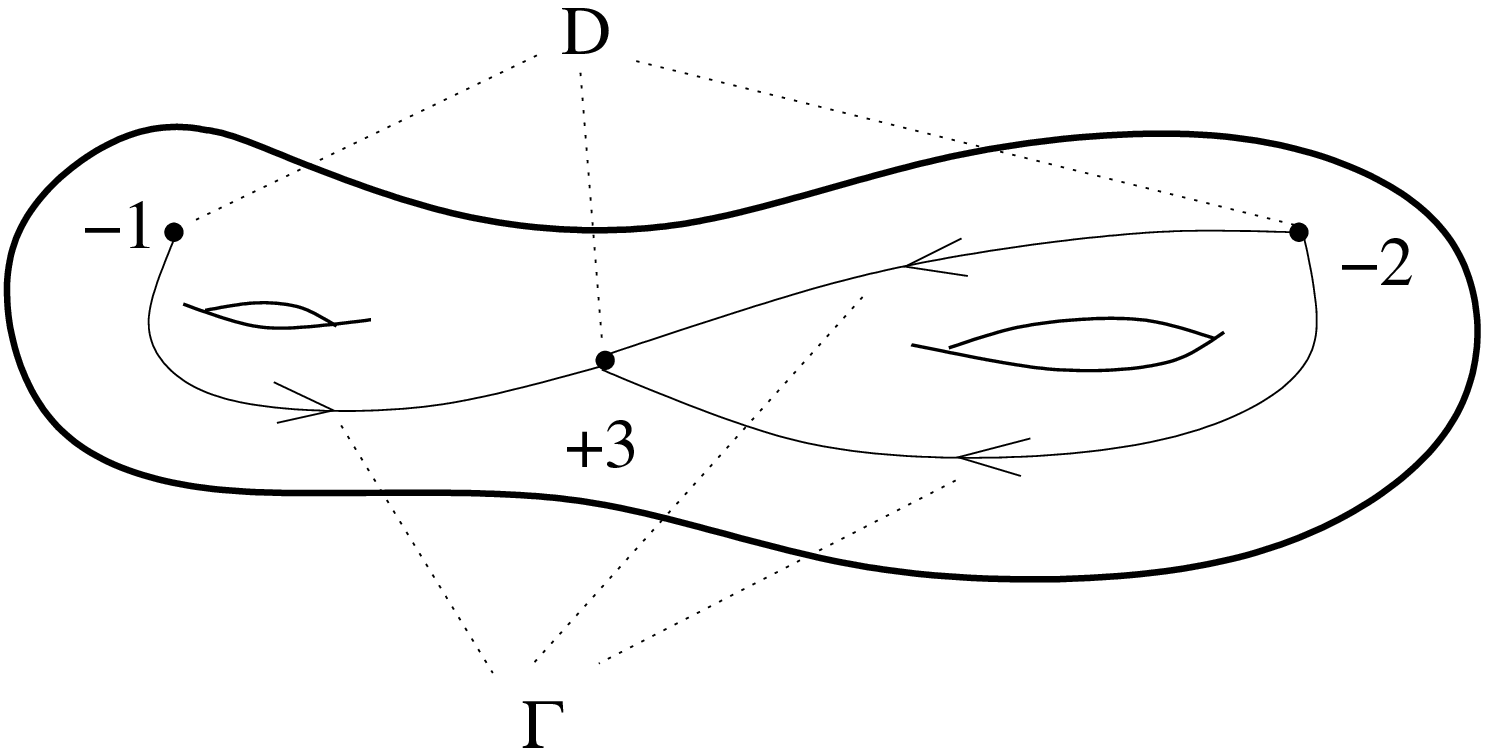}

and $J^{1}(X)$ is the\emph{ Jacobian} of $X$. The $1$-current $\kappa:=\hat{\omega}-2\pi i\delta_{\Gamma}$
is closed; moreover, if $\widetilde{AJ}(D)=0$ then $\Gamma$ may
be chosen so that all $\int_{\Gamma}\omega_{i}=0$ $\implies$ $\int_{X}\kappa\wedge\omega_{i}=0$.
We can therefore smooth $\kappa$ in its cohomology class to $\omega=\kappa-d\eta$
($\omega\in\Omega^{1}(X)$; $\eta\in D^{0}(X)=$0-currents), and \begin{equation}\label{eqn abel1}
f:= \exp \left\{ \int \left( \hat{\omega}-\omega \right) \right\} \\
\end{equation}
\begin{equation}
\label{eqn abel2}= e^{2\pi i \int \delta_{\Gamma}} e^{\eta} \\
\end{equation} is single-valued (though possibly discontinuous) by (\ref{eqn abel2})
and meromorphic (though possibly multivalued) by (\ref{eqn abel1}).
Locally at $p_{i}$, $e^{\int\frac{n_{i}}{z}dz}=Cz^{n_{i}}$ has the
right degree; and so the divisor of $f$ is precisely $D$. Conversely,
if for $f\in\CC(X)^{*}$, $D=(f)=f^{-1}(0)-f^{-1}(\infty)$, then
$t\mapsto\int_{f^{-1}(\overrightarrow{0.t})}(\cdot)$ induces a holomorphic
map $\PP^{1}\to J^{1}(X)$. Such a map is necessarily constant (say,
to avoid pulling back a nontrivial holomorphic $1$-form), and by
evaluating at $t=0$ one finds that this constant is zero. So we have
proved part (i) of
\begin{thm}
\label{thm Abel}(i) \emph{{[}Abel{]}} Writing $Z^{1}(X)_{\text{rat}}$
for the divisors of functions $f\in\CC(X)^{*}$, $\widetilde{AJ}$
descends to an injective homomorphism of abelian groups\[
CH^{1}(X)_{\text{hom}}:=\frac{Z^{1}(X)_{\text{hom}}}{Z^{1}(X)_{\text{rat}}}\overset{AJ}{\longrightarrow}J^{1}(X).\]

(ii) \emph{{[}Jacobi Inversion{]}} $AJ$ is surjective; in particular,
fixing $q_{1},\ldots,q_{g}\in X(\CC)$ the morphism $Sym^{g}X\to J^{1}(X)$
induced by $p_{1}+\cdots+p_{g}\mapsto\int_{\partial^{-1}(\sum p_{i}-q_{i})}(\cdot)$
is birational.
\end{thm}
Here $\partial^{-1}D$ means any $1$-chain bounding on $D$. Implicit
in (ii) is that $J^{1}(X)$ is an (abelian) algebraic variety; this
is a consequence of ampleness of the theta line bundle (on $J^{1}(X)$)
induced by the polarization\[
Q:\, H^{1}(X,\ZZ)\times H^{1}(X,\ZZ)\to\ZZ\]
(with obvious extensions to $\QQ$, $\RR$, $\CC$) defined equivalently
by cup product, intersection of cycles, or integration $(\omega,\eta)\mapsto\int_{X}\omega\wedge\eta$.
The ampleness boils down to the \emph{second Riemann bilinear relation},
which says that $iQ(\cdot,\bar{\cdot})$ is positive definite on $\Omega^{1}(X)$.

\subsection{Normal functions}

We now wish to vary the Abel-Jacobi map in families. Until $\S2$,
all our normal functions shall be over a curve $S$. Let $\X$ be
a smooth projective surface, and $\bar{\pi}:\X\to S$ a (projective)
morphism which is 

(a) smooth off a finite set $\Sigma=\{s_{1},\ldots,s_{e}\}\subset S$,
and

(b) locally of the form $(x_{1},x_{2})\mapsto x_{1}x_{2}$ at singularities
(of $\bar{\pi}$).

Write $X_{s}:=\bar{\pi}^{-1}(s)$ ($s\in S$) for the fibres. The
singular fibres $X_{s_{i}}$ ($i=1,\ldots,e$) then have only nodal
(ordinary double point) singularities

$\mspace{150mu}$\includegraphics[scale=0.4]{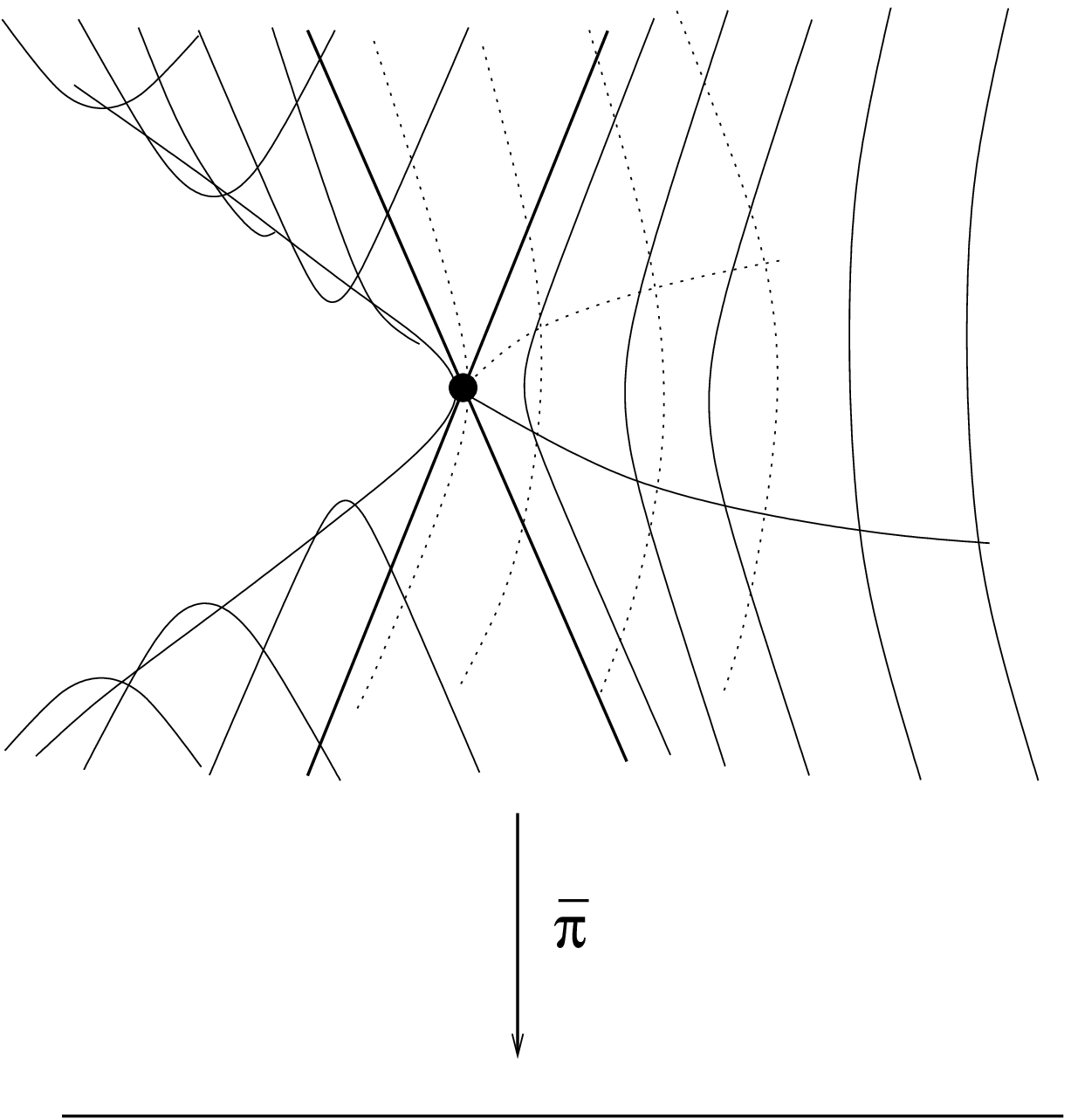}

and writing $\X^{*}$ for their complement we have $\pi:\X^{*}\to S^{*}:=S\backslash\Sigma$.
Fixing a general $s_{0}\in S^{*}$, the local monodromies $T_{s_{i}}\in$\linebreak$Aut\left(H^{1}(X_{s_{0}},\ZZ)=:\HH_{\ZZ,s_{0}}\right)$
of the local system $\HH_{\ZZ}:=R^{1}\pi_{*}\ZZ_{\X^{*}}$ are then
computed by the Picard-Lefschetz formula \begin{equation}\label{eqn PL}
(T_{s_i}-I)\gamma = \sum_j (\gamma \cdot \delta_j)\delta_j . \\
\end{equation} Here $\{\delta_{j}\}$ are the Poincar\'e duals of the (quite possibly
non-distinct) vanishing cycle classes $\in\ker\left\{ H_{1}(X_{s_{0}},\ZZ)\to H_{1}(X_{s_{i}},\ZZ)\right\} $
associated to each node on $X_{s_{i}}$; we note $(T_{s_{i}}-I)^{2}=0$.
For a family of elliptic curves, (\ref{eqn PL}) is just the familiar
\emph{Dehn twist:}

\includegraphics[scale=0.5]{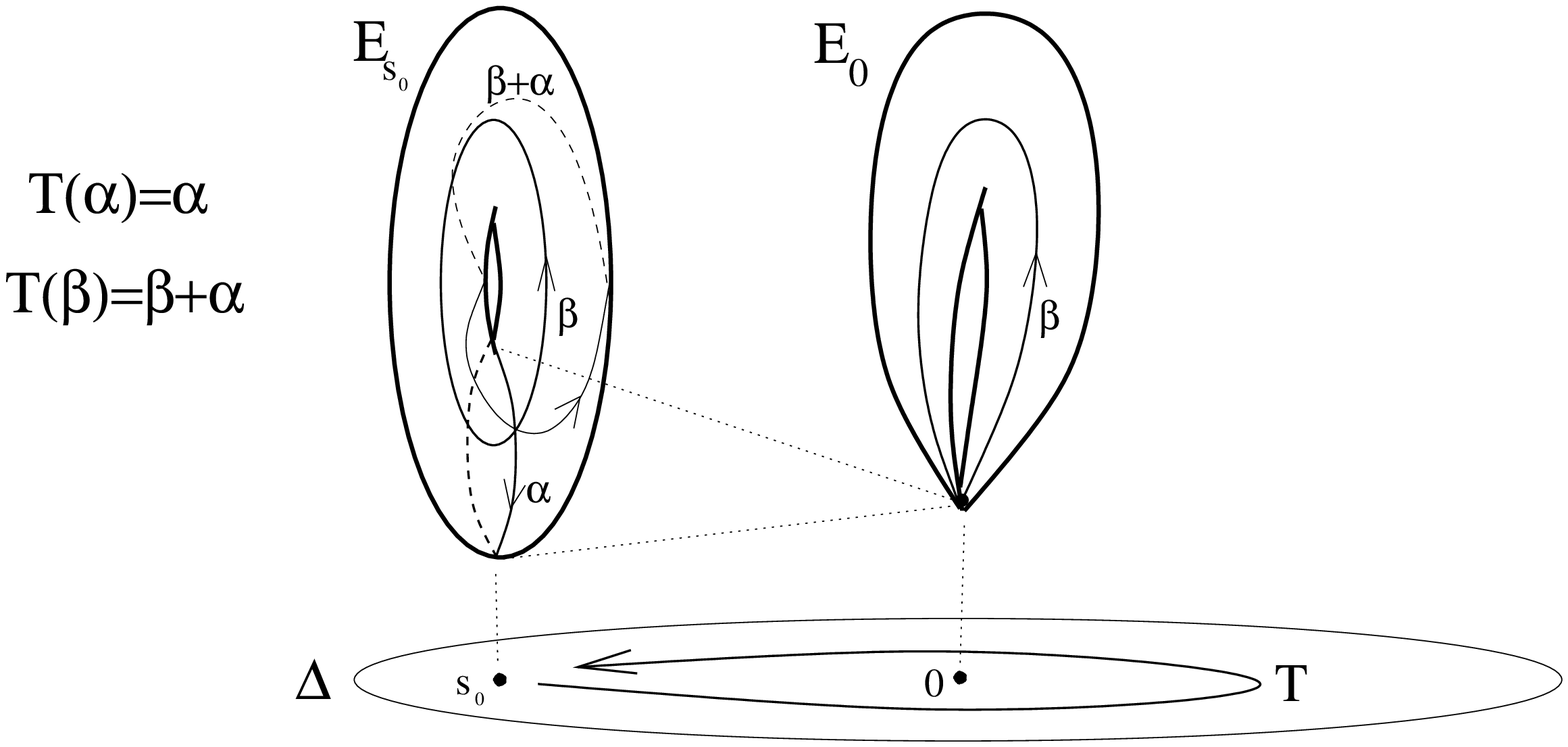}

(For the reader new to such pictures, the two crossing segments in
the previous {}``local real'' picture become the two touching {}``thimbles'',
i.e. a small neighborhood of the singularity in $E_{0}$, in this
diagram.)

Now, in our setting, the bundle of Jacobians $\J:=\bigcup_{s\in S^{*}}J^{1}(X_{s})$
is a complex (algebraic) manifold. It admits a partial compactification
to a fiber space of complex abelian Lie groups, by defining $J^{1}(X_{s_{i}}):=\frac{H^{0}(\omega_{X_{s_{i}}})}{\text{im}\{H^{1}(X_{s_{i}},\ZZ)\}}$
($\omega_{x_{s}}=$ dualizing sheaf) and $\J_{e}:=\cup_{s\in S}J^{1}(X_{s})$.
(How this is topologized will be discussed in a more general context
in $\S5$.) The same notation will denote their sheaves of sections,
\begin{equation}\label{eqn ses1} 
0 \to \HH_{\ZZ} \to \F^{\vee} \to \J \to 0 \mspace{50mu} (\text{on}\, \, S^*) \\
\end{equation}
\begin{equation}
\label{eqn ses2}0 \to \HH_{\ZZ,e} \to (\F_e)^{\vee} \to \J_e  \to 0 \mspace{50mu} (\text{on}\, \, S) \\
\end{equation} with $\F:=\pi_{*}\omega_{\X/S}$, $\F_{e}:=\bar{\pi}_{*}\omega_{\X/S}$,
$\HH_{\ZZ}=R^{1}\pi_{*}\ZZ$, $\HH_{\ZZ,e}=R^{1}\bar{\pi}_{*}\ZZ$.
\begin{defn}
A \emph{normal function} (NF) is a holomorphic section (over $S^{*}$)
of $\J$. An \emph{extended (or Poincar\'e) normal function} (ENF)
is a holomorphic section (over $S$) of $\J_{e}$. A NF is \emph{extendable}
if it lies in $\text{im}\{H^{0}(S,\J_{e})\to H^{0}(S^{*},\J)\}$.
\end{defn}
Next consider the long-exact cohomology sequence (sections over $S^{*}$)
\begin{equation}\label{eqn les1} 
0 \to H^0(\HH_{\ZZ}) \to H^0(\F^{\vee}) \to H^0(\J) \to H^1(\HH_{\ZZ}) \to H^1(\F^{\vee}) ; \\
\end{equation} the \emph{topological invariant} of a normal function $\nu\in H^{0}(\J)$
is its image $[\nu]\in H^{1}(S^{*},\HH_{\ZZ})$. It is easy to see
that the restriction of $[\nu]$ to $H^{1}(\Delta_{i}^{*},\HH_{\ZZ})$
($\Delta_{i}$ a punctured disk about $s_{i}$) computes the local
monodromy $(T_{s_{i}}-I)\tilde{\nu}$ (where $\tilde{\nu}$ is a multivalued
local lift of $\nu$ to $\F^{\vee}$), modulo the monodromy of topological
cycles. We say that $\nu$ is locally liftable if all these restrictions
vanish, i.e. if $(T_{s_{i}}-I)\tilde{\nu}\in\text{im}\{(T_{s_{i}}-I)\HH_{\ZZ,s_{0}}\}$.
Together with the assumption that as a (multivalued, singular) {}``section''
of $\F_{e}^{\vee}$, $\tilde{\nu}_{e}$ has at worst logarithmic divergence
at $s_{i}$ (the {}``logarithmic growth'' in the title), this is
equivalent to extendability.

\subsection{Normal functions of geometric origin}

Let $\sZ\in Z^{1}(\X)_{prim}$ be a divisor properly intersecting
fibres of $\bar{\pi}$ and avoiding its singularities, and which is
\emph{primitive} in the sense that each $Z_{s}:=\sZ\cdot X_{s}$ ($s\in S^{*}$)
is of degree 0. (In fact, the intersection conditions can be done
away with, by moving the divisor in a rational equivalence.) Then
$s\mapsto AJ(Z_{s})$ defines a section $\nu_{\sZ}$ of $\J$, and
it can be shown that a multiple $N\nu_{\sZ}=\nu_{N\sZ}$ of $\nu_{\sZ}$
is always extendable. One says that $\nu_{\sZ}$ itself is \emph{admissible}.

Now assume $\bar{\pi}$ has a section $\sigma:S\to\X$ (also avoiding
singularities) and consider the analogue of (\ref{eqn les1}) for
$\J_{e}$\[
0\to\frac{H^{0}(\F_{e}^{\vee})}{H^{0}(\HH_{\ZZ,e})}\to H^{0}(\J_{e})\to\ker\left\{ H^{1}(\HH_{\ZZ,e})\to H^{1}(\F_{e}^{\vee})\right\} \to0.\]
With a bit of work, this becomes \begin{equation}\label{eqn ses3}
\xymatrix{ 0 \to J^1(\X/S)_{fix} \ar [r] & \text{ENF} \ar [r]^{\left[ \cdot \right] \mspace{50mu}} & \frac{Hg^1(\X)_{prim}}{\ZZ \left\langle [X_{s_0}] \right\rangle } \to 0} ,  \\
\end{equation} where the Jacobian of the fixed part $J^{1}(\X/S)_{fix}\hookrightarrow J^{1}(X_{s})$
($\forall s\in S$) gives a constant subbundle of $\J_{e}$ and the
primitive Hodge classes $Hg^{1}(\X)_{prim}$ are the $Q$-orthogonal
complement of a general fibre $X_{s_{0}}$ of $\bar{\pi}$ in $Hg^{1}(\X):=H^{2}(\X,\ZZ)\cap H^{1,1}(\X,\CC)$.
\begin{prop}
Let $\nu$ be an ENF.

(i) If $[\nu]=0$ then $\nu$ is a constant section of $\J_{fix}:=\bigcup_{s\in S}J^{1}(\X/S)_{fix}\subset\J_{e}$;

(ii) If $(\nu=)\nu_{\sZ}$ is of geometric origin, then $[\nu_{\sZ}]=\overline{[\sZ]}$
($[\sZ]=$ fundamental class);

(iii) \emph{{[}Poincar\'e Existence Theorem{]}} Every ENF is of geometric
origin.
\end{prop}
We note that (i) follows from considering sections $\{\omega_{1},\ldots,\omega_{g}\}(s)$
of $\F_{e}^{\vee}$ whose restrictions to general $X_{s}$ are linearly
independent (such do exist), evaluating a lift $\tilde{\nu}\in H^{0}(\F_{e}^{\vee})$
against them, and applying Liouville's Theorem. The resulting constancy
of the abelian integrals, by a result in Hodge Theory (cf. end of
$\S1.6$), implies the membership of $\nu(s)\in\J_{fix}$. To see
(iii), apply {}``Jacobi inversion with parameters'' and $q_{i}(s)=\sigma(s)$
($\forall i$) over $S^{*}$ (really, over the generic point of $S$),
and then take Zariski closure.%
\footnote{Here the $q_{i}(s)$ are as in Theorem \ref{thm Abel}(ii) (but varying
with respect to a parameter). If at a generic point $\nu(\eta)$ is
a special divisor then additional argument is needed.%
} Finally, when $\nu$ is geometric, the monodromies of a lift $\tilde{\nu}$
(to $\F_{e}^{\vee}$) around each loop in $S$ (which determine $[\nu]$)
are just the corresponding monodromies of a bounding $1$-chain $\Gamma_{s}$
($\partial\Gamma_{s}=Z_{s}$) 

$\mspace{100mu}$\includegraphics[scale=0.4]{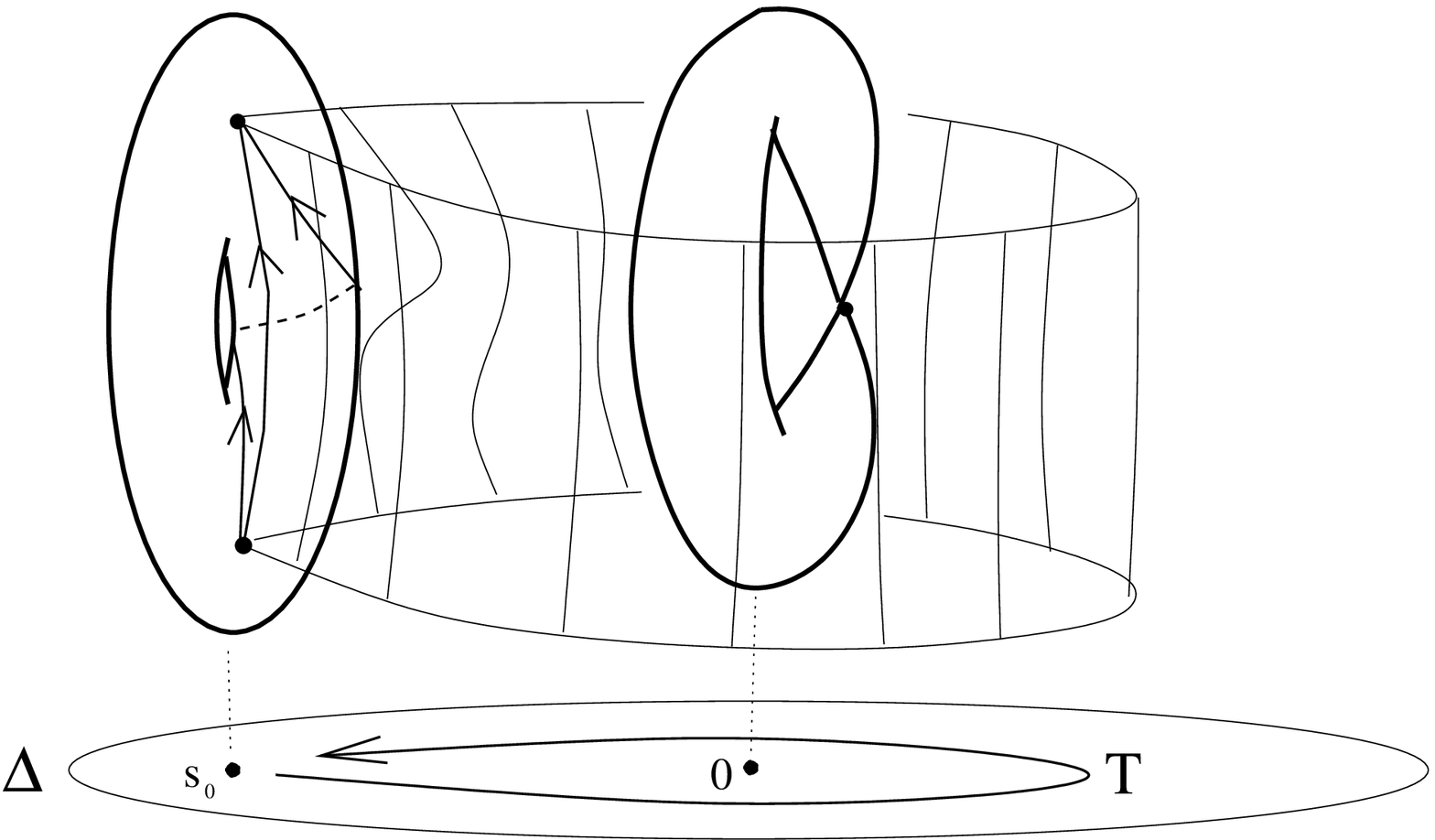}

which identify with the Leray $(1,1)$ component of $[\sZ]$ in $H^{2}(\X)$;
this gives the gist of (ii).

A normal function is said to be \emph{motivated over $K$} ($K\subset\CC$
a subfield) if it is of geometric origin as above, and if the coefficients
of the defining equations of $\sZ$, $\X$, $\bar{\pi}$, and $S$
belong to $K$.

\subsection{Lefschetz (1,1) Theorem}

Now take $X\subset\PP^{N}$ to be a smooth projective surface of degree
$d$, and $\{X_{s}:=X\cdot H_{s}\}_{s\in\PP^{1}}$ a \emph{Lefschetz
pencil} of hyperplane sections: the singular fibres have exactly one
(nodal) singularity. Let $\beta:\X\twoheadrightarrow X$ denote the
blow-up at the base locus $B:=\bigcap_{s\in\PP^{1}}X_{s}$ of the
pencil, and $\bar{\pi}:\,\X\to\PP^{1}=:S$ the resulting fibration.
We are now in the situation considered above, with $\sigma(S)$ replaced
by $d$ sections $E_{1}\amalg\cdots\amalg E_{d}=\beta^{-1}(B)$, and
fibres of genus $g={d-1 \choose 2}$; and with the added bonus that
there is no torsion in any $H^{1}(\Delta_{i}^{*},\HH_{\ZZ})$, so
that admissible $\implies$ extendable. Hence, given $Z\in Z^{1}(X)_{prim}$
($\deg(Z\cdot X_{s_{0}})=0$): $\beta^{*}Z$ is primitive, $v_{Z}:=\nu_{\beta^{*}Z}$
is an ENF, and $[v_{Z}]=\beta^{*}[Z]$ under $\beta^{*}:Hg^{1}(X)_{prim}\hookrightarrow\frac{Hg^{1}(\X)_{prim}}{\ZZ\left\langle \left[X_{s_{0}}\right]\right\rangle }$. 

If on the other hand we start with a Hodge class $\xi\in Hg^{1}(X)_{prim}$,
$\beta^{*}\xi$ is (by (\ref{eqn ses3}) + Poincar\'e existence)
the class of a geometric ENF $\nu_{\sZ}$; and $[\sZ]\equiv[\nu_{\sZ}]\equiv\beta^{*}\xi$
mod $\ZZ\left\langle \left[X_{s_{0}}\right]\right\rangle $ $\implies$
$\xi\equiv\beta_{*}\beta^{*}\xi\equiv[\beta_{*}\sZ=:Z]$ in $\frac{Hg^{1}(X)}{\ZZ\left\langle \left[X_{s_{0}}\right]\right\rangle }$
$\implies$ $\xi=[Z']$ for some $Z'\in Z^{1}(X)_{(prim)}$. This
is the gist of Lefschetz's original proof \cite{L} of
\begin{thm}
\label{thm lef11}Let $X$ be a (smooth projective algebraic) surface.
The fundamental class map $CH^{1}(X)\overset{[\cdot]}{\to}Hg^{1}(X)$
is (integrally) surjective.
\end{thm}
This continues to hold in higher dimension, as can be seen from an
inductive treatment with ENF's or (more easily) from the {}``modern''
treatment of Theorem \ref{thm lef11} using the exponential exact
sheaf sequence \[
0\to\ZZ_{X}\longrightarrow\mathcal{O}_{X}\overset{\tiny e^{2\pi i(\cdot)}}{\longrightarrow}\mathcal{O}_{X}^{*}\to0.\]
One simply puts the induced long-exact sequence in the form\[
0\to\frac{H^{1}(X,\mathcal{O})}{H^{1}(X,\ZZ)}\to H^{1}(X,\mathcal{O}^{*})\to\ker\left\{ H^{2}(X,\ZZ)\to H^{2}(X,\mathcal{O})\right\} \to0,\]
and interprets it as \begin{equation}
\xymatrix{0 \ar [r] & J^1(X) \ar [r] & \left\{ {\begin{matrix} {\text{holomorphic}} \\ {\text{line bundles}} \end{matrix} } \right\} \ar @{-->} [d] \ar [r] & Hg^1(X) \ar [r] & 0 \\ & & CH^1(X) \ar [ur] } \\
\end{equation} where the dotted arrow takes the divisor of a meromorphic section
of a given bundle. Existence of the section is a standard but nontrivial
result.

We note that for $\X\to\PP^{1}$ a Lefschetz pencil of $X$, in (\ref{eqn ses3})
$J^{1}(\X/\PP^{1})_{fix}=J^{1}(X):=\frac{H^{1}(X,\CC)}{F^{1}H^{1}(X,\CC)+H^{1}(X,\ZZ)}$,
which is zero if $X$ is a complete intersection; in that case $ENF$
is finitely generated and $Hg^{1}(\X)_{prim}\overset{\beta^{*}}{\hookrightarrow}ENF$.
\begin{example}
For $X$ a cubic surface $\subset\PP^{3}$, divisors with support
on the $27$ lines already surject onto $Hg^{1}(X)=H^{2}(X,\ZZ)\cong\ZZ^{7}$.
Differences of these lines generate all primitive classes, hence all
of $\text{im}(\beta^{*})$($\cong\ZZ^{6}$) in ENF($\cong\ZZ^{8}$).
Note that $\J_{e}$ is essentially an elliptic surface and ENF comprises
the (holomorphic) sections passing through the $\CC^{*}$'s over points
of $\Sigma$. There are no torsion sections.
\end{example}

\subsection{Griffiths's AJ map}

A $\ZZ$-Hodge structure (HS) of weight $m$ comprises a finitely
generated abelian group $H_{\ZZ}$ together with a descending filtration
$F^{\bullet}$ on $H_{\CC}:=H_{\ZZ}\otimes_{\ZZ}\CC$ satisfying $F^{p}H_{\CC}\oplus\overline{F^{m-p+1}H_{\CC}}=H_{\CC}$,
the \emph{Hodge filtration}; we denote the lot by $H$. Examples include
the $m^{\text{th}}$ (singular/Betti + de Rham) cohomology groups
of smooth projective varieties $/\CC$, with $F^{p}H_{dR}^{m}(X,\CC)$
being that part of the de Rham cohomology represented by $C^{\infty}$
forms on $X^{an}$ with \emph{at least} $p$ holomorphic differentials
wedged together in each monomial term. (These are forms of \emph{Hodge
type} $(p,m-p)+(p+1,m-p-1)+\cdots$; note that $H_{\CC}^{p,m-p}:=F^{p}H_{\CC}\cap\overline{F^{m-p}H_{\CC}}$.)
To accommodate $H^{m}$ of non-smooth or incomplete varieties, the
notion of a ($\ZZ$-)mixed Hodge structure (MHS) $V$ is required:
in addition to $F^{\bullet}$ on $V_{\CC}$, introduce a decreasing
\emph{weight filtration} $W_{\bullet}$ on $V_{\QQ}$ such that the
$\left(Gr_{i}^{W}V_{\QQ},\left(Gr_{i}^{W}(V_{\CC},F^{\bullet})\right)\right)$
are weight $i$ $\QQ$-HS. Mixed Hodge structures have Hodge and Jacobian
groups $Hg^{p}(V):=\ker\{V_{\ZZ}\oplus F^{p}W_{2p}V_{\CC}\to V_{\CC}\}$
(for for $V_{\ZZ}$ torsion-free becomes $V_{\ZZ}\cap F^{p}W_{2p}V_{\CC}$)
and $J^{p}(V):=\frac{W_{2p}V_{\CC}}{F^{p}W_{2p}V_{\CC}+W_{2p}V_{\QQ}\cap V_{\ZZ}}$,
with special cases $Hg^{m}(X):=Hg^{m}(H^{2m}X))$ and $J^{m}(X):=J^{m}(H^{2m-1}(X))$.
Jacobians of HS yield complex tori, and subtori correspond bijectively
to sub-HS. 

A \emph{polarization} of a HS $H$ is a morphism $Q$ of HS (defined
over $\ZZ$; complexification respects $F^{\bullet}$) from $H\times H$
to the trivial HS $\ZZ(-m)$ of weight $2m$ (and type $(m,m)$),
such that viewed as a pairing $Q$ is nondegenerate and satisfies
a positivity constraint generalizing that in $\S1.1$ (the \emph{second
Hodge-Riemann bilinear relation}). A consequence of this definition
is that under $Q$, $F^{p}$ is the annihilator of $F^{m-p+1}$ (the
\emph{first Hodge-Riemann bilinear relation} in abstract form). If
$X$ is a smooth projective variety of dimension $d$, $[\Omega]$
the class of a hyperplane section, write (for $k\leq d$, say) $H^{m}(X,\QQ)_{prim}:=\ker\{H^{m}(X,\QQ)\overset{\cup\Omega^{d-k+1}}{\longrightarrow}H^{2d-m+2}(X,\QQ)\}$.
This Hodge structure is then polarized by $Q(\xi,\eta):=(-1)^{{m \choose 2}}\int_{X}\xi\wedge\eta\wedge\Omega^{d-k}$,
$[\Omega]$ the class of a hyperplane section (obviously since this
is a $\QQ$-HS, the polarization is only defined $/\QQ$).

Let $X$ be a smooth projective $(2m-1)$-fold; we shall consider
some equivalence relations on algebraic cycles of codimension $m$
on $X$. Writing $Z^{m}(X)$ for the free abelian group on irreducible
(complex-)codimension $p$ subvarieties of $X$, two cycles $Z_{1},Z_{2}\in Z^{m}(X)$
are homologically equivalent if their difference bounds a $C^{\infty}$
chain $\Gamma\in C_{2m-1}^{top}(X^{an};\ZZ)$ (of real dimension $2m-1$).
Algebraic equivalence is generated by (the projection to $X$ of)
differences of the form $W\cdot(X\times\{p_{1}\})-W\cdot(X\times\{p_{2}\})$
where $C$ is an algebraic curve, $W\in Z^{m}(X\times C)$, and $p_{1},p_{2}\in C(\CC)$
(or $C(K)$ if we are working over a subfield $K\subset\CC$). Rational
equivalence is obtained by taking $C$ to be rational (i.e. $C\cong\PP^{1}$),
and for $m=1$ is generated by divisors of meromorphic functions.
We write $Z^{m}(X)_{rat}$ for cycles $\equiv_{rat}0$, etc; note
that $CH^{m}(X):=\frac{Z^{m}(X)}{Z^{m}(X)_{rat}}\supset CH^{m}(X)_{hom}:=\frac{Z^{m}(X)_{hom}}{Z^{m}(X)_{rat}}\supset CH^{m}(X)_{alg}:=\frac{Z^{m}(X)_{alg}}{Z^{m}(X)_{rat}}$
are proper inclusions in general. 

Now let $W\subset X\times C$ be an irreducible subvariety of codimension
$m$, with $\pi_{X}$,$\pi_{C}$ the projections from a desingularization
of $W$ to $X$ resp. $C$. If we put $Z_{i}:=\pi_{X_{*}}\pi_{C}^{*}\{p_{i}\}$,
then $Z_{1}\equiv_{alg}Z_{2}$ $\implies$ $Z_{1}\equiv_{hom}Z_{2}$,
which can be seen explicitly by setting $\Gamma:=\pi_{X_{*}}\pi_{C}^{*}(\overrightarrow{q.p})$
(so that $Z_{1}-Z_{2}=\partial\Gamma$). 

$\mspace{80mu}$\includegraphics[scale=0.5]{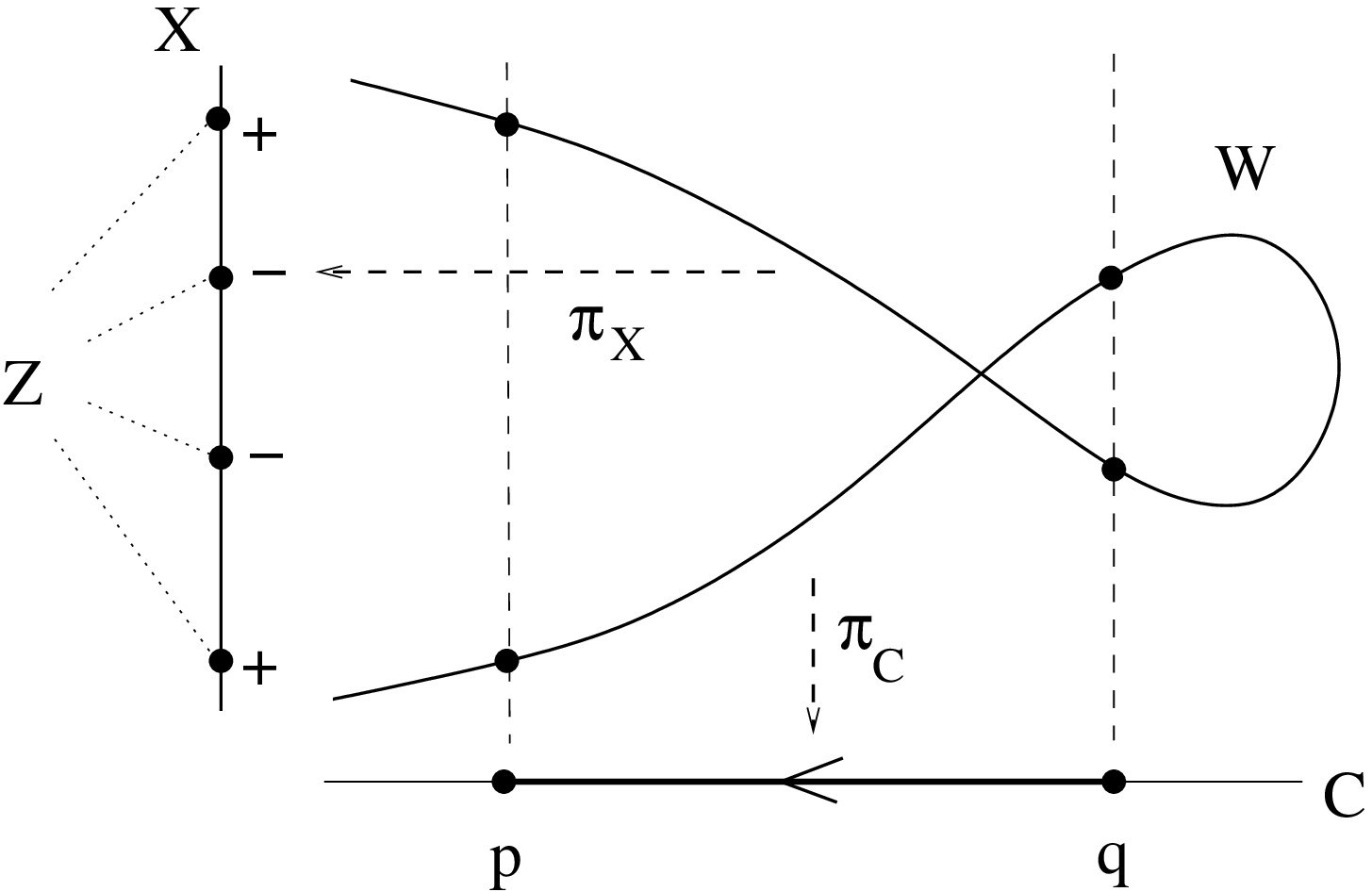}

Let $\omega$ be a $d$-closed form of Hodge type $(j,2m-j-1)$ on
$X$, for $j$ at least $m$. Consider $\int_{\Gamma}\omega=\int_{p}^{q}\kappa$,
where $\kappa:=\pi_{C_{*}}\pi_{X}^{*}\omega$ is a $d$-closed $1$-current
of type $(j-m+1,m-j)$ as integration along the $(m-1)$-dimensional
fibres of $\pi_{C}$ eats up $(m-1,m-1)$. So $\kappa=0$ unless $j=m$,
and by a standard regularity theorem in that case $\kappa$ is holomorphic.
In particular, if $C$ is rational, we have $\int_{\Gamma}\omega=0$.
This is essentially the reasoning behind the following result:
\begin{prop}
The Abel-Jacobi map \begin{equation}\label{eqn ajdef1} 
\xymatrix{CH^m(X)_{hom} \ar [r]^{AJ \mspace{70mu}} & \frac{\left( F^mH^{2m-1}(X,\CC) \right)^{\vee} }{\int_{H_{2m-1}(X,\ZZ)}(\cdot )} \cong J^m(X) } \\
\end{equation} induced by $Z=\partial\Gamma\mapsto\int_{\Gamma}(\cdot)$, is well-defined
and restricts to \begin{equation}\label{eqn ajdef2} 
\xymatrix{CH^m(X)_{alg} \ar [r]^{AJ_{alg} \mspace{180mu}} & \frac{F^mH^{2m-1}_{hdg}(X,\CC)}{\int_{H_{2m-1}(X,\ZZ)}(\cdot)} \cong J^m(H^{2m-1}_{hdg}(X)) =: J_h^m(X)
} \\
\end{equation} where $H_{hdg}^{2m-1}(X)$ is the largest sub-HS of $H^{2m-1}(X)$
contained (after $\otimes\CC$) in $H^{m-1,m}(X,\CC)\oplus H^{m,m-1}(X,\CC)$.
While $J^{m}(X)$ is in general only a complex torus, $J_{h}^{m}(X)$
is an abelian variety and defined (along with the point $AJ_{alg}(Z)$)
over the field of definition of $X$.\end{prop}
\begin{rem}
\label{rem AJ}(i) To see that $J_{h}^{m}(X)$ is an abelian variety,
one uses the Kodaira embedding theorem: by the Hodge-Riemann bilinear
relations, the polarization of $H^{2m-1}(X)$ induces a K?ler metric
$h(u,v)=-iQ(u,\bar{v})$ on $J_{h}^{m}(X)$ with rational K?ler class.

(ii) The mapping (\ref{eqn ajdef1}) is neither surjective nor injective
in general, and (\ref{eqn ajdef2}) is not injective in general; however,
(\ref{eqn ajdef2}) is conjectured to be surjective, and regardless
of this $J_{alg}^{m}(X):=\text{im}(AJ_{alg})\subseteq J_{h}^{m}(X)$
is in fact a sub-abelian-variety.

(iii) A point in $J^{m}(X)$ is naturally the invariant of an extension
of MHS \[
0\to(H=)H^{2m-1}(X,\ZZ(m))\to E\to\ZZ(0)\to0\]
(where the {}``twist'' $\ZZ(m)$ reduces weight by $2m$, to $(-1)$).
The invariant is evaluated by taking two lifts $\nu_{F}\in F^{0}W_{0}E_{\CC}$,
$\nu_{\ZZ}\in W_{0}E_{\ZZ}$ of $1\in\ZZ(0)$, so that $\nu_{F}-\nu_{\ZZ}\in W_{0}H_{\CC}$
is well-defined modulo the span of $F^{0}W_{0}H_{\CC}$ and $W_{0}H_{\ZZ}$
hence is in $J^{0}(H)\cong J^{m}(X)$. The resulting isomorphism $J^{m}(X)\cong Ext_{\text{MHS}}^{1}(\ZZ(0),H^{2m-1}(X,\ZZ(m)))$
is part of an extension-class approach to $AJ$ maps (and their generalizations)
due to Carlson \cite{Ca}.

(iv) The Abel-Jacobi map appears in \cite{Gr3}.
\end{rem}

\subsection{Horizontality}

Generalizing the setting of $\S1.2$, let $\X$ be a smooth projective
$2m$-fold fibred over a curve $S$ with singular fibres $\{X_{s_{i}}\}$
each of either

(i) \emph{NCD}(=normal crossing divisor) \emph{type}: locally $(x_{1},\ldots,x_{2m})\overset{\pi}{\longmapsto}\prod_{j=1}^{k}x_{j}$;
or

(ii) \emph{ODP}(=ordinary double point) \emph{type}: locally $(\underline{x})\longmapsto\sum_{j=1}^{2m}x_{j}^{2}$.

An immediate consequence is that all $T_{s_{i}}\in Aut\left(H^{2m-1}(X_{s_{0}},\ZZ)\right)$
are \emph{unipotent}: $(T_{s_{i}}-I)^{n}=0$ for $n\geq2m$ in case
(i) or $n\geq2$ in case (ii). (If all fibers are of NCD type, then
we say the family $\{X_{s}\}$ of $(2m-1)$-folds is \emph{semistable}.)

The Jacobian bundle of interest is $\J:=\bigcup_{s\in S^{*}}J^{m}(X_{s})\,(\supset\J_{alg})$.
Writing \[
\left\{ \F^{(m)}:=\RR^{2m-1}\pi_{*}\Omega_{\X^{*}/S^{*}}^{\bullet\geq m}\right\} \subset\left\{ \H:=\RR^{2m-1}\pi_{*}\Omega_{\X^{*}/S^{*}}^{\bullet}\right\} \supset\left\{ \HH_{\ZZ}:=R^{2m-1}\pi_{*}\ZZ_{\X^{*}}\right\} ,\]
and noting $\F^{\vee}\cong\frac{\H}{\F}$ via $Q:\,\H^{2m-1}\times\H^{2m-1}\to\mathcal{O}_{S^{*}}$,
the sequences (\ref{eqn ses1}) and (\ref{eqn les1}), as well as
the definitions of NF and topological invariant $[\cdot]$, all carry
over. A normal function of geometric origin, likewise, comes from
$\sZ\in Z^{m}(\X)_{prim}$ with $Z_{s_{0}}:=\sZ\cdot X_{s_{0}}\equiv_{hom}0$
(on $X_{s_{0}}$), but now has an additional feature known as \emph{horizontality},
which we now explain.

Working locally over an analytic ball $(s_{0}\in)\, U\subset S^{*}$,
let $\tilde{\omega}\in$$\Gamma(\X_{U},F^{m+1}\Omega_{\X^{\infty}}^{2m-1})$
be a {}``lift'' of $\omega(s)\in\Gamma(U,\F^{m+1})$, and $\Gamma_{s}\in C_{2m-1}^{top}(X_{s};\ZZ)$
be a continuous family of chains with $\partial\Gamma_{s}=Z_{s}$.
Let $P^{\e}$ be a path from $s_{0}$ to $s_{0}+\e$; then $\hat{\Gamma}^{\e}:=\bigcup_{s\in P^{\e}}\Gamma_{s}$
has boundary $\Gamma_{s_{0}+\e}-\Gamma_{s_{0}}+\bigcup_{s\in P^{\e}}Z_{s}$,
and \begin{equation}\label{eqn horiz} 
\begin{matrix}
{\left( \frac{\partial}{\partial s} \int_{\Gamma_s} \omega(s) \right)_{s=s_0} = \lim_{\e \to 0} \frac{1}{\e} \int_{\Gamma_{s_0+\e}-\Gamma_{s_0}}\tilde{\omega} = }
\\
{\lim_{\e \to 0}\frac{1}{\e} \left(  \int_{\partial \hat{\Gamma}^{\e}} \tilde{\omega} - \int_{s_0}^{s_0+\e} \int_{Z_s} \omega(s) \right) =}
\\
{ \int_{\Gamma_{s_0}} \left\langle \widetilde{d/dt}, d\tilde{\omega} \right\rangle - \int_{Z_{s_0}} \omega(s_0) } 
\end{matrix} \\
\end{equation} where $\pi_{*}\widetilde{d/dt}=d/dt$ (with $\widetilde{d/dt}$ tangent
to $\hat{\Gamma}^{\e}$, $\hat{Z}^{\e}$).

The Gauss-Manin connection $\nabla:\,\H\to\H\otimes\Omega_{S^{*}}^{1}$
differentiates the periods of cohomology classes (against topological
cycles) in families, satisfies Griffiths transversality $\nabla(\F^{m})\subset\F^{m-1}\otimes\Omega_{S^{*}}^{1}$,
and is computed by $\nabla\omega=[\left\langle \widetilde{d/dt},d\tilde{\omega}\right\rangle ]\otimes dt$.
Moreover, the pullback of any form of type $F^{m}$ to $Z_{s_{0}}$
(which is of dimension $m-1$) is zero, so that $\int_{Z_{s_{0}}}\omega(s_{0})=0$
and $\int_{\Gamma_{s_{0}}}\nabla\omega$ is well-defined. If $\tilde{\Gamma}\in\Gamma(U,\H)$
is any lift of $AJ(\Gamma_{s})\in\Gamma(U,\J)$, we therefore have\[
Q\left(\nabla_{d/dt}\tilde{\Gamma},\omega\right)=\frac{d}{ds}Q(\tilde{\Gamma},\omega)-Q(\tilde{\Gamma},\nabla\omega)\]
\[
=\frac{d}{ds}\int_{\Gamma_{s}}\omega-\int_{\Gamma_{s}}\nabla\omega\]
which is zero by (\ref{eqn horiz}) and the remarks just made. We
have shown that $\nabla_{d/dt}\tilde{\Gamma}$ kills $\F^{m+1}$,
and so $\nabla_{d/dt}\tilde{\Gamma}$ is a local section of $\F^{m-1}$.
\begin{defn}
A NF $\nu\in H^{0}(S^{*},\J)$ is \emph{horizontal} if for any local
lift $\tilde{\nu}\in\Gamma(U,\H)$, $\nabla\tilde{\nu}\in\Gamma(U,\F^{m-1}\otimes\Omega_{U}^{1})$.
Equivalently, if we set $\H_{hor}:=\ker\left(\H\overset{\nabla}{\to}\frac{\H}{\F^{m-1}}\otimes\Omega_{S^{*}}^{1}\right)\supset\F^{m}=:\F$,
$(\F^{\vee})_{hor}:=\frac{\H_{hor}}{\F}$, and $\J_{hor}:=\frac{(\F^{\vee})_{hor}}{\HH_{\ZZ}}$,
then $\text{NF}_{hor}:=H^{0}(S,\J_{hor})$.
\end{defn}
Much as an $AJ$ image was encoded in a MHS in Remark \ref{rem AJ}(ii),
we may encode horizontal normal functions in terms of variations of
MHS. A VMHS $\V/S^{*}$ consists of a $\ZZ$-local system $\VV$ with
an increasing filtration of $\VV_{\QQ}:=\VV_{\ZZ}\otimes_{\ZZ}\QQ$
by sub- local systems $W_{i}\VV_{\QQ}$, a decreasing filtration of
$\V_{(\mathcal{O})}:=\VV_{\QQ}\otimes_{\QQ}\mathcal{O}_{S^{*}}$ by
holomorphic vector bundles $\F^{j}(=\F^{j}\V)$, and a connection
$\nabla:\,\V\to\V\otimes\Omega_{S^{*}}^{1}$ such that (i) $\nabla(\VV)=0$,
(ii) the fibres $(\VV_{s},W_{\bullet},V_{s},F_{s}^{\bullet})$ yield
$\ZZ$-MHS, and (iii) {[}transversality{]} $\nabla(\F^{j})\subset\F^{j-1}\otimes\Omega_{S^{*}}^{1}$.
(Of course, a VHS is just a VMHS with one nontrivial $Gr_{i}^{W}\VV_{\QQ}$,
and $((\HH_{\ZZ},\H,\F^{\bullet}),\nabla)$ in the geometric setting
above gives one.) A horizontal normal function corresponds to an extension
\begin{equation}\label{eqn AJasExt} 
0 \to \overset{\Tiny \begin{matrix} {\text{wt. -1}} \\ {\text{VHS}} \end{matrix} \normalsize }{\overbrace{\H(m)}} \to \E \to \ZZ(0)_{S^*} \to 0 \\
\end{equation} {}``varying'' the setup of Remark \ref{rem AJ}(iii), with the
transversality of the lift of $\nu_{F}(s)$ (together with flatness
of $\nu_{\ZZ}(s)$) reflecting horizontality.
\begin{rem}
Allowing the left-hand term of (\ref{eqn AJasExt}) to have weight
less than $-1$ yields {}``higher'' normal functions related to
families of \emph{generalized} ({}``higher'') algebraic cycles.
These have been studied in \cite{DM1}, \cite{DM2}, and \cite{DK},
and will be considered in later sections.
\end{rem}
An important result on VHS over a smooth quasi-projective base is
that the global sections $H^{0}(S^{*},\VV)$ (resp. $H^{0}(S^{*},\VV_{\RR})$,
$H^{0}(S^{*},\VV_{\CC})$) span the $\QQ$-local system (resp. its
$\otimes\RR$, $\otimes\CC$) of a (necessarily constant) sub-VMHS
$\subset\V$, called the \emph{fixed part} $\V_{fix}$ (with constant
Jacobian bundle $\J_{fix}$).

\subsection{Infinitesimal invariant}

Given $\nu\in\text{NF}_{hor}$, the {}``$\nabla\tilde{\nu}$'' for
various local liftings patch together after going modulo $\nabla\F^{m}\subset\F^{m-1}\otimes\Omega_{S^{*}}^{1}$.
If $\nabla\tilde{\nu}=\nabla f$ for $f\in\Gamma(U,\F^{m})$, then
the alternate lift $\tilde{\nu}-f$ is flat, i.e. equals $\sum_{i}c_{i}\gamma_{i}$
where $\{\gamma_{i}\}\subset\Gamma(U,\VV_{\ZZ})$ is a basis and the
$c_{i}$ are complex constants. Since the composition ($s\in S^{*}$)
$H^{2m-1}(X_{s},\RR)\hookrightarrow H^{2m-1}(X_{s},\CC)\twoheadrightarrow\frac{H^{2m-1}(X_{s},\CC)}{F^{m}}$
is an isomorphism, we may take the $c_{i}\in\RR$, and then they are
unique in $\RR/\ZZ$. This implies that $[\nu]$ lies in the torsion
group $\ker\left(H^{1}(\HH_{\ZZ})\to H^{1}(\HH_{\RR})\right)$, so
that a multiple $N\nu$ lifts to $H^{0}(S^{*},\HH_{\RR})\subset\H_{fix}$.
This motivates the definition of an infinitesimal invariant \begin{equation}
\xymatrix{\delta \nu \in \HH^1 \left( S^*, \F^m \overset{\nabla}{\to} \F^{m-1}\otimes \Omega^1_{S^*} \right)  \ar @{=} [r]^{\mspace{90mu} \text{if}\, S^*}_{\mspace{90mu} \text{affine}} & H^0 \left( S,\frac{\F^{m-1}\otimes \Omega^1}{\F^m} \right)   } \\
\end{equation} as the image of $\nu\in H^{0}\left(S^{*},\frac{\H_{hor}}{\F}\right)$
under the connecting homomorphism induced by \begin{equation}
\tiny \xymatrix{
0 \to \text{Cone}\left( \F^m \overset{\nabla}{\to}\F^{m-1}\otimes \Omega^1 \right) [-1] \to \text{Cone}\left( \H \overset{\nabla}{\to} \H \otimes \Omega^1 \right) [-1] \to \frac{\H_{hor}}{\F} \to 0.} \\
\normalsize
\end{equation}
\begin{prop}
If $\delta\nu=0$, then up to torsion, $[\nu]=0$ and $\nu$ is a
(constant) section of $\J_{fix}$.
\end{prop}
An interesting application to the differential equations satisfied
by normal functions is essentially due to Manin \cite{Ma}. For simplicity
let $S=\PP^{1}$, and suppose $\H$ is generated by $\omega\in H^{0}(S^{*},\F^{2m-1})$
as a $D$-module, with monic \emph{Picard-Fuchs operator} $F(\nabla_{\delta_{s}:=s\frac{d}{ds}})\in\CC(\PP^{1})^{*}[\nabla_{\delta_{s}}]$
killing $\omega$. Then its periods satisfy the homogeneous P-F equation
$F(\delta_{s})\int_{\gamma_{i}}\omega=0$, and one can look at the
multivalued holomorphic function $Q(\tilde{\nu},\omega)$ (where $Q$
is the polarization, and $\tilde{\nu}$ is a multivalued lift of $\nu$
to $\H_{hor}/\F$), which in the geometric case is just $\int_{\Gamma_{s}}\omega(s)$.
The resulting equation \begin{equation}
(2 \pi i)^m F(\delta_s)Q(\tilde{\nu},\omega) =: G(s) \\
\end{equation} is called the \emph{inhomogeneous Picard-Fuchs equation} of $\nu$.
\begin{prop}
\label{prop DEnf}(i) \cite{DM1} $G\in\CC(\PP^{1})^{*}$ is a rational
function holomorphic on $S^{*}$; in the $K$-motivated setting (taking
also $\omega\in H^{0}(\PP^{1},\bar{\pi}_{*}\omega_{\X/\PP^{1}})$,
and hence $F$, over $K$), $G\in K(\PP^{1})^{*}$.

(ii) \cite{Ma,Gr1} $G\equiv0$ $\iff$ $\delta\nu=0$.\end{prop}
\begin{example}
\label{ex MW}\cite{MW} The solutions to\[
(2\pi i)^{2}\left\{ \delta_{z}^{4}-5z\prod_{\ell=1}^{4}(5\delta_{z}+\ell)\right\} (\cdot)=-\frac{15}{4}\sqrt{z}\]
are the membrane integrals $\int_{\Gamma_{s}}\omega(s)$ for a family
of $1$-cycles on the mirror quintic family of Calabi-Yau $3$-folds.
(The family of cycles is actually only well-defined on the double-cover
of this family, as reflected by the $\sqrt{z}$.) What makes this
example particularly interesting is the {}``mirror dual'' interpretation
of the solutions as generating functions of open Gromov-Witten invariants
of a fixed Fermat quintic $3$-fold.
\end{example}
The horizontality relation $\nabla\tilde{\nu}\in\F^{m-1}\otimes\Omega^{1}$
is itself a differential equation, and the constraints it puts on
$\nu$ over higher-dimensional bases will be studied in $\S5.4-5$.

Returning to the setting described in $\S1.6$, there are \emph{canonical
extensions} $\H_{e},\F_{e}^{\bullet}$ of $\H,\F^{\bullet}$ across
the $s_{i}$ as holomorphic vector bundles resp. subbundles (reviewed
in $\S2$ below); e.g. if all fibres are of NCD type then $\F_{e}^{p}\cong\RR^{2m-1}\bar{\pi}_{*}\Omega_{\X/S}^{\bullet\geq p}(\log(\X\backslash\X^{*}))$.
Writing%
\footnote{Warning: while $\H_{e}$ has no jumps in rank, the stalk of $\HH_{\ZZ,e}$
at $s_{i}\in\Sigma$ is of strictly smaller rank than at $s\in S^{*}$.%
} $\HH_{\ZZ,e}:=R^{2m-1}\bar{\pi}_{*}\ZZ_{\X}$ and $\H_{e,hor}:=\ker\left\{ \H_{e}\overset{\nabla}{\to}\frac{\H_{e}}{\F_{e}^{m-1}}\otimes\Omega_{S}^{1}(\log\Sigma)\right\} $,
we have short exact sequences \begin{equation}
0 \to \HH_{\ZZ,e} \to \frac {\H_{e(,hor)}}{\F^m_e} \to \J_{e(,hor)} \to 0 \\
\end{equation} and set $\text{ENF}_{(hor)}:=H^{0}(S,\J_{e(,hor)})$.
\begin{thm}
\label{thm infinv}(i) $\sZ\in Z^{m}(\X)_{prim}$ $\implies$ $N\nu_{\sZ}\in\text{ENF}_{hor}$
for some $N\in\NN$; and

(ii) $\nu\in\text{ENF}_{hor}$ with $[\nu]$ torsion $\implies$ $\delta\nu=0$.\end{thm}
\begin{rem}
(ii) is essentially a consequence of the proof of Cor. 2 in \cite{S2}.
For $\nu\in\text{ENF}_{hor}$, $\delta\nu$ lies in the subspace $\HH^{1}\left(S,\F^{m}\overset{\nabla}{\to}\F_{e}^{m-1}\otimes\Omega_{S}^{1}(\log\Sigma)\right)$,
the restriction of $\HH^{1}\left(S^{*},\F^{m}\overset{\nabla}{\to}\F^{m-1}\otimes\Omega_{S^{*}}^{1}\right)\to H^{1}(S^{*},\HH_{\CC})$
to which is injective.
\end{rem}

\subsection{The Hodge Conjecture?}

Putting together Theorem \ref{thm infinv}(ii) and Proposition \ref{prop DEnf},
we see that a horizontal ENF with trivial topological invariant lies
in $H^{0}(S,\J_{fix})=:J^{m}(\X/S)_{fix}$ (constant sections). In
fact, the long-exact sequence associated to (17) yields\[
0\to J^{m}(\X/S)_{fix}\to\text{ENF}_{hor}\overset{[\cdot]}{\to}\frac{Hg^{m}(\X)_{prim}}{\text{im}\left\{ Hg^{m-1}(X_{s_{0}})\right\} }\to0,\]
with $[\nu_{\sZ}]=\overline{[\sZ]}$ (if $\nu_{\sZ}\in\text{ENF}$)
as before. If $\X\overset{\bar{\pi}}{\to}\PP^{1}=S$ is a Lefschetz
pencil on a $2m$-fold $X$, this becomes \begin{equation}\label{eqn lefproof}
\small \xymatrix{
J^m(X) \ar @{^(->} [r] & \text{ENF}_{hor} \ar @{->>} [r]^{[\cdot ]\mspace{130mu}}_{(*)\mspace{130mu}} & Hg^m(X)_{prim} \oplus \ker \left\{ { \begin{matrix} {Hg^{m-1}(B)} \\ {\to Hg^m(X)} \end{matrix} } \right\} \\ 
 & CH^m(\X)_{prim} \ar [u]^{\nu_{(\cdot)}} \\
\ker([\cdot]) \ar [uu]^{AJ} \ar @{^(->} [r] & CH^m(X)_{prim} \ar @/_3pc/ [uu]_{v_{(\cdot)}} \ar [r]^{[\cdot]}_{(**)} \ar [u]^{\beta^*} & Hg^m(X)_{prim} \ar @{^(->} [uu]_{(\text{id.},0)}
} \\
\normalsize
\end{equation}where surjectivity of ({*}) is due to Zucker (cf. Theorems \ref{thm zuck1}-\ref{thm zuck2}
in $\S3$ below; his result followed on work of Griffiths and Bloch
establishing the surjectivity for \emph{sufficiently ample} Lefschetz
pencils). What we are after ($\otimes\QQ$) is surjectivity of the
fundamental class map ({*}{*}). This would clearly follow from surjectivity
of $\nu_{(\cdot)}$, i.e. a Poincar\'e existence theorem, as in $\S1.4$.
By Remark \ref{rem AJ}(ii) this cannot work in most cases; however
we do have
\begin{thm}
The Hodge Conjecture HC$(m,m)$ is true for $X$ if $J^{m}(X_{s_{0}})=J^{m}(X_{s_{0}})_{alg}$
for a general member of the pencil.\end{thm}
\begin{example}
\cite{Zu1} As $J^{2}=J_{alg}^{2}$ is true for cubic threefolds by
the work of Griffiths and Clemens \cite{GC}, HC$(2,2)$ holds for
cubic fourfolds in $\PP^{5}$.
\end{example}
The Lefschetz paradigm, of taking a $1$-parameter family of slices
of a primitive Hodge class to get a normal function and constructing
a cycle by Jacobi inversion, appears to have led us (for the most
part) to a dead end in higher codimension. A beautiful new idea of
Griffiths and Green, to be described in $\S3$, replaces the Lefschetz
pencil by a complete linear system (of higher degree sections of $X$)
so that $\dim(S)\gg1$, and proposes to recover algebraic cycles $dual$
to the given Hodge class from features of the (admissible) normal
function in codimension $\geq2$ on $S$.

\subsection{Deligne cycle-class}

This replaces the fundamental and $AJ$ classes by one object. Writing
$\ZZ(m):=(2\pi i)^{m}\ZZ$, define the Deligne cohomology of $X$
(smooth projective of any dimension) by $H_{\D}^{*}(X^{an},\ZZ(m)):=$\[
H^{*}\left(\text{Cone}\left\{ C_{top}^{\bullet}(X^{an};\ZZ(m))\oplus F^{m}\D^{\bullet}(X^{an})\to D^{\bullet}(X^{an})\right\} [-1]\right),\]
and $c_{\D}:CH^{m}(X)\to H_{\D}^{2m}(X,\ZZ(m))$ by $Z\mapsto(2\pi i)^{m}(Z_{top},\delta_{Z},0)$.
One easily derives the exact sequence\[
0\to J^{m}(X)\to H_{\D}^{2m}(X,\ZZ(m))\to Hg^{m}(X)\to0,\]
which invites comparison to the top row of (\ref{eqn lefproof}).

\section{Limits and Singularities of Normal Functions}

Focusing on the geometric case, we now wish to give the reader a basic
intuition for many of the objects --- singularities, N\'eron models,
limits of NF's and VHS --- which will be treated from a more formal
Hodge-theoretic perspective in later sections.%
\footnote{Owing to our desire to limit preliminaries and/or notational complications
here, there are a few unavoidable inconsistencies of notation between
this and later sections.%
} The first part of this section ($\S\S2.2-8$) considers a cohomologically
trivial cycle on a 1-parameter semistably degenerating family of odd-dimensional
smooth projective varieties. Such a family has two invariants {}``at''
the central singular fibre: 
\begin{itemize}
\item the limit of the Abel-Jacobi images of the intersections of the cycle
with the smooth fibres, and 
\item the Abel-Jacobi image of the intersection of the cycle with the singular
fibre. 
\end{itemize}
We define what these mean and explain the precise sense in which they
agree, which involves limit mixed Hodge strutures and the Clemens-Schmid
sequence, and links limits of $AJ$ maps to the Bloch-Beilinson regulator
on higher $K$-theory.

In the second part, we consider what happens if the cycle is only
assumed to be homologically trivial \emph{fibrewise}. In this case,
just as the fundamental class of a cycle on a variety must be zero
to define its $AJ$ class, the family of cycles has a singularity
class which must be zero in order to define the limit $AJ$ invariant.
Singularities are first introduced for normal functions arising from
families of cycles, and then in the abstract setting of admissible
normal functions (and higher normal functions). At the end we say
a few words about the relation of singularities to the Hodge conjecture,
their r\^ole in multivariable N\'eron models, and the analytic obstructions
to singularities discovered by M. Saito, topics which $\S\S$ 3, 5.1-2,
and 5.3-5 (respectively) will elaborate extensively upon.

We shall begin by recasting $c_{\D}$ from $\S1.9$ in a more formal
vein, which works $\otimes\QQ$. The reader should note that henceforth
in this paper, we have to introduce appropriate Hodge twists (largely
suppressed in $\S1$) into VHS, Jacobians, and related objects.

\subsection{AJ map}

As we saw in $\S1$, this is the basic Hodge-theoretic invariant attached
to a cohomologically trivial algebraic cycle on a smooth projective
algebraic variety $X/\CC$; say $\dim(X)=2m-1$. In the diagram below,
if $cl_{X,\QQ}(Z)=0$ then $Z=\partial\Gamma$ for $\Gamma$ (say)
a rational $C^{\infty}$ $(2m-1)$-chain on $X^{\text{an}}$, and
$\int_{\Gamma}\in(F^{m}H^{2m-1}(X,\CC))^{\vee}$ induces $AJ_{X,\QQ}(Z)$.
\begin{equation}\label{eqn analogydiag1} \small \xymatrix{ & \text{Hom}_{_{\text{MHS}}} \left( \QQ(0), H^{2m}(X,\QQ(m)) \right) \ar @{=} [r]  & (H^{2m}(X))^{(m,m)}_{\QQ} \\ CH^m(X) \ar [ur]^{cl_X} \ar [r] & \text{Ext}^1_{_{D^b\text{MHS}}} \left( \QQ(0), \K^{\bullet} [2m](m) \right) \ar [u] \\  \ker (cl_X) \ar @{^(->} [u] \ar [r]^{AJ_X \mspace{120mu}} & \text{Ext}^1_{_{\text{MHS}}} \left( \QQ(0), H^{2m-1}(X,\QQ(m)) \right) \ar @{=} [r] \ar [u] & J^m(X)_{\QQ} \cong \frac{\left( F^mH^{2m-1}_{\CC} \right)^{\vee}}{H^{2m-1}_{\QQ(m)}} } \\
\normalsize
\end{equation} The middle term in the vertical short-exact sequence, which is isomorphic
to Deligne cohomology and Beilinson's absolute Hodge cohomology $H_{\H}^{2m}(X^{\text{an}},\QQ(m))$,
can be regarded as the ultimate strange fruit of Carlson's work on
extensions of mixed Hodge structures. Here $\K^{\bullet}$ is a canonical
complex of MHS quasi-isomorphic (non-canonically) to $\oplus_{i}H^{i}(X)[-i]$,
constructed from two general configurations of hyperplane sections
$\{H_{i}\}_{i=0}^{2m-1}$, $\{\tilde{H}_{j}\}_{j=0}^{2m-1}$ of $X$.
More precisely, looking (for $|I|,\,|J|\,>\,0$) at the corresponding
{}``cellular'' cohomology groups \[
C_{H,\tilde{H}}^{I,J}(X):=H^{2m-1}(X\backslash\cup_{i\in I}H_{i},\cup_{j\in J}H_{j}\backslash\cdots;\QQ),\]
one sets \[
\K^{\ell}:=\oplus_{I,J:\,|I|-|J|=\ell-2m+1}C_{H,\tilde{H}}^{I,J}(X),\]
cf. \cite{RS}. (Ignoring the description of $J^{m}(X)$ and $AJ$,
and the comparisons to $c_{\D},H_{\D}$, all of this works for smooth
quasi-projective $X$ as well; the vertical short-exact sequence is
true even without smoothness.)

The reason for writing $AJ$ in this way is to make plain the analogy
to (\ref{eqn analogydiag2}) below. We now pass back to $\ZZ$-coefficients.

\subsection{AJ in degenerating families}

To let $AJ_{X}(Z)$ vary with respect to a parameter, consider a semistable
degeneration (SSD) over an analytic disk \begin{equation}
\xymatrix{\X^* \ar @{^(->} [r] \ar [d]^{\pi} & \X \ar [d]^{\bar{\pi}} & X_0 \ar [d] \ar @{=} [r] \ar @{_(->} [l]^{\imath_0} & \cup_i Y_i \\ \Delta^* \ar @{^(->} [r]^{\jmath} & \Delta & \left\{ 0 \right\} \ar @{_(->} [l]} \\
\end{equation} where $X_{0}$ is a reduced NCD with smooth irreducible components
$Y_{i}$, $\X$ is smooth of dimension $2m$, $\bar{\pi}$ is proper
and holomorphic, and $\pi$ is smooth. An algebraic cycle $\sZ\in Z^{m}(\X)$
properly intersecting fibers gives rise to a family\[
Z_{s}:=\sZ\cdot X_{s}\in Z^{m}(X_{s})\,,\,\,\,\,\,\,\,\, s\in\Delta.\]
Assume $0=[\sZ]\in H^{2m}(\X)$ {[}$\implies$ $0=[Z_{s}]\in H^{2m}(X_{s})${]};
then is there a sense in which \begin{equation}\label{eqn limAJ}
\lim_{s\to 0} AJ_{X_s}(Z_s) = AJ_{X_0}(Z_0) ? \\
\end{equation} (Of course, we have yet to say what either side means.)

\subsection{Classical example}

Consider a degeneration of elliptic curves $E_{s}$ which pinches
3 loops in the same homology class to points, yielding for $E_{0}$
three $\PP^{1}$'s joined at $0$ and $\infty$ (called a {}``N\'eron
$3$-gon'' or {}``Kodaira type $I_{3}$'' singular fiber). 

\includegraphics[scale=0.55]{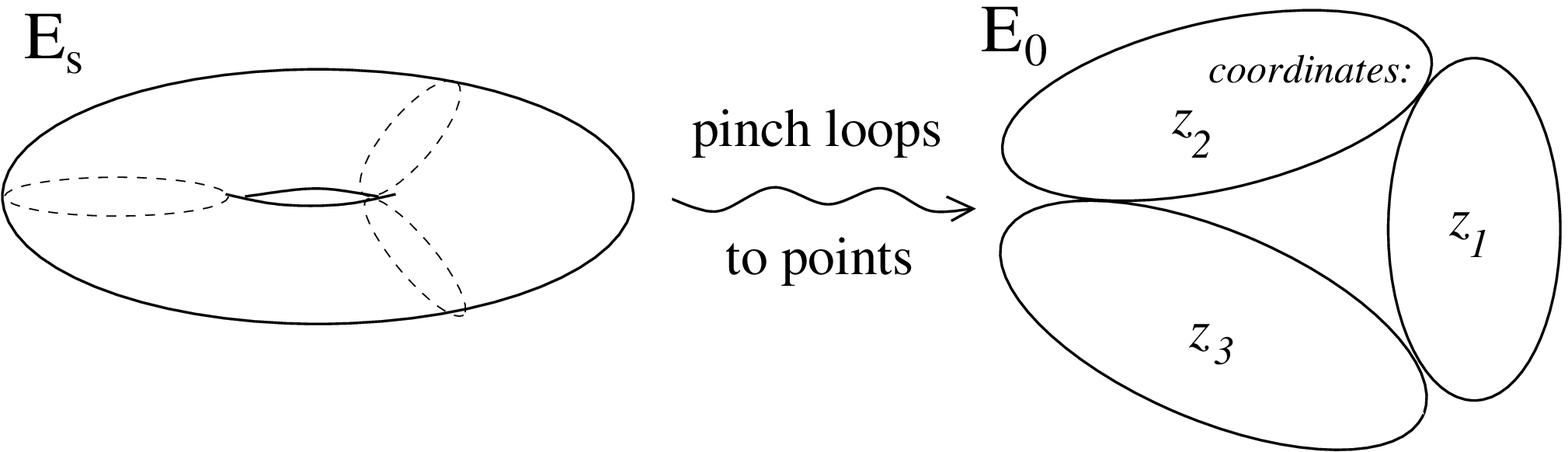}

Denote the total space by $\mathcal{E}\overset{\bar{\pi}}{\to}\Delta$.
One has a family of holomorphic $1$-forms $\omega_{s}\in\Omega^{1}(E_{s})$
limiting to $\{\text{dlog}(z_{j})\}_{j=1}^{3}$ on $E_{0}$; this
can be thought of as a holomorphic section of $R^{0}\bar{\pi}_{*}\Omega_{\mathcal{E}/\Delta}^{1}(\log E_{0})$.

There are two distinct possibilities for limiting behavior when $Z_{s}=p_{s}-q_{s}$
is a difference of points. (These do not include the case where one
or both of $p_{0}$, $q_{0}$ lies in the intersection of two of the
$\PP^{1}$'s, since in that case $\sZ$ is not considered to properly
intersect $X_{0}$.)

Case (I): 

\includegraphics[scale=0.55]{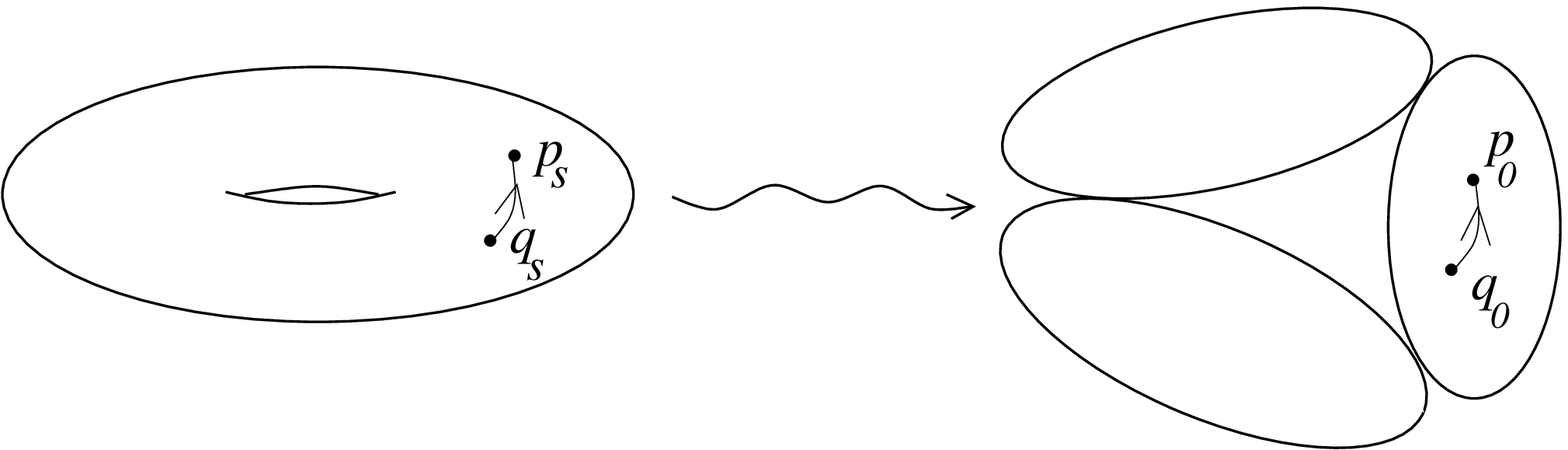}

Here $p_{0}$ and $q_{0}$ lie in the same $\PP^{1}$ (the $j=1$
component, say): in which case $AJ_{E_{s}}(Z_{s})=\int_{q_{s}}^{p_{s}}\omega_{s}\in\CC/\ZZ\left\langle \int_{\alpha_{s}}\omega_{s},\int_{\beta_{s}}\omega_{s}\right\rangle $
limits to $\int_{p_{0}}^{q_{0}}\text{dlog}(z_{1})=\log\left(\frac{z_{1}(p_{0})}{z_{1}(q_{0})}\right)\in\CC/2\pi i\ZZ$.\\
\\
Case (II):

\includegraphics[scale=0.55]{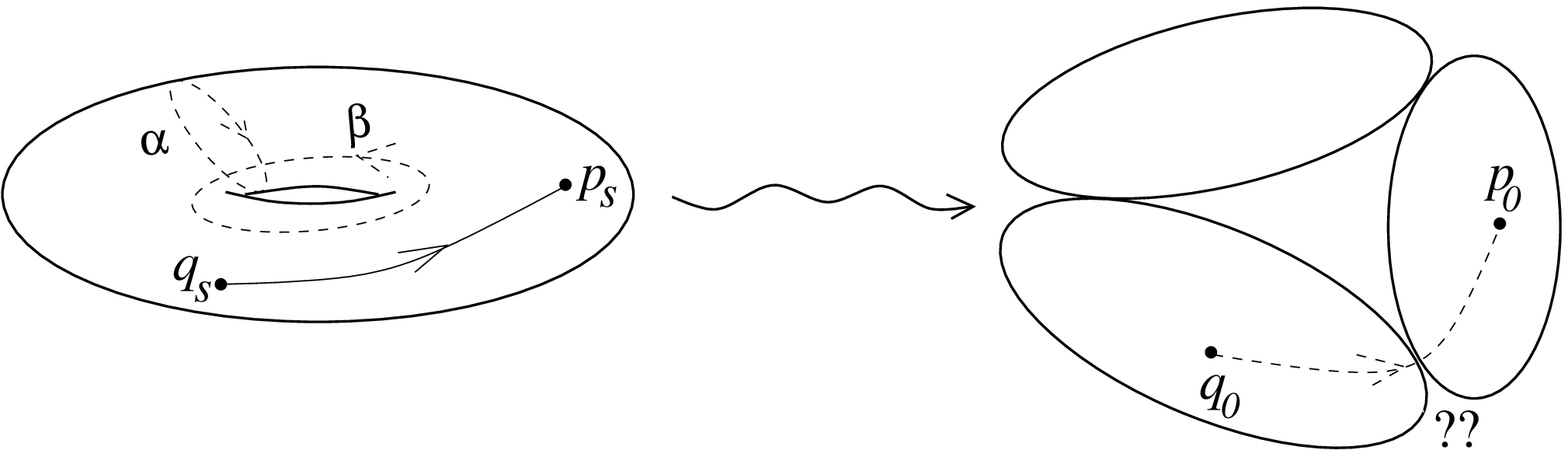}

In this case, $p_{0}$ and $q_{0}$ lie in different $\PP^{1}$ components,
in which case $0\neq[Z_{0}]\in H^{2}(X_{0})$ {[}$\implies$ $[\sZ]\neq0${]}
and we say that $AJ(Z_{0})$ is {}``obstructed''.

\subsection{Meaning of the LHS of (\ref{eqn limAJ})}

If we assume only that $0=[\sZ^{*}]\in H^{2m}(\X^{*})$, then \begin{equation}\label{eqn AJ(Zs)}
AJ_{X_s}(Z_s) \in J^m(X_s) \\
\end{equation} is defined for each $s\in\Delta^{*}$. We can make this into a horizontal,
holomorphic section of a bundle of intermediate Jacobians, which is
what we shall mean henceforth by a \emph{normal function} (on $\Delta^{*}$
in this case). 

Recall the ingredients of a variation of Hodge structure (VHS) over
$\Delta^{*}$\[
\H=((\HH,\H_{\mathcal{O}},\F^{\bullet}),\nabla)\,,\,\,\,\,\nabla\F^{p}\subset\F^{p-1}\otimes\Omega_{S}^{1}\,,\,\,\,\,0\to\HH\to\frac{\H_{\mathcal{O}}}{\F^{m}}\to\J\to0\]
where $\HH=R^{2m-1}\pi_{*}\ZZ(m)$ is a local system, $\H_{\mathcal{O}}=\HH\otimes_{\ZZ}\mathcal{O}_{\Delta^{*}}$
is {[}the sheaf of sections of{]} a holomorphic vector bundle with
holomorphic subbundles $\F^{\bullet}$, and these yield HS's $H_{s}$
fiberwise (notation: $H_{s}=(\HH_{s},H_{s(,\CC)},F_{s}^{\bullet})$).
Henceforth we shall abbreviate $\H_{\mathcal{O}}$ to $\H$. 

Then (\ref{eqn AJ(Zs)}) yields a section of the intermediate Jacobian
bundle\[
\nu_{\sZ}\in\Gamma(\Delta^{*},\J).\]
Any holomorphic vector bundle over $\Delta^{*}$ is trivial, each
trivialization inducing an extension to $\Delta$. The extensions
we want are the {}``canonical'' or {}``privileged'' ones (denoted
$(\cdot)_{e}$); as in $\S1.7$, we define an extended Jacobian bundle
$\J_{e}$ by \begin{equation}\label{eqn defJe} 0\to \jmath_*\HH \to \frac{\H_e}{\F^m_e}\to \J_e \to 0 .\\
\end{equation}
\begin{thm}
\cite{EZ} There exists a holomorphic $\bar{\nu}_{\sZ}\in\Gamma(\Delta,\J_{e})$
extending $\nu_{\sZ}$.
\end{thm}
Define $\lim_{s\to0}AJ_{X_{S}}(Z_{s}):=\bar{\nu}_{\sZ}(0)$ in $(\J_{e})_{0}$,
the fiber over $0$ of the Jacobian bundle. To be precise: since $H^{1}(\Delta,\jmath_{*}\HH)=\{0\}$,
we can lift the $\bar{\nu}_{\sZ}$ to a section of the middle term
of (\ref{eqn defJe}), i.e. of a vector bundle, evaluate at $0$,
then quotient by $(\jmath_{*}\HH)_{0}$.

\subsection{Meaning of the RHS of (\ref{eqn limAJ})}

Higher Chow groups\[
CH^{p}(X,n):=\frac{\left\{ \begin{array}{c}
\text{"admissible, closed" codimension p}\\
\text{algebraic cycles on }X\times\mathbb{A}^{n}\end{array}\right\} }{\text{"higher" rational equivalence}}\]
were introduced by Bloch to compute algebraic $K_{n}$-groups of $X$,
and come with {}``regulator maps'' $reg^{p,n}$ to generalized intermediate
Jacobians\[
J^{p,n}(X):=\frac{H^{2p-n-1}(X,\CC)}{F^{p}H^{2p-n-1}(X,\CC)+H^{2p-n-1}(X,\ZZ(p))}.\]
(Explicit formulas for $reg^{p,n}$ have been worked out by the first
author with J. Lewis and S. M\"uller-Stach in \cite{KLM}.) The singular
fiber $X_{0}$ has motivic cohomology groups $H_{\M}^{*}(X_{0},\ZZ(\cdot))$
built out of higher Chow groups on the substrata \[
Y^{[\ell]}:=\amalg_{|I|=\ell+1}Y_{I}:=\amalg_{|I|=\ell+1}(\cap_{i\in I}Y_{i})\]
(which yield a semi-simplicial resolution of $X_{0}$). Inclusion
induces\[
\imath_{0}^{*}:\, CH^{m}(\X)_{hom}\to H_{\M}^{2m}(X_{0},\ZZ(m))_{hom}\]
and we define $Z_{0}:=\imath_{0}^{*}\sZ$. The $AJ$ map\[
AJ_{X_{0}}:\, H_{\M}^{2m}(X_{0},\ZZ(m))_{hom}\to J^{m}(X_{0}):=\frac{H^{2m-1}(X_{0},\CC)}{\left\{ \begin{array}{c}
F^{m}H^{2m-1}(X_{0},\CC)+\\
H^{2m-1}(X_{0},\ZZ(m))\end{array}\right\} }\]
is built out of regulator maps on substrata, in the sense that the
semi-simplicial structure of $X_{0}$ induces {}``weight'' filtrations
$M_{\bullet}$ on both sides%
\footnote{For the advanced reader, we note that if $M_{\bullet}$ is Deligne's
weight filtration on $H^{2m-1}(X_{0},\ZZ(m))$, then $M_{-\ell}J^{m}(X_{0}):=Ext_{\text{MHS}}^{1}(\ZZ(0),M_{-\ell-1}H^{2m-1}(X_{0},\ZZ(m)))$.
The definition of the $M_{\bullet}$ filtration on motivic cohomology
is much more involved, and we must refer the reader to \cite[sec. III.A]{GGK}.%
} and\[
Gr_{-\ell}^{M}H_{\M}^{2m}(X_{0},\ZZ(m))_{hom}\overset{Gr_{-\ell}^{M}AJ}{\longrightarrow}Gr_{-\ell}^{M}J^{m}(X_{0})\]
boils down to\[
\{\text{subquotient of }CH^{m}(Y^{[\ell]},\ell)\}\overset{\mathit{reg}^{m,\ell}}{\longrightarrow}\{\text{subquotient of }J^{m,\ell}(Y^{[\ell]})\}.\]

\subsection{Meaning of equality in (\ref{eqn limAJ})}

Specializing (\ref{eqn defJe}) to $0$, we have \[
\left(\bar{\nu}_{\sZ}(0)\in\right)\, J_{lim}^{m}(X_{s}):=(\J_{e})_{0}=\frac{(\H_{e})_{0}}{(\F_{e}^{m})_{0}+(\jmath_{*}\HH)_{0}},\]
where $(\jmath_{*}\HH)_{0}$ are the monodromy invariant cycles (and
we are thinking of the fiber $(\H_{e})_{0}$ over $0$ as the limit
MHS of $\H$, see next subsection). H. Clemens \cite{Cl1} constructed
a retraction map $\mathfrak{r}:\X\twoheadrightarrow X_{0}$ inducing
\begin{equation}
\xymatrix{H^{2m-1}(X_0,\ZZ) \ar [rdddd]_{\mu} \ar [r]^{\mathfrak{r}^*} & H^{2m-1}(\X,\ZZ) \ar [d] \\ & \Gamma (\Delta^*,\HH) \ar [d] \\ & \Gamma(\Delta,\jmath_*\HH) \ar [d] \\ & (\jmath_*\HH)_0 \ar @{^(->} [d] \\ & H^{2m-1}_{lim}(X_s,\ZZ) }\\
\end{equation} (where $\mu$ is a morphism of MHS) which in turn induces\[
J(\mu):\, J^{m}(X_{0})\to J_{lim}^{m}(X_{s}).\]

\begin{thm}
\label{thm GGKmain}\cite{GGK} $\lim_{s\to0}AJ_{X_{s}}(Z_{s})=J(\mu)\left(AJ_{X_{0}}(Z_{0})\right).$
\end{thm}

\subsection{Graphing normal functions}

On $\Delta^{*}$, let $T:\,\HH\to\HH$ be the counterclockwise monodromy
transformation, which is unipotent since the degeneration is semistable.
Hence the monodromy logarithm \[
N:=\log(T)=\sum_{k=1}^{2m-1}\frac{(-1)^{k-1}}{k}(T-I)^{k}\]
 is defined, and we can use it to {}``untwist'' the local system
$\otimes\QQ$:\[
\HH_{\QQ}\mapsto\tilde{\HH}_{\QQ}:=\exp\left(-\frac{\log s}{2\pi i}N\right)\HH_{\QQ}\hookrightarrow\H_{e}.\]
In fact, this yields a basis for, and defines, the privileged extension
$\H_{e}$. Moreover, since $N$ acts on $\tilde{\HH}_{\QQ}$, it acts
on $\H_{e}$, and therefore on $(\H_{e})_{0}=H_{lim}^{2m-1}(X_{s})$,
inducing a {}``weight monodromy filtration'' $M_{\bullet}$. Writing
$H=H_{lim}^{2m-1}(X_{s},\QQ(m))$, this is the unique filtration $\{0\}\subset M_{-2m}\subset\cdots\subset M_{2m-2}=H$
satisfying $N(M_{k})\subset M_{k-2}$ and $N^{k}:Gr_{-1+k}^{M}H\overset{\cong}{\to}Gr_{-1-k}^{M}H$
for all $k$. In general it is centered about the weight of the original
variation (cf. the convention in the Introduction).
\begin{example}
In the {}``Dehn twist'' example of $\S1.2$, $N=T-I$ (with $N(\alpha)=0$,
$N(\beta)=\alpha$) so that $\tilde{\alpha}=\alpha$, $\tilde{\beta}=\beta-\frac{\log s}{2\pi i}\alpha$
are monodromy free and yield an $\mathcal{O}_{\Delta}$-basis of $\H_{e}$.
We have $M_{-3}=\{0\}$, $M_{-2}=M_{-1}=\left\langle \alpha\right\rangle $,
$M_{0}=H$.\end{example}
\begin{rem}
Rationally, $\ker(N)=\ker(T-I)$ even when $N\neq T-I$.
\end{rem}
By \cite{Cl1}, $\mu$ maps $H^{2m-1}(X_{0})$ onto $\ker(N)\subset H_{lim}^{2m-1}(X_{s})$
and is compatible with the two $M_{\bullet}$'s; together with Theorem
\ref{thm GGKmain} this implies
\begin{thm}
$\lim_{s\to0}AJ_{X_{s}}(Z_{s})\in J^{m}\left(\ker(N)\right)\,(\subset J_{lim}^{m}(X_{s}))$.
{[}Here we really mean $\ker(T-I)$ so that $J^{m}$ is defined integrally.{]}
\end{thm}
We remark that
\begin{itemize}
\item this was not visible classically for curves ($J^{1}(\ker(N))=J_{lim}^{1}(X_{s})$)
\item replacing $(\J_{e})_{0}$ by $J^{m}(\ker(N))$ yields $\J_{e}'$,
which is a {}``slit-analytic%
\footnote{that is, each point has a neighborhood of the form: open ball about
$\underline{0}$ in $\CC^{a+b}$ intersected with $((\CC^{a}\backslash\{\underline{0}\})\times\CC^{b})\cup(\{\underline{0}\}\times\CC^{c})$,
where $c\leq b$.%
} Hausdorff topological space'' ($\J_{e}$ is non-Hausdorff because
in the quotient topology there are nonzero points in $(\J_{e})_{0}$
that look like limits of points in the zero-section of $\J_{e}$,
hence cannot be separated from $0\in(\J_{e})_{0}$.%
\footnote{see the example before Theorem II.B.9 in \cite{GGK}. %
}) This is the correct extended Jacobian bundle for graphing {}``unobstructed''
(in the sense of the classical example) or {}``singularity-free''
normal functions. Call this the {}``pre-N\'eron-model''.
\end{itemize}

\subsection{Non-classical example}

Take a degeneration of Fermat quintic 3-folds\[
\X\,\,\,\,=\,\,\,\,\begin{array}{c}
\text{semi-stable}\\
\text{reduction}\end{array}\text{ of }\left\{ s\sum_{j=1}^{4}z_{j}^{5}=\prod_{k=0}^{4}z_{k}\right\} \,\,\,\,\subset\,\,\,\,\PP^{4}\times\Delta,\]
so that $X_{0}$ is the union of $5$ $\PP^{3}$'s blown up along
curves $\cong C=\{x^{5}+y^{5}+z^{5}=0\}$. Its motivic cohomology
group $H_{\M}^{4}(X_{0},\QQ(2))_{hom}$ has $Gr_{0}^{M}\cong$10 copies
of $Pic^{0}(C)$, $Gr_{-1}^{M}\cong$40 copies of $\CC^{*}$, $Gr_{-2}^{W}=\{0\}$,
and $Gr_{-3}^{M}\cong K_{3}^{ind}(\CC)$. One has a commuting diagram
\begin{equation}
\xymatrix{H^4_{\M}(X_0,\QQ(2))_{hom} \ar [r]^{\mspace{30mu} AJ_{X_0}} & J^2(X_0)_{\QQ} \ar @{=} [r] & J^2(\ker(N))_{\QQ} \\ K^{ind}_3(\CC) \ar [r]^{reg^{2,3}} \ar @{^(->} [u] & \CC/(2\pi i)^2 \QQ  \ar [r]^{\Im} \ar @{^(->} [u] & \RR } \\
\end{equation} and explicit computations with higher Chow precycles in \cite[sec. 4]{GGK}
lead to the result:
\begin{thm}
There exists a family of $1$-cycles $\sZ\in CH^{2}(\X)_{hom,\QQ}$
such that $Z_{0}\in M_{-3}H_{\M}^{4}$ and $\Im(AJ_{X_{0}}(Z_{0}))=D_{2}(\sqrt{-3})$
(where $D_{2}$ is the Bloch-Wigner function). 
\end{thm}
Hence, $\lim_{s\to0}AJ_{X_{s}}(Z_{s})\neq0$ and so the general $Z_{s}$
in this family is not rationally equivalent to zero. The main idea
is that the family of cycles limits to a (nontrivial) $higher$ cycle
in a substratum of the singular fiber.

\subsection{Singularities in 1 parameter}

If only $[Z_{s}]=0$ ($s\in\Delta^{*}$), and $[\sZ^{*}]=0$ $fails$,
then\[
\lim_{s\to0}AJ\text{ is obstructed}\]
and we say $\bar{\nu}_{\sZ}(s)$ has a singularity (at $s=0$), measured
by the finite group\[
G\cong\frac{Im(T_{\QQ}-I)\cap\HH_{\ZZ}}{Im(T_{\ZZ}-I)}=\left\{ \begin{array}{c}
\ZZ/3\ZZ\text{ in the classical example}\\
\\(\ZZ/5\ZZ)^{3}\text{ in the non-classical ex.}\end{array}\right..\]
(The $(\ZZ/5\ZZ)^{3}$ is generated by differences of lines limiting
to distinct components of $X_{0}$.) The N\'eron model is then obtained
by replacing $J(\ker(N))$ (in the pre-N\'eron-model) by its product
with $G$ (this will graph $all$ admissible normal functions {[}defined
below{]}).

The next example demonstrates the {}``finite-group'' (or torsion)
nature of singularities in the 1-parameter case. In $\S2.10$ we will
see how this feature disappears when there are many parameters.
\begin{example}
\label{ex 1parSING}

Let $\xi\in\CC$ be general and fixed. Then \[
C_{s}=\{x^{2}+y^{2}+s(x^{2}y^{2}+\xi)=0\}\]
defines a family of elliptic curves (in $\PP^{1}\times\PP^{1}$) over
$\Delta^{*}$ degenerating to a N\'eron 2-gon at $s=0$. The cycle
\[
Z_{s}:=\left(i\sqrt{\frac{1+\xi s}{1+s}},1\right)-\left(-i\sqrt{\frac{1+\xi s}{1+s}},1\right)\]
is nontorsion, with points limiting to distinct components. 

\includegraphics[scale=0.5]{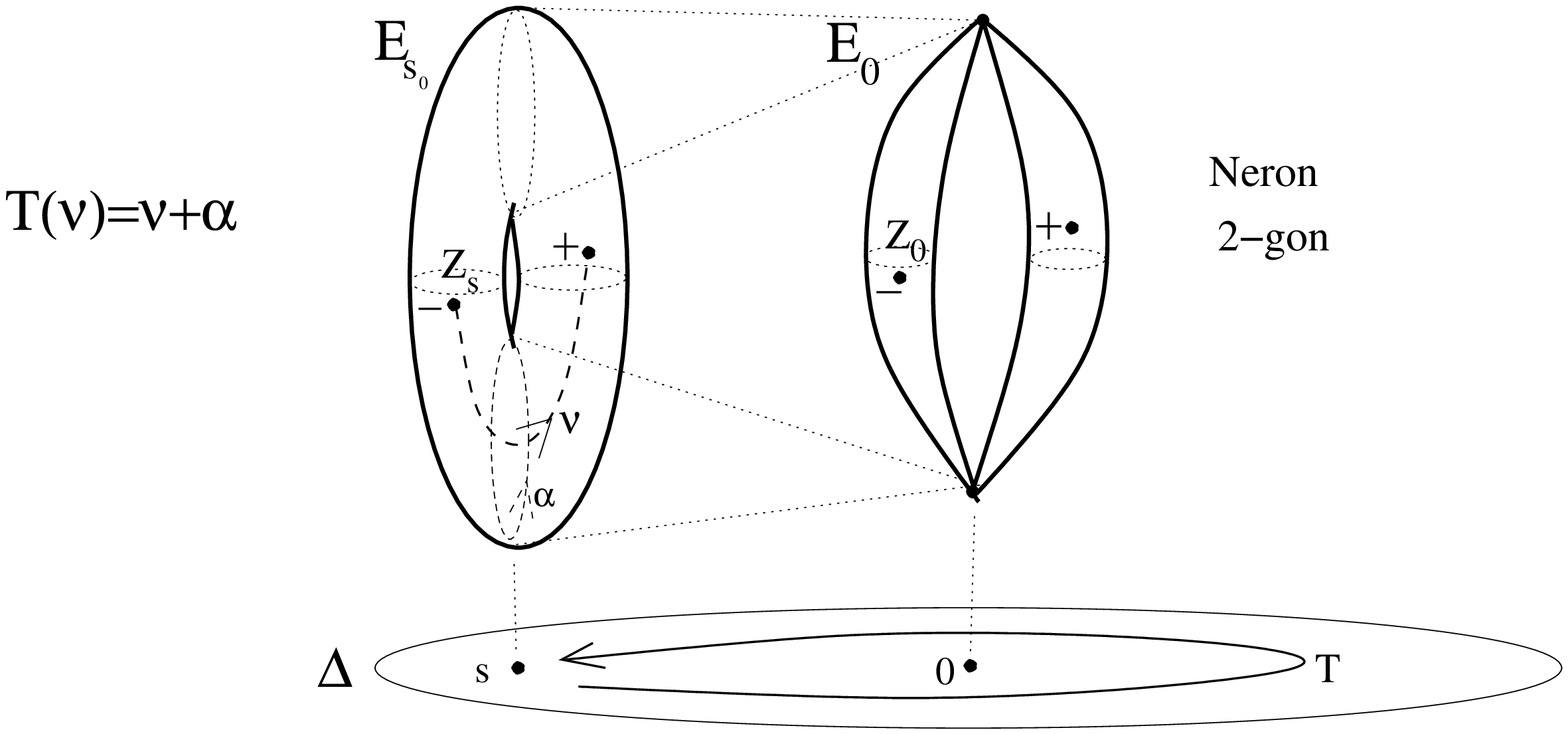}

Hence, $AJ_{C_{s}}(Z_{s})=:\nu(s)$ limits to the non-identity component($\cong\CC^{*}$)
of the N\'eron model. The presence of the non-identity component
removes the obstruction (observed in $\S2.3$ case (II)) to graphing
ANF's with singularities.\end{example}
\begin{itemize}
\item $\otimes\QQ$, we can {}``correct'' this: write $\alpha,\beta$
for a basis for $H^{1}(C_{s})$ and $N$ for the monodromy log about
$0$, which sends $\alpha\mapsto0$ and $\beta\mapsto2\alpha$. Since
$N(\nu)=\alpha=N(\frac{1}{2}\beta)$, $\nu-\frac{1}{2}\beta$ will
pass through the identity component($\cong\CC/\QQ(1)$ after tensoring
with $\QQ$, however).
\item Alternately, to avoid $\otimes\QQ$, one can add a $2$-torsion cycle
like \[
T_{s}:=(i\xi^{\frac{1}{4}},\xi^{\frac{1}{4}})-(-i\xi^{\frac{1}{4}},-\xi^{\frac{1}{4}}).\]

\end{itemize}

\subsection{Singularities in 2 parameters}
\begin{example}
\label{ex 2parSING}Now we will effectively allow $\xi$ (from the
last example) to vary: consider the smooth family \[
C_{s,t}:=\{x^{2}+y^{2}+sx^{2}y^{2}+t=0\}\]
over $(\Delta^{*})^{2}$. The degenerations $t\to0$ and $s\to0$
pinch physically distinct cycles in the same homology class to zero,
so that $C_{0,0}$ is an $I_{2}$; we have obviously that $N_{1}=N_{2}$
(both send $\beta\mapsto\alpha\mapsto0$). Take \[
Z_{s,t}:=\left(i\sqrt{\frac{1+t}{1+s}},1\right)-\left(-i\sqrt{\frac{1+t}{1+s}},1\right)\]
for our family of cycles, which splits between the two components
of the $I_{2}$ at $(0,0)$. 

\includegraphics[scale=0.5]{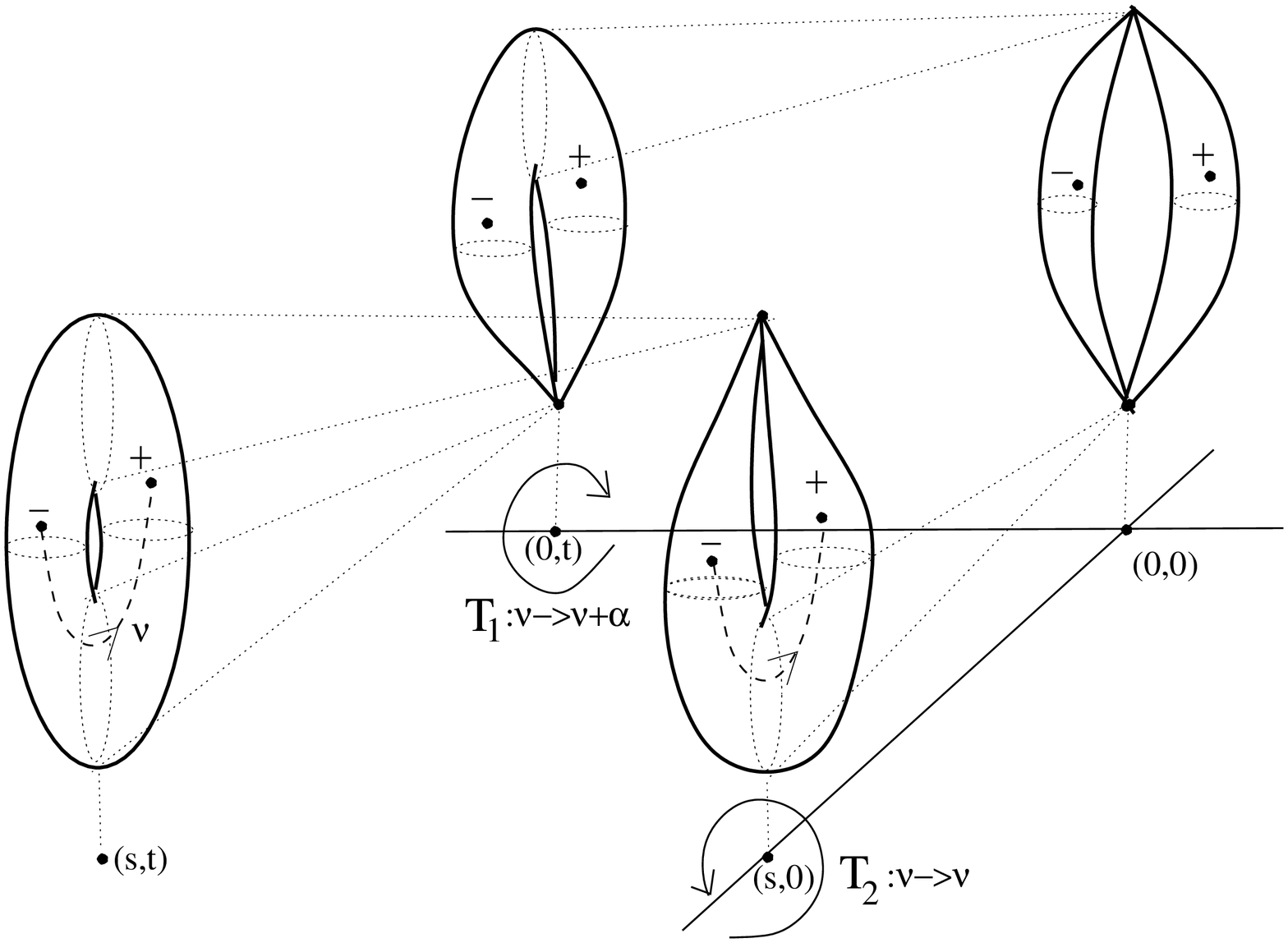}

Things go much more {}``wrong'' here --- here are 3 ways to see
this:\end{example}
\begin{itemize}
\item try to correct monodromy (as we did in Ex. \ref{ex 1parSING} with
$-\frac{1}{2}\beta$): $N_{1}(\nu)=\alpha$, $N_{1}(\beta)=\alpha$,
$N_{2}(\nu)=0$, $N_{2}(\beta)=\alpha$ $\implies$ impossible
\item in $T_{s}$ (from Ex. 1), $\xi^{\frac{1}{4}}$ becomes (here) $\left(\frac{t}{s}\right)^{\frac{1}{4}}$
--- so its obvious extension isn't well-defined. In fact, there is
NO $2$-torsion family of cycles with fiber over $(0,0)$ a difference
of two points in the two distinct components of $C_{0,0}$ (i.e.,
which limits to have the same cohomology class in $H^{2}(C_{0,0})$
as $Z_{0,0}$).
\item take the {}``motivic limit'' of $AJ$ at $t=0$: under the uniformization
of $C_{s,0}$ by \[
\PP^{1}\ni z\longmapsto\left(\frac{2z}{1-sz^{2}},\frac{2iz}{1+sz^{2}}\right),\]
\[
\left[\frac{i}{s}(1+\sqrt{1+s})\right]-\left[\frac{i}{s}(1-\sqrt{1+s})\right]\longmapsto Z_{s,0}.\]
Moreover, the isomorphism $\CC^{*}\cong K_{1}(\CC)\cong M_{-1}H_{\M}^{2}(C_{s,0},\ZZ(1))\,(\ni Z_{s,0})$
sends \[
\frac{1+\sqrt{1+s}}{1-\sqrt{1+s}}\in\CC^{*}\]
to $Z_{s,0}$, and at $s=0$ (considering it as a precycle in $Z^{1}(\Delta,1)$)
this obviously has a residue.
\end{itemize}
The upshot is that $nontorsion$ singularities appear in codimension
2 and up.

\subsection{Admissible normal functions}

We now pass to the abstract setting of a complex analytic manifold
$\bar{S}$ (for example a polydisk or smooth projective variety) with
Zariski open subset $S$, writing $D=\bar{S}\setminus S$ for the
complement. Throughout, we shall assume that $\pi_{0}(S)$ is finite
and $\pi_{1}(S)$ is finitely generated. Let $\V=(\VV,\V_{(\mathcal{O})},\F^{\bullet},W_{\bullet})$
be a variation of MHS over $S$.

\emph{Admissibility} is a condition which guarantees (at each $x\in D$)
a well-defined limit MHS for $\V$ up to the action $\F^{\bullet}\mapsto\exp(\lambda\log T)\F^{\bullet}$
($\lambda\in\CC$) of local unipotent monodromies $T\in\rho(\pi_{1}(U_{x}\cap S))$.
If $D$ is a divisor with local normal crossings at $x$, and $\V$
is admissible, then a choice of coordinates $s_{1},\ldots,s_{m}$
on an analytic neighborhood $U=\Delta^{k}$ of $x$ (with $\{s_{1}\cdots s_{m}=0\}=D$)
produces the LMHS $(\psi_{\underline{s}}\V)_{x}$. Here we shall only
indicate what admissibility, and this LMHS, is in two cases: variations
of pure HS, and generalized normal functions (cf. Definition \ref{def HNF}.

As a consequence of Schmid's nilpotent- and $SL_{2}$-orbit theorems,
pure variation is always admissible. If $\V=\H$ is a pure variation
in one parameter, we have (at least in the unipotent case) already
defined {}``$H_{lim}$'' and now simply replace that notation by
{}``$(\psi_{\underline{s}}\H)_{x}$''. In the multiple parameter
(or non-unipotent) setting, simply pull the variation back to an analytic
curve $\Delta^{*}\to(\Delta^{*})^{m}\times\Delta^{k-m}\subset S$
whose closure passes through $x$, and take the LMHS of that. The
resulting $(\psi_{\underline{s}}\H)_{x}$ is independent of the choice
of curve (up to the action of local monodromy mentioned earlier).
In particular, letting $\{N_{i}\}$ denote the local monodromy logarithms,
the weight filtration $M_{\bullet}$ on $(\psi_{\underline{s}}\H)_{x}$
is just the weight monodromy filtration attached to their sum $N:=\sum a_{i}N_{i}$
(where the $\{a_{i}\}$ are arbitrary positive integers). 

Now let $r\in\NN$.
\begin{defn}
\label{def HNF}A \emph{(higher) normal function} over $S$ is a VMHS
of the form $\V$ in (the short-exact sequence) \begin{equation}
0 \to \H \longrightarrow \V \longrightarrow \ZZ_S(0) \to 0 \\
\end{equation} where $\H$ is a {[}pure{]} VHS of weight $(-r)$ and the {[}trivial,
constant{]} variation $\ZZ_{S}(0)$ has trivial monodromy. (The terminology
{}``higher'' only applies when $r>1$.) This is equivalent to a
holomorphic, horizontal section of the generalized Jacobian bundle
$J(\H):=\frac{\H}{\F^{0}\H+\HH_{\ZZ}}$.\end{defn}
\begin{example}
Given a smooth proper family $\X\overset{\pi}{\to}S$, with $x_{0}\in S$.
A higher algebraic cycle $\sZ\in CH^{p}(\X,r-1)_{prim}:=\ker\{CH^{p}(\X,r-1)\to CH^{p}(X_{x_{0}},r-1)\to\mathit{Hg}^{p,r-1}(X_{x_{0}})\}$
yields a section of $J(R^{2p-r}\pi_{*}\CC\otimes\mathcal{O}_{S})=:\J^{p,r-1}$;
this is what we shall mean by a \emph{(higher) normal function of
geometric origin}.%
\footnote{Note that $\mathit{Hg}^{p,r-1}(X_{x_{0}})_{\QQ}:=H^{2p-r+1}(X_{x_{0}},\QQ(p))\cap F^{p}H^{2p-r+1}(X_{x_{0}},\CC)$
is actually zero for $r>1$, so that the {}``$prim$'' comes for
free for some multiple of $\sZ$.%
} (The notion of \emph{motivation over $K$} likewise has an obvious
extension from the classical $1$-parameter case in $\S1$.)

We now give the definition of admissibility for VMHS of the form in
Defn. \ref{def HNF} (but simplifying to $D=\{s_{1}\cdots s_{k}=0\}$),
starting with the local unipotent case. For this we need Deligne's
definition \cite{De1} of the $I^{p,q}(H)$ of a MHS $H$, for which
the reader may refer to Theorem \ref{thm Ipq's} (in $\S4$) below.
To simplify notation, we shall abbreviate $I^{p,q}(H)$ to $H^{(p,q)}$,
so that e.g. $H_{\QQ}^{(p,p)}=I^{p,p}(H)\cap H_{\QQ}$, and drop the
subscript $x$ for the LMHS notation.\end{example}
\begin{defn}
\label{def admissibility}Let $S=(\Delta^{*})^{k}$, $\V\in NF^{r}(S,\H)_{\QQ}$
(i.e., as in Definition \ref{def HNF}, $\otimes\QQ$), and $x=(\underline{0})$.

(I) {[}unipotent case{]} Assume the monodromies $T_{i}$ of $\HH$
are unipotent, so that the logarithms $N_{i}$ and associated monodromy
weight filtrations $M_{\bullet}^{(i)}$ are defined. (Note that the
$\{N_{i}\}$ resp. $\{T_{i}\}$ automatically commute since any local
system must be a representation of $\pi_{1}((\Delta^{*})^{k})$, an
abelian group.) We may {}``untwist'' the local system $\otimes\QQ$
via $\tilde{\VV}:=exp\left(\frac{-1}{2\pi\sqrt{-1}}\sum_{i}\log(s_{i})N_{i}\right)\VV_{(\QQ)}$,
and set $\V_{e}:=\tilde{\VV}\otimes\mathcal{O}_{\Delta^{k}}$ for
the Deligne extension. Then $\V$ is ($\bar{S}$-)admissible iff

$\quad$(a) $\H$ is polarizable

$\quad$(b) $\exists$ lift $\nu_{\QQ}\in(\tilde{\VV})_{0}$ of $1\in\QQ(0)$
such that $N_{i}\nu_{\QQ}\in M_{-2}^{(i)}(\psi_{\underline{s}}\H)_{\QQ}$
($\forall i$)

$\quad$(c) $\exists$ lift $\nu_{F}(s)\in\Gamma(\overline{S},\V_{e})$
of $1\in\QQ_{S}(0)$ such that $\nu_{F}|_{S}\in\Gamma(S,\F^{0})$.

(II) In general there exists a minimal finite cover $\zeta:\,(\Delta^{*})^{k}\to(\Delta^{*})^{k}$
(sending $\underline{s}\mapsto\underline{s}^{\underline{\mu}}$) such
that the $T_{i}^{\mu_{i}}$ are unipotent. $\V$ is admissible iff
$\zeta^{*}\V$ satisfies (a),(b),(c).
\end{defn}
The main result \cite{K,SZ} is then that $\V\in NF^{r}(S,\H)_{\bar{S}}^{ad}$
has well-defined $\psi_{\underline{s}}\V$, given as follows. On the
underlying rational structure $(\tilde{\VV})_{0}$ we put the weight
filtration $M_{i}=M_{i}\psi_{\underline{s}}\H+\QQ\left\langle \nu_{\QQ}\right\rangle $
for $i\geq0$ and $M_{i}=M_{i}\psi_{\underline{s}}\H$ for $i<0$;
while on its complexification ($\cong(\V_{e})_{0}$) we put the Hodge
filtration $F^{j}=F^{j}\psi_{\underline{s}}\H_{\CC}+\CC\left\langle \nu_{F}(0)\right\rangle $
for $j\leq0$ and $F^{j}=F^{j}\psi_{\underline{s}}\H$ for $j>0$.
(Here we are using the inclusion $\tilde{\HH}\subset\tilde{\VV}$,
and the content of the statement is that this actually does define
a MHS.) 

We can draw some further conclusions from (a)-(c). With some work,
from (I)(c) it follows that

(c') $\nu_{F}(0)$ gives a lift of $1\in\QQ(0)$ satisfying $N_{i}\nu_{F}(0)\in(\psi_{\underline{s}}\H)^{(-1,-1)};$ 

and one can also show that the $N_{i}\nu_{\QQ}\in M_{-2}(\psi_{\underline{s}}\H)_{\QQ}$
($\forall i$). Furthermore, if $r=1$ then each $N_{i}\nu_{\QQ}$
{[}resp. $N_{i}\nu_{F}(0)${]} belongs to the image under $N_{i}:\,\psi_{\underline{s}}\H\to\psi_{\underline{s}}\H(-1)$
of a rational {[}resp. type-$(0,0)${]} element. (To see this, use
the properties of $N_{i}$ to deduce that $im(N_{i})\supseteq M_{-r-1}^{(i)}$;
then note that for $r=1$ we have, from (b) and (c), $N_{i}\nu_{F}(0),\, N_{i}\nu_{\QQ}\in M_{-2}^{(i)}$.)

(III) The definition of admissibility over an arbitrary smooth base
$S$ together with good compactification $\bar{S}$ is then local,
i.e. reduces to the $(\Delta^{*})^{k}$ setting. Another piece of
motivation for the definition of admissibility is this, for which
we refer the reader to \cite[Thm. 7.3]{BZ}:
\begin{thm}
Any (higher) normal function of geometric origin is admissible.
\end{thm}

\subsection{Limits and singularities of ANF's}

Now the idea of the {}``limit of a normal function'' should be to
interpret $\psi_{\underline{s}}\V$ as an extension of $\QQ(0)$ by
$\psi_{\underline{s}}\H$. The obstruction to being able to do this
is the singularity, as we now explain. All MHS in this section are
$\QQ$-MHS.

According to \cite[Cor. 2.9]{BFNP} $NF^{r}(S,\H)_{\bar{S}}^{ad}\otimes\QQ\cong Ext_{VMHS(S)_{\bar{S}}^{ad}}^{1}(\QQ(0),\H)$,
and one has an equivalence of categories $VMHS(S)_{\bar{S}}^{ad}\simeq MHM(S)_{\bar{S}}^{ps}$.
We want to push (in a sense canonically extend) our ANF $\V$ into
$\bar{S}$ and restrict the result to $x$. Of course, writing $\jmath:S\hookrightarrow\bar{S}$,
$\jmath_{*}$ is not right exact; so to preserve our extension, we
take the derived functor $R\jmath_{*}$ and land in the derived category
$D^{b}MHM(\bar{S})$. Pulling back to $D^{b}MHM(\{x\})\cong D^{b}MHS$
by $\imath_{x}^{*}$, we have defined an invariant $(\imath_{x}^{*}R\jmath_{*})^{Hdg}$:
\begin{equation}\label{eqn analogydiag2} \xymatrix{& & \text{Hom}_{_{\text{MHS}}} \left( \QQ(0), H^1\K^{\bullet} \right)   \\ NF^r(S,\H)  \ar [urr]^{sing_x} \ar [rr]_{(\imath^*_xR\jmath_*)^{\text{Hdg}} \mspace{120mu}} & & \text{Ext}^1_{_{D^b\text{MHS}}} \left( \QQ(0), \K^{\bullet} := \imath_x^* Rj_* \H \right) \ar [u] \\  \ker (sing_x) \ar @{^(->} [u] \ar [rr]^{lim_x \mspace{60mu}} & & \text{Ext}^1_{_{\text{MHS}}} \left( \QQ(0), H^0 \K^{\bullet} \right) \ar [u] } \\
\end{equation} where the diagram makes a clear analogy to (\ref{eqn analogydiag1}). 

For $S=(\Delta^{*})^{k}$ and $\HH_{\ZZ}$ unipotent \[
\K^{\bullet}\simeq\left\{ \psi_{\underline{s}}\H\overset{\oplus N_{i}}{\longrightarrow}\oplus_{i}\psi_{\underline{s}}\H(-1)\longrightarrow\oplus_{i<j}\psi_{\underline{s}}\H(-2)\longrightarrow\cdots\right\} ,\]
and \[
sing_{x}:\, NF^{r}((\Delta^{*})^{k},\H)_{\Delta^{k}}^{ad}\to(H^{1}\K^{\bullet})_{\QQ}^{(0,0)}\,\,\left(\cong{coker}(N)(-1)\text{ for k=1}\right)\]
is induced by $\V\mapsto\{N_{i}\nu_{\QQ}\}\equiv\{N_{i}\nu_{F}(0)\}$.
The limits, which are computed by\[
lim_{x}:\,\ker(sing_{x})\to J(\cap_{i}\ker(N_{i})),\]
more directly generalize the 1-parameter picture. The target $J(\cap\ker(N_{i}))$
is exactly what to put in over $\underline{0}$ to get the multivariable
pre-N\'eron-model.

We have introduced the general case $r\geq1$ because of interesting
applications of higher normal functions to irrationality proofs, local
mirror symmetry \cite{DK}. In case $r=1$ --- i.e. we are dealing
with classical normal functions --- we can replace $R\jmath_{*}$
in the above by perverse intermediate extension $\jmath_{!*}$ (which
by a lemma in \cite{BFNP} preserves the extension in this case, cf.
Thm. \ref{thm BFNP2.11} below). Correspondingly, $\K^{\bullet}$
is replaced by the local intersection cohomology complex \[
\K_{red}^{\bullet}\simeq\left\{ \psi_{\underline{s}}\H\overset{\oplus N_{i}}{\longrightarrow}\oplus_{i}\text{Im}(N_{i})(-1)\longrightarrow\oplus_{i<j}\text{Im}(N_{i}N_{j})(-2)\to\cdots\right\} ;\]
while the target for $lim_{x}$ is unchanged, that for $sing_{x}$
is reduced to $0$ if $k=1$ and \begin{equation}\label{eqn singtarget codim2}\left( \frac{\ker(N_1)\cap \text{im}(N_2)}{N_2(\ker N_1)} \right) ^{(-1,-1)}_{\QQ}\\
\end{equation} if $k=2$.

\subsection{Applications of singularities}

We hint at some good things to come:

(i) Replacing the $sing_{x}$-target (e.g., (\ref{eqn singtarget codim2}))
by actual $images$ of ANF's, and using their differences to glue
pre-N\'eron components together yields a generalized Neron model
(over $\Delta^{r}$, or $\bar{S}$ more generally) graphing ANF's.
Again over $x$ one gets an extension of a discrete (but not necessarily
finite) singularity group by the torus $J(\cap\ker(N_{i}))$. A. Young
\cite{Yo} did this for abelian varieties, then \cite{BPS} for general
VHS. This will be described more precisely in $\S5$.2.

(ii) (Griffiths and Green \cite{GG}) The Hodge Conjecture (HC) on
a $2p$-dimensional smooth projective variety $X$ is equivalent to
the following statement for \emph{each} primitive Hodge $(p,p)$ class
$\zeta$ and very ample line bundle $\L\to X$: there exists $k\gg0$
such that the natural normal function%
\footnote{cf. $\S\S3.2-3$, especially \eqref{eqn nf2}.%
} $\nu_{\zeta}$ over $|\L^{k}|\setminus\hat{X}$ (the complement of
the dual variety in the linear system) has a nontorsion singularity
at some point of $\hat{X}$. So, in a \emph{sense}, the analogue of
HC for $(\Delta^{*})^{k}$ is surjectivity of $sing_{x}$ onto $(H^{1}\K_{red}^{\bullet})_{\QQ}^{(0,0)}$,
and this $fails$:

(iii) (M. Saito \cite{S6}, Pearlstein \cite{Pe3}) Let $\H_{0}/\Delta^{*}$
be a VHS of weight $3$ rank $4$ with nontrivial Yukawa coupling.
Twisting it into weight $-1$, assume the LMHS is of type $II_{1}$:
$N^{2}=0$, $rk(Gr_{-2}^{M})=1$. Take for $\H/(\Delta^{*})^{2}$
the pullback of $\H_{0}$ by $(s,t)\mapsto st$. Then (\ref{eqn singtarget codim2})$\neq\{0\}=sing_{\underline{0}}\{NF^{1}((\Delta^{*})^{2},\H)_{\Delta^{2}}^{ad}\}.$
The obstruction to the existence of normal functions with nontrivial
singularity is analytic; and comes from a differential equation produced
by the horizontality condition (see $\S5.4-5$).

(iv) One can explain the meaning of the residue of the limit $K_{1}$
class in Example \ref{ex 2parSING} above: writing $\jmath^{1}:(\Delta^{*})^{2}\hookrightarrow\Delta^{*}\times\Delta$,
$\jmath^{2}:\Delta^{*}\times\Delta\hookrightarrow\Delta^{2}$, factor
$(\imath_{x}^{*}Rj_{!*})^{\text{Hdg}}$ by $(\imath_{x}^{*}R\jmath_{*}^{2})^{\text{Hdg}}\circ(\imath_{\Delta^{*}}^{*}\jmath_{!*}^{1})^{\text{Hdg}}$
(where the $\imath^{*}Rj_{*}^{2}$ corresponds to the residue). That
is, limit a normal function (or family of cycles) to a higher normal
function (or family of higher Chow cycles) over a codimension-$1$
boundary component; the latter can then have (unlike normal functions)
a singularity in codimension $1$ --- i.e. in codimension $2$ with
respect to the original normal function. 

This technique gives a quick proof of the existence of singularities
for the Ceresa cycle by limiting it to an Eisenstein symbol (see \cite{Co}
and the Introduction to \cite{DK}). Additionally, one gets a geometric
explanation of why one does not expect the singularities in (ii) to
be supported in high-codimension substrata of $\hat{X}$ (supporting
very degenerate hypersurfaces of $X$): along these substrata one
may reach (in the sense of (iv)) higher Chow cycles with rigid $AJ$
invariants, hence no residues. For this reason codimension $2$ tends
to be a better place to look for singularities than in much higher
codimension. These {}``shallow'' substrata correspond to hypersurfaces
with ordinary double points, and it was the original sense of \cite{GG}
that such points should trace out an algebraic cycle {}``dual''
to the original Hodge class, giving an \emph{effective} proof of the
HC.

\section{Normal Functions and the Hodge Conjecture}

In this section, we discuss the connection between normal functions
and the Hodge conjecture, picking up where $\S1$ left off. We begin
with a review of some properties of the Abel--Jacobi map. Unless otherwise
noted, all varieties are defined over $\CC$.

\subsection{Zucker's Theorem on Normal Functions}

Let $X$ be a smooth projective variety of dimension $d_{X}$. Recall
that $J_{h}^{p}(X)$ is the intermediate Jacobian associated to the
maximal rationally defined Hodge substructure $H$ of $H^{2p-1}(X)$
such that $H_{\CC}\subset H^{p,p-1}(X)\oplus H^{p-1,p}(X)$, and that
(by a result of Lieberman \cite{Li}) \begin{equation}\label{eqn lieberman}
\begin{matrix} {J^p(X)_{alg}=\text{im} \left\{ AJ_X:CH^p(X)_{alg}\to J^p(X) \right\}} \\ {\text{is a sub- abelian variety of }J^p(X)_h .} \end{matrix}  \\
\end{equation}
\begin{notation}
If $f:X\to Y$ is a projective morphism then $f^{sm}$ denotes the
restriction of $f$ to the largest Zariski open subset of $Y$ over
which $f$ is smooth. Also, unless otherwise noted, in this section,
the underlying lattice $\HH_{\ZZ}$ of every variation of Hodge structure
is assumed to be torsion free, and hence for a geometric family $f:X\to Y$,
we are really considering $\HH_{\ZZ}=(R^{k}f_{*}^{sm}\ZZ)/\{{\rm torsion}\}$. 
\end{notation}
As reviewed in $\S1$, Lefschetz proved that every integral $(1,1)$
class on a smooth projective surface is algebraic by studying Poincar?normal functions associated to such cycles. We shall begin here by
revisiting Griffiths's program (also recalled in $\S1$) to prove
the Hodge conjecture for higher codimension classes by extending Lefschetz's
methods: By induction on dimension, the Hodge conjecture can be reduced
to the case of middle dimensional Hodge classes on even dimensional
varieties \cite[Lec. 14]{Le1}. Suppose therefore that $X\subseteq\mathbb{P}^{k}$
is a smooth projective variety of dimension $2m$. Following \cite[Sec. 4]{Zu2},
let us pick a Lefschetz pencil of hyperplane sections of $X$, i.e.
a family of hyperplanes $H_{t}\subseteq\mathbb{P}^{k}$ of the form
$t_{0}w_{0}+t_{1}w_{1}=0$ parametrized by $t=[t_{0},t_{1}]\in\mathbb{P}^{1}$
relative to a suitable choice of homogeneous coordinates $w=[w_{0},\dots,w_{k}]$
on $\mathbb{P}^{k}$ such that: 
\begin{itemize}
\item For all but finitely many points $t\in\mathbb{P}^{1}$, the corresponding
hyperplane section of $X_{t}=X\cap H_{t}$ is smooth; 
\item The base locus $B=X\cap\{w\in\mathbb{P}^{k}\mid w_{0}=w_{1}=0\}$
is smooth; 
\item Each singular hyperplane section of $X$ has exactly one singular
point, which is an ordinary double point. 
\end{itemize}
Given such a Lefschetz pencil, let \[
Y=\{\,(x,t)\in X\times\mathbb{P}^{1}\mid x\in H_{t}\,\}\]
 and $\pi:Y\to\mathbb{P}^{1}$ denote projection onto the second factor.
Let $U$ denote the set of points $t\in\mathbb{P}^{1}$ such that
$X_{t}$ is smooth and $\mathcal{H}$ be the variation of Hodge structure
over $U$ with integral structure $\HH_{\ZZ}=R^{2m-1}\pi_{*}^{sm}\ZZ(m)$.
Furthermore, by Schmid's nilpotent orbit theorem \cite{Sc}, the Hodge
bundles $\mathcal{F}^{\bullet}$ have a canonical extension to a system
of holomorphic bundles $\mathcal{F}_{e}^{\bullet}$ over $\mathbb{P}^{1}$.
Accordingly, we have a short exact sequence of sheaves \begin{equation}
0 \to j_* \HH_{\ZZ} \to \H_e/\F^m_e \to \J^m_e \to 0 \\
\end{equation} where $j:U\to\mathbb{P}^{1}$ is the inclusion map. As before, let
us call an element $\nu\in H^{0}(\mathbb{P}^{1},\mathcal{J}_{e}^{m})$
a Poincar?normal function. Then, we have the following two results
\cite[Thms. 4.57, 4.17]{Zu2}, the second of which is known as {}``the
Theorem on Normal Functions'':
\begin{thm}
\label{thm zuck1}Every Poincar?normal function satisfies Griffiths
horizontality.
\end{thm}
$\vspace{-10mm}$
\begin{thm}
\label{thm zuck2}Every primitive integral Hodge class on $X$ is
the cohomology class of a Poincar?normal function. 
\end{thm}
The next step in the proof of the Hodge conjecture via this approach
is to show that for $t\in U$, the Abel--Jacobi map \[
AJ:CH^{m}(X_{t})_{hom}\to J^{m}(X_{t})\]
is surjective. However, for $m>1$ this is rarely true (even granting
the conjectural equality of $J^{m}(X)_{alg}$ and $J_{h}^{m}(X)$)
since $J^{m}(X_{t})\neq J_{h}^{m}(X_{t})$ unless $H^{2m-1}(X_{t},\CC)=H^{m,m-1}(X_{t})\oplus H^{m-1,m}(X_{t})$.
In plenty of cases of interest $J_{h}^{m}(X)$ is in fact trivial;
Theorem \ref{thm trivJH1} and Example \ref{ex trivJH2} below give
two different instances of this.
\begin{thm}
\label{thm trivJH1} \cite[Ex. 14.18]{Le1} If $X\subseteq\mathbb{P}^{k}$
is a smooth projective variety of dimension $2m$ such that $H^{2m-1}(X)=0$
and $\{X_{t}\}$ is a Lefschetz pencil of hyperplane sections of $X$
such that $F^{m+1}H^{2m-1}(X_{t})\neq0$ for every smooth hyperplane
section, then for generic $t\in U$, $J_{h}^{m}(X_{t})=0$.
\end{thm}
$\vspace{-10mm}$
\begin{thm}
If $J_{h}^{p}(X)=0$, then the image of $CH^{m}(W)_{hom}$ in $J^{p}(X)$
under the Abel--Jacobi map is countable.\end{thm}
\begin{proof}
(Sketch) As a consequence of (\ref{eqn lieberman}), if $J_{h}^{p}(X)=0$
the Abel--Jacobi map vanishes on $CH^{p}(X)_{alg}$. Therefore, the
cardinality of the image of the Abel-Jacobi map on $CH^{p}(X)_{hom}$
is bounded by the cardinality of the Griffiths group $CH^{p}(X)_{hom}/CH^{p}(X)_{alg}$,
which is known to be countable.\end{proof}
\begin{example}
\label{ex trivJH2}Specific hypersurfaces with $J_{h}^{p}(X)=0$ were
constructed by Shioda \cite{Sh}: Let $Z_{m}^{n}$ denote the hypersurface
in $\mathbb{P}^{n+1}$ defined by the equation \[
\sum_{i=0}^{n+1}\, x_{i}x_{i+1}^{m-1}=0\qquad(x_{n+2}=x_{0})\]
 Suppose that $n=2p-1>1$, $m\geq2+3/(p-1)$ and \[
d_{0}=\{(m-1)^{n+1}+(-1)^{n+1}\}/m\]
 is prime. Then $J_{h}^{p}(Z_{m}^{n})=0$.
\end{example}

\subsection{Singularities of admissible normal functions}

In \cite{GG}, Griffiths and Green proposed an alternative program
for proving the Hodge conjecture by studying the singularities of
normal functions over higher dimensional parameter spaces. Following
\cite{BFNP}, let $S$ a complex manifold and $\mathcal{H}=(\HH_{\ZZ},\mathcal{F}^{\bullet}\mathcal{H}_{\mathcal{O}})$
be a variation of polarizable Hodge structure of weight $-1$ over
$S$. Then, we have the short exact sequence\[
0\to\HH_{\ZZ}\to\H/\F^{0}\to J(\H)\to0\]
of sheaves and hence an associated long exact sequence in cohomology.
In particular, the cohomology class ${\rm cl}(\nu)$ of a normal function
$\nu\in H^{0}(S,J(\mathcal{H}))$ is just the image of $\nu$ under
the connecting homomorphism \[
\partial:H^{0}(S,J(\mathcal{H}))\to H^{1}(S,\HH_{\ZZ}).\]

Suppose now that $S$ is a Zariski open subset of a smooth projective
variety $\bar{S}$. Then, the singularity of $\nu$ at $p\in\bar{S}$
is the quantity \[
\sigma_{\ZZ,p}(\nu)=\varinjlim_{p\in U}\,{\rm cl}(\nu|_{U\cap S})\in\varinjlim_{p\in U}\, H^{1}(U\cap S,\HH_{\ZZ})=(R^{1}j_{*}\HH_{\ZZ})_{p}\]
where the limit is taken over all analytic open neighborhoods $U$
of $p$, and $j:S\to\bar{S}$ is the inclusion map. The image of $\sigma_{\ZZ,p}(\nu)$
in cohomology with rational coefficients will be denoted $\text{sing}_{p}(\nu_{\zeta})$.
\begin{rem}
If $p\in S$ then $\sigma_{\ZZ,p}(\nu)=0$.\end{rem}
\begin{thm}
\cite{S1} Let $\nu$ be an admissible normal function on a Zariski
open subset of a curve $\bar{S}$. Then, $\sigma_{\ZZ,p}(\nu)$ is
of finite order for each point $p\in\bar{S}$.\end{thm}
\begin{proof}
By \cite{S1}, an admissible normal function $\nu:S\to J(\mathcal{H})$
is equivalent to an extension \begin{equation}
0 \to \H \to \V \to \ZZ(0) \to 0 \\
\end{equation} in the category of admissible variations of mixed Hodge structure.
By the monodromy theorem for variations of pure Hodge structure, the
local monodromy of $\V$ about any point $p\in\bar{S}-S$ is always
quasi-unipotent. Without loss of generality, let us assume that it
is unipotent and that $T=e^{N}$ is the local monodromy of $\V$ at
$p$ acting on some fixed reference fiber with integral structure
$V_{\ZZ}$. Then, due to the length of the weight filtration $W$,
the existence of the relative weight filtration of $W$ and $N$ is
equivalent to the existence of an $N$-invariant splitting of $W$
\cite[Prop. 2.16]{SZ}. In particular, let $e_{\ZZ}\in V_{\ZZ}$ project
to $1\in Gr_{0}^{W}\cong\ZZ(0)$. Then, by admissibility, there exists
an element $h_{\QQ}\in H_{\QQ}=W_{-1}\cap V_{\QQ}$ such that \[
N(e_{\ZZ}+h_{\QQ})=0\]
and hence $(T-I)(e_{\ZZ}+h_{\QQ})=0$.%
\footnote{Alternatively, one can just derive this from Defn. \eqref{def admissibility}(I).%
} Any two such choices of $e_{\ZZ}$ differ by an element $h_{\ZZ}\in W_{-1}\cap V_{\ZZ}$.
Therefore, an admissible normal function $\nu$ determines a class
\[
[\nu]=[(T-I)e_{\ZZ}]\in\frac{(T-I)(H_{\QQ})}{(T-I)(H_{\ZZ})}\]
 Tracing through the definitions, one finds that the left hand side
of this equation can be identified with $\sigma_{\ZZ,p}(\nu)$ whereas
the right hand side is exactly the torsion subgroup of $(R^{1}j_{*}\HH_{\ZZ})_{p}$. \end{proof}
\begin{defn}
\cite{BFNP} An admissible normal function $\nu$ defined on a Zariski
open subset of $\bar{S}$ is singular on $\bar{S}$ if there exists
a point $p\in\bar{S}$ such that $\text{sing}_{p}(\nu)\neq0$.
\end{defn}
Let $S$ be a complex manifold and $f:X\to S$ be a family of smooth
projective varieties over $S$. Let $\mathcal{H}$ be the variation
of pure Hodge structure of weight $-1$ over $S$ with integral structure
$\HH_{\ZZ}=R^{2p-1}f_{*}\ZZ(p)$. Then, an element $w\in J^{p}(X)\,(=J^{0}(H^{2p-1}(X,\ZZ(p))))$
defines a normal function $\nu_{w}:S\to J(\mathcal{H})$ by the rule
\begin{equation}\label{eqn nf1} 
\nu_w(s)=i_s^*(w) \\
\end{equation} where $i_{s}$ denote inclusion of the fiber $X_{s}=f^{-1}(s)$ into
$X$. More generally, let $H_{\mathcal{D}}^{2p}(X,\ZZ(p))$ denote
the Deligne cohomology of $X$, and recall that we have a short exact
sequence \[
0\to J^{p}(X)\to H_{\mathcal{D}}^{2p}(X,\ZZ(p))\to H^{p,p}(X,\ZZ(p))\to0\]
 Call a Hodge class \[
\zeta\in H^{p,p}(X,\ZZ(p)):=H^{p,p}(X,\CC)\cap H^{2p}(X,\ZZ(p))\]
 primitive with respect to $f$ if $i_{s}^{*}(\zeta)=0$ for all $s\in S$,
and let $H_{{\rm prim}}^{p,p}(X,\ZZ(p))$ denote the group of all
such primitive Hodge classes. Then, by the functoriality of Deligne
cohomology, a choice of lifting $\tilde{\zeta}\in H_{\mathcal{D}}^{2p}(X,\ZZ(p))$
of a primitive Hodge class $\zeta$ determines a map $\nu_{\tilde{\zeta}}:S\to J(\mathcal{H})$.
A short calculation (cf. \cite[Ch. 10]{CMP}) shows that $\nu_{\tilde{\zeta}}$
is a (horizontal) normal function over $S$. Furthermore, in the algebraic
setting (i.e. $X,S,f$ are algebraic), $\nu_{\tilde{\zeta}}$ is an
admissible normal function \cite{S1}. Let $\ANF(S,\mathcal{H})$
denote the group of admissible normal functions with underlying variation
of Hodge structure $\mathcal{H}$. By abuse of notation, let $J^{p}(X)\subset\ANF(S,\mathcal{H})$
denote the image of the intermediate Jacobian $J^{p}(X)$ in $\ANF(S,\mathcal{H})$
under the map $w\mapsto\nu_{w}$. Then, since any two lifts $\tilde{\zeta}$
of $\zeta$ to Deligne cohomology differ by an element of the intermediate
Jacobian $J^{p}(X)$, it follows that we have a well-defined map \begin{equation}\label{eqn nf2}
AJ: H^{p,p}_{prim}(X,\ZZ(p)) \to \ANF (S,\H)/J^p(X). \\
\end{equation}
\begin{rem}
We are able to drop the notation $NF(S,\H)_{\bar{S}}^{ad}$ used in
$\S2$, because in the global algebraic case it can be shown that
admissibility is independent of the choice of compactification $\bar{S}$.
\end{rem}

\subsection{The Main Theorem}

Returning to the program of Griffiths and Green, let $X$ be a smooth
projective variety of dimension $2m$ and $L\to X$ be a very ample
line bundle. Let $\bar{P}=|L|$ and \begin{equation}\label{eqn incidence}
\X = \left\{ (x,s)\in X\times \bar{P} \mid s(x)=0 \right\} \\
\end{equation}be the incidence variety associated to the pair $(X,L)$. Let $\pi:\mathcal{X}\to\bar{P}$
denote projection on the second factor, and let $\hat{X}\subset\bar{P}$
denote the dual variety of $X$ (i.e. the points $s\in\bar{P}$ such
that $X_{s}=\pi^{-1}(s)$ is singular). Let $\mathcal{H}$ be the
variation of Hodge structure of weight $-1$ over $P=\bar{P}-\hat{X}$
attached to the local system $R^{2m-1}\pi_{*}^{sm}\ZZ(m)$.

For a pair $(X,L)$ as above, an integral Hodge class $\zeta$ of
type $(m,m)$ on $X$ is primitive with respect to $\pi^{sm}$ if
and only if it is primitive in the usual sense of being annihilated
by cup product with $c_{1}(L)$. Let $H_{{\rm prim}}^{m,m}(X,\ZZ(m))$
denote the group of all such primitive Hodge classes, and note that
$H_{{\rm prim}}^{m,m}(X,\ZZ(m))$ is unchanged upon replacing $L$
by $L^{\otimes d}$ for $d>0$. Given $\zeta\in H_{{\rm prim}}^{m,m}(X,\ZZ(m))$
let \[
\nu_{\zeta}=\AJ(\zeta)\in\ANF(P,\mathcal{H})/J^{m}(X)\]
 be the associated normal function (\ref{eqn nf2}).
\begin{lem}
If $\nu_{w}:P\to J(\mathcal{H})$ is the normal function (\ref{eqn nf1})
associated to an element $w\in J^{m}(X)$ then $\sing_{p}(\nu_{w})=0$
at every point $p\in\hat{X}$. 
\end{lem}
Accordingly, for any point $p\in\hat{X}$ we have a well defined map
\[
\overline{\text{sing}_{p}}:\ANF(P,\mathcal{H})/J^{m}(X)\to(R^{1}j_{*}\HH_{\QQ})_{p}\]
 which sends the element $[\nu]\in\ANF(P,\mathcal{H})/J^{m}(X)$ to
$\sing_{p}(\nu)$. In keeping with our prior definition, we say that
$\nu_{\zeta}$ is singular on $\bar{P}$ if there exists a point $p\in\hat{X}$
such that $\sing_{p}(\nu)\neq0$.
\begin{conjecture}
\label{conj HCv2}\cite{GG}\cite{BFNP} Let $L$ be a very ample
line bundle on a smooth projective variety $X$ of dimension $2m$.
Then, for every non-torsion class $\zeta$ in $H_{{\rm prim}}^{m,m}(X,\ZZ(m))$
there exists an integer $d>0$ such that $\AJ(\zeta)$ is singular
on $\bar{P}=|L^{\otimes d}|$. \end{conjecture}
\begin{thm}
\label{thm GGbnfp}\cite{GG}\cite{BFNP}\cite{dCM} Conjecture \eqref{conj HCv2}
holds (for every even dimensional smooth projective variety) if and
only if the Hodge conjecture is true. 
\end{thm}
To outline the proof of Theorem \ref{thm GGbnfp}, observe that for
any point $p\in\hat{X}$, we have the diagram \begin{equation}\label{eqn dottedBETAp}\xymatrix{H^{m,m}_{prim}(X,\ZZ(m)) \ar [r]^{AJ \mspace{20mu}} \ar [d]_{\alpha_p} & \ANF(P,\H)/J^m(X) \ar [d]_{\overline{\text{sing}_p}} \\ H^{2m}(X_p,\QQ(m)) \ar @{..>} [r]^{\beta_p}_{??} & (R^1j_*\HH_{\QQ})_p } \\
\end{equation} where $\alpha_{p}:H_{prim}^{m,m}(X,\ZZ(m))\to H^{2m}(X_{p},\QQ(m))$
is the restriction map.

Suppose that there exists a map \begin{equation}
\beta_p: H^{2m}(X_p,\QQ(m)) \to (R^1j_*\HH_{\QQ})_p \\
\end{equation} which makes the diagram (\ref{eqn dottedBETAp}) commute, and that
after replacing $L$ by $L^{\otimes d}$ for some $d>0$ the restriction
of $\beta_{p}$ to the image of $\alpha_{p}$ is injective. Then,
existence of a point $p\in\hat{X}$ such that $\sing_{p}(\nu_{\zeta})\neq0$
implies that the Hodge class $\zeta$ restricts non-trivially to $X_{p}$.
Now recall that by Poincar?duality and the Hodge-Riemann bilinear
relations, the Hodge conjecture for a smooth projective variety $Y$
is equivalent to the statement that for every rational $(q,q)$ class
on $Y$ there exists an algebraic cycle $W$ of dimension $2q$ on
$Y$ such that $\gamma\cup[W]\neq0$. 

Let $f:\tilde{X}_{p}\to X_{p}$ be a resolution of singularities of
$X_{p}$ and $g=i\circ f$ where $i:X_{p}\to X$ is the inclusion
map. By a weight argument $g^{*}(\zeta)\neq0$, and so there exists
a class $\xi\in\mathit{Hg}^{m-1}(\tilde{X}_{p})$ with $\xi\cup\zeta\neq0$.
Embedding $\tilde{X}_{p}$ in some projective space, and inducing
on \emph{even} dimension, we can assume that the Hodge conjecture
holds for a general hyperplane section $\mathcal{Y}\overset{\mathcal{I}}{\hookrightarrow}\tilde{X}_{p}$.
This yields an algebraic cycle $\mathcal{W}$ on $\mathcal{Y}$ with
$[\mathcal{W}]=\mathcal{I}^{*}(\xi)$. Varying $\mathcal{Y}$ in a
pencil, and using weak Lefschetz, $\mathcal{W}$ traces out%
\footnote{more precisely, one uses here a spread or Hilbert scheme argument,
cf. for example the beginning of Chap. 14 of \cite{Le1}.%
} a cycle $W=\sum_{j}\, a_{j}W_{j}$ on $\tilde{X}_{p}$ with $[W]=\xi$,
so that $g^{*}(\zeta)\cup[W]\neq0$; in particular, $\zeta\cup g_{*}[W_{j}]\neq0$
for some $j$. 

Conversely, by the work of Thomas \cite{Th}, if the Hodge conjecture
is true then the Hodge class $\zeta$ must restrict non-trivially
to some singular hyperplane section of $X$ (again for some $L^{\otimes d}$
for $d$ sufficiently large). Now one uses the injectivity of $\beta_{p}$
on $\text{im}(\alpha_{p})$ to conclude that $\nu_{\zeta}$ has a
singularity.
\begin{example}
Let $X\subset\PP^{3}$ be a smooth projective surface. For every $\zeta\in H_{prim}^{1,1}(X,\ZZ(1))$,
there is a reducible hypersurface section $X_{p}\subset X$ and component
curve $W$ of $X_{p}$ such that $\deg(\zeta|_{W})\neq0$. (Note that
$\deg(\zeta|_{X_{p}})$ is necessarily $0$.) As the reader should
check, this follows easily from Lefschetz (1,1). Moreover (writing
$d$ for the degree of $X_{p}$), $p$ is a point in a codimension
$\geq2$ substratum $S'$ of $\hat{X}\subset\PP H^{0}(\mathcal{O}(d))$
(since fibers over codim. 1 substrata are irreducible), and $\text{sing}_{q}(\nu_{\zeta})\neq0$
$\forall q\in S'$.\end{example}
\begin{rem}
\label{rem BOUND}There is a central geometric issue lurking in Conj.
\ref{conj HCv2}: \emph{if the HC holds, and $L=\mathcal{O}_{X}(1)$}
(for some projective embedding of $X$)\emph{,} \emph{is there some
minimum $d_{0}$ -- uniform in some sense -- for which $d\geq d_{0}$
implies that $\nu_{\zeta}$ is singular?} In \cite{GG} it is established
that, at best, such a $d_{0}$ could only be uniform in moduli of
the pair $(X,\zeta)$. (For example, in the case $\dim(X)=2$, $d_{0}$
is of the form $C\times|\zeta\cdot\zeta|$, for $C$ a constant. Since
the self-intersection numbers of integral classes becoming Hodge in
various Noether-Lefschetz loci increase without bound, there is certainly
not any $d_{0}$ uniform in moduli of $X$.) Whether there is some
such {}``lower bound'' of this form remains an open question in
higher dimension.
\end{rem}

\subsection{Normal functions and intersection cohomology}

The construction of the map $\beta_{p}$ depends on the decomposition
theorem of Beilinson-Bernstein-Deligne \cite{BBD} and Morihiko Saito's
theory of mixed Hodge modules \cite{S4}. As first step in this direction,
recall \cite{CKS2} that that if $\mathcal{H}$ is a variation of
pure Hodge structure of weight $k$ defined on the complement $S=\bar{S}-D$
of a normal crossing divisor on a smooth projective variety $\bar{S}$
then \[
H_{(2)}^{\ell}(S,\HH_{\RR})\cong\IH^{\ell}(\bar{S},\HH_{\RR})\]
 where the left hand side is $L^{2}$-cohomology and the right hand
side is intersection cohomology. Furthermore, via this isomorphism
$\IH^{\ell}(\bar{S},\HH_{\CC})$ inherits a canonical Hodge structure
of weight $k+\ell$. 
\begin{rem}
If $Y$ is a complex algebraic variety then $\MHM(Y)$ is the category
of mixed Hodge modules on $Y$. The category $\MHM(Y)$ comes equipped
with a functor \[
rat:\MHM(Y)\to{\rm Perv}(Y)\]
 to the category of perverse sheaves on $Y$. If $Y$ is smooth and
$\V$ is a variation of mixed Hodge structure on $Y$ then $\V[d_{Y}]$
is a mixed Hodge module on $Y$, and $rat(\V[d_{Y}])\cong\VV[d_{Y}]$
is just the underlying local system $\VV$ shifted into degree $-d_{Y}$.
\end{rem}
If $Y^{\circ}$ is a Zariski open subset of $Y$ and $\mathcal{P}$
is a perverse sheaf on $Y^{\circ}$ then \[
\IH^{\ell}(Y,\mathcal{P})=\HH^{\ell-d_{Y}}(Y,j_{!*}\mathcal{P}[d_{Y}])\]
 where $j_{!*}$ is the middle extension functor \cite{BBD} associated
to the inclusion map $j:Y^{\circ}\to Y$. Likewise, for any point
$y\in Y$, the local intersection cohomology of $\mathcal{P}$ at
$y$ is defined to be \[
\IH^{\ell}(Y,\mathcal{P})_{y}=\HH^{k-d_{Y}}(\{y\},i^{*}j_{!*}\mathcal{P}[d_{Y}])\]
 where $i:\{y\}\to Y$ is the inclusion map. If $\mathcal{P}$ underlies
a MHM, the theory of MHM puts natural MHS on these groups, which in
particular is how the pure HS on $\IH^{\ell}(\bar{S},\HH_{\CC})$
comes about.
\begin{thm}
\label{thm BFNP2.11}\cite[Thm. 2.11]{BFNP} Let $\bar{S}$ be a smooth
projective variety and $\mathcal{H}$ be a variation of pure Hodge
structure of weight $-1$ on a Zariski open subset $S\subset\bar{S}$.
Then, the group homomorphism \[
{\rm cl}:\ANF(S,\mathcal{H})\to H^{1}(S,\HH_{\QQ})\]
 factors through $\IH^{1}(\bar{S},\HH_{\QQ})$.\end{thm}
\begin{proof}
(Sketch) Let $\nu\in\ANF(S,\mathcal{H})$ be represented by an extension
\[
0\to\mathcal{H}\to\V\to\ZZ(0)\to0\]
in the category of admissible variations of mixed Hodge structure
on $S$. Let $j:S\to\bar{S}$ be the inclusion map. Then, because
$\V$ has only two non-trivial weight graded quotients which are adjacent,
it follows by \cite[Lemma 2.18]{BFNP} that \[
0\to j_{!*}\mathcal{H}[d_{S}]\to j_{!*}\mathcal{V}[d_{S}]\to\QQ(0)[d_{S}]\to0\]
 is exact in $\MHM(\bar{S})$.\end{proof}
\begin{rem}
In this particular context, $j_{!*}\mathcal{V}[d_{S}]$ can be described
as the unique prolongation of $\mathcal{V}[d_{S}]$ to $\bar{S}$
with no non-trivial sub or quotient object supported on the essential
image of the functor $i:\MHM(Z)\to\MHM(\bar{S})$ where $Z=\bar{S}-S$
and $i:Z\to\bar{S}$ is the inclusion map.
\end{rem}
In the local case of an admissible normal function on a product of
punctured polydisks $(\Delta^{*})^{r}$ with unipotent monodromy,
the fact that $\sing_{0}(\nu)$ (where $0$ is the origin of $\Delta^{r}\supseteq(\Delta^{*})^{r}$)
factors through the local intersection cohomology groups can be seen
as follows: Such a normal function $\nu$ gives a short exact sequence
of local systems \[
0\to\HH_{\QQ}\to\VV_{\QQ}\to\QQ(0)\to0\]
 over $(\Delta^{*})^{r}$. Fix a reference fiber $V_{\QQ}$ of $\VV_{\QQ}$
and let $N_{j}\in\Hom(V_{\QQ},V_{\QQ})$ denote the monodromy logarithm
of $\VV_{\QQ}$ about the $j^{\text{th}}$ punctured disk. Then \cite{CKS2},
we get a complex of finite dimensional vector spaces \[
B^{p}(V_{\QQ})=\bigoplus_{i_{1}<i_{2}<\cdots<i_{p}}\, N_{i_{1}}N_{i_{2}}\cdots N_{i_{p}}(V_{\QQ})\]
 with differential $d$ which acts on the summands of $B^{p}(V_{\QQ})$
by the rule \[
N_{i_{1}}\cdots\hat{N}_{i_{\ell}}\cdots N_{i_{p+1}}(V_{\QQ})\stackrel{(-1)^{\ell-1}N_{i_{\ell}}}{\to}N_{i_{1}}\cdots N_{i_{\ell}}\cdots N_{i_{p+1}}(V_{\QQ})\]
(and taking the sum over all insertions). Let $B^{*}(H_{\QQ})$ and
$B^{*}(\QQ(0))$ denote the analogous complexes attached to the local
systems $\HH_{\QQ}$ and $\QQ(0)$. By \cite{GGM}, the cohomology
of the complex $B^{*}(H_{\QQ})$ computes the local intersection cohomology
of $\HH_{\QQ}$. In particular, since the complexes $B^{*}(\QQ(0))$
and $B^{*}(H_{\QQ})$ sit inside the standard Koszul complexes which
compute the ordinary cohomology of $\QQ(0)$ and $H_{\QQ}$, in order
show that $\sing_{0}$ factors through $\IH^{1}(\HH_{\QQ})$ it is
sufficient to show that $\pd{\rm cl}(\nu)\in H^{1}((\Delta^{*})^{r},\HH_{\QQ})$
is representable by an element of $B^{1}(H_{\QQ})$. Indeed, let $v$
be an element of $V_{\QQ}$ which maps to $1\in\QQ(0)$. Then, \[
\pd\,{\rm cl}(\nu)=\pd1=[(N_{1}(v),\cdots,N_{r}(v))]\]
 By admissibility and the short length of the weight filtration, for
each $j$ there exists an element $h_{j}\in H_{\QQ}$ such that $N_{j}(h_{j})=N_{j}(v)$,
which is exactly the condition that \[
(N_{1}(v),\dots,N_{r}(v))\in B^{1}(V_{\QQ}).\]

\begin{thm}
\cite[Thm. 2.11]{BFNP} Under the hypothesis of Theorem \eqref{thm BFNP2.11},
for any point $p\in\bar{S}$ the group homomorphism $\sing_{p}:\ANF(S,\mathcal{H})\to(R^{1}j_{*}\HH_{\QQ})_{p}$
factors through the local intersection cohomology group $\IH^{1}(\HH_{\QQ})_{p}$.
\end{thm}
To continue, we need to pass from Deligne cohomology to absolute Hodge
cohomology. Recall that $\MHM({\rm Spec}(\CC))$ is the category MHS
of graded-polarizable $\QQ$ mixed Hodge structures. Let $\QQ(p)$
denote the Tate object of type $(-p,-p)$ in MHS and $\QQ_{Y}(p)=a_{Y}^{*}\QQ(p)$
where $a_{Y}:Y\to{\rm Spec}(\CC)$ is the structure morphism. Let
$\QQ_{Y}=\QQ_{Y}(0)$.
\begin{defn}
Let $M$ be an object of $\MHM(Y)$. Then, \[
H_{\AH}^{n}(Y,M)=Hom_{\DbMHM}(\QQ_{Y},M[n])\]
 is the absolute Hodge cohomology of $M$.
\end{defn}
The functor $\text{rat}:\MHM(Y)\to{\rm Perv}(Y)$ induces a {}``cycle
class map'' \[
\text{rat}:H_{\AH}^{n}(Y,M)\to\HH^{n}(Y,rat(M))\]
 from the absolute Hodge cohomology of $M$ to the hypercohomology
of $rat(M)$. In the case where $Y$ is smooth and projective, $H_{\AH}^{2p}(Y,\QQ_{Y}(p))$
is the Deligne cohomology group $H_{\mathcal{D}}^{2p}(Y,\QQ(p))$
and $\text{rat}$ is the cycle class map on Deligne cohomology.
\begin{defn}
Let $\bar{S}$ be a smooth projective variety and $\V$ be an admissible
variation of mixed Hodge structure on a Zariski open subset $S$ of
$\bar{S}$. Then, \begin{eqnarray*}
\IH_{\AH}^{n}(\bar{S},\mathcal{V}) & = & \Hom_{\DbMHM(\bar{S})}(\QQ_{\bar{S}}[d_{S}-n],j_{!*}\V[d_{S}])\\
\IH_{\AH}^{n}(\bar{S},\mathcal{V})_{s} & = & \Hom_{\DbMHS}(\QQ[d_{S}-n],i^{*}j_{!*}\V[d_{S}])\end{eqnarray*}
 where $j:S\to\bar{S}$ and $i:\{s\}\to\bar{S}$ are inclusion maps. 
\end{defn}
The following lemma links absolute Hodge cohomology and admissible
normal functions:
\begin{lem}
\cite[Prop. 3.3]{BFNP} Let $\mathcal{H}$ be a variation of pure
Hodge structure of weight $-1$ defined on a Zariski open subset $S$
of a smooth projective variety $\bar{S}$. Then, $\IH_{\AH}^{1}(\bar{S},\H)\cong\ANF(S,\mathcal{H})\otimes\QQ$.
\end{lem}

\subsection{Completion of diagram (\ref{eqn dottedBETAp})}

Let $f:X\to Y$ be a projective morphism between smooth algebraic
varieties. Then, by the work of Morihiko Saito \cite{S4}, there is
a direct sum decomposition \begin{equation}\label{eqn mhm1} 
f_* \QQ_X[d_X] = \bigoplus_i H^i\left( f_* \QQ_X  [d_X] \right) [-i] \\
\end{equation} in $\MHM(Y)$. Furthermore, each summand $H^{i}(f_{*}\QQ_{X}[d_{X}])$
is pure of weight $d_{X}+i$ and admits a decomposition according
to codimension of support: \begin{equation}\label{eqn mhm2} 
H^i \left( f_* \QQ_X [d_X] \right) [-i] = \oplus_j E_{ij}[-i], \\
\end{equation} i.e. $E_{ij}[-i]$ is a sum of Hodge modules supported on codimension
$j$ subvarieties of $Y$. Accordingly, we have a system of projection
operators (inserting arbitrary twists) \begin{eqnarray*}
 & \hphantom{a} & \oplus\,\Pi_{ij}:H_{\AH}^{n}(X,\QQ(\ell)[d_{X}])\stackrel{\cong}{\to}\oplus_{ij}\, H_{\AH}^{n-i}(Y,E_{ij}(\ell))\\
 & \hphantom{a} & \oplus\,\Pi_{ij}:H_{\AH}^{n}(X_{p},\QQ(\ell)[d_{X}])\stackrel{\cong}{\to}\oplus_{ij}\, H_{\AH}^{n-i}(Y,\iota^{*}E_{ij}(\ell))\\
 & \hphantom{a} & \oplus\,\Pi_{ij}:\HH^{n}(X,rat(\QQ(\ell)[d_{X}]))\stackrel{\cong}{\to}\oplus_{ij}\,\HH^{n-i}(Y,rat(E_{ij}(\ell)))\\
 & \hphantom{a} & \oplus\,\Pi_{ij}:\HH^{n}(X_{p},rat(\QQ(\ell)[d_{X}]))\stackrel{\cong}{\to}\oplus_{ij}\,\HH^{n-i}(Y,\iota^{*}rat(E_{ij}(\ell)))\end{eqnarray*}
 where $p\in Y$ and $\iota:\{p\}\to Y$ is the inclusion map.
\begin{lem}
\cite[Eqn. 4.12]{BFNP} Let $\mathcal{H}^{q}=R^{q}f_{*}^{sm}\QQ_{X}$
and recall that we have a decomposition \[
\H^{2k-1}=\H_{van}^{2k-1}\oplus\H_{fix}^{2k-1}\]
 where $\H_{fix}^{2k-1}$ is constant and $\H_{van}^{2k-1}$ has no
global sections. For any point $p\in Y$, we have a commutative diagram
\begin{equation}\label{eqn BFNPdiagram}\xymatrix{H^{2k}_{\AH}(X,\QQ(k)) \ar [d]^{i^*} \ar [r]_{\Pi \mspace{40mu}} & \ANF(Y^{sm},\H_{van}^{2k-1}(k)) \ar [d]^{i^*} \\ H^{2k}(X_p,\QQ(k)) \ar [r]_{\Pi} & \IH^1(\H^{2k-1}(k))_p } \\
\end{equation} where $Y^{sm}$ is the largest Zariski open set over which $f$ is
smooth and $\Pi$ is induced by $\Pi_{r0}$ for $r=2k-1-d_{X}+d_{Y}$.
\end{lem}
We now return to setting of Conjecture \ref{conj HCv2}: $X$ is a
smooth projective variety of dimension $2m$, $L$ is a very ample
line bundle on $X$ and $\mathcal{X}$ is the associated incidence
variety (\ref{eqn incidence}), with projections $\pi:\mathcal{X}\to\bar{P}$
and ${\rm pr}:\mathcal{X}\to X$. Then, we have the following {}``Perverse
weak Lefschetz theorem'':
\begin{thm}
\label{thm pervWL}\cite[Thm. 5.1]{BFNP} Let $\mathcal{X}$ be the
incidence variety associated to the pair $(X,L)$ and $\pi_{*}\QQ_{\mathcal{X}}=\oplus_{ij}\, E_{ij}$
in accord with (\ref{eqn mhm1}) and (\ref{eqn mhm2}). Then, \end{thm}
\begin{itemize}
\item $E_{ij}=0$ unless $i\cdot j=0$.
\item $E_{i0}=H^{i}(X,\QQ_{X}[2m-1])\otimes\QQ_{\bar{P}}[d_{\bar{P}}]$.
for $i<0$.
\end{itemize}
Note that by hard Lefschetz, $E_{ij}\cong E_{-i,j}(-i)$ \cite{S4}.

To continue, recall that given a Lefschetz pencil $\Lambda\subset\bar{P}$
of hyperplane sections of $X$, we have an associated system of vanishing
cycles $\{\delta_{p}\}_{p\in\Lambda\cap\hat{X}}\subset H^{2m-1}(X_{t},\QQ)$
on the cohomology of the smooth hyperplane sections $X_{t}$ of $X$
with respect to $\Lambda$. As one would expect, the vanishing cycles
of $\Lambda$ are \emph{non-vanishing} if for some (hence all) $p\in\Lambda\cap\hat{X}$,
$\delta_{p}\neq0$ (in $H^{2m-1}(X_{t},\QQ)$). Furthermore, this
property depends only on $L$ and not the particular choice of Lefschetz
pencil $\Lambda$. This property can always be arranged by replacing
$L$ by $L^{\otimes d}$ for some $d>0$.
\begin{thm}
\label{thm vanishing-cycles} If all vanishing cycles are non-vanishing
then $E_{01}=0$. Otherwise, $E_{01}$ is supported on a dense open
subset of $\hat{X}$.
\end{thm}
Using the Theorems \ref{thm pervWL} and \ref{thm vanishing-cycles},
we now prove that the following diagram commutes:

\begin{equation}\label{eqn PRVdiagram}\xymatrix{H^{2m}_{\D}(X,\ZZ(m))_{prim} \ar [r]_{AJ \mspace{20mu}} \ar [d]_{\text{pr}^*} & \ANF(P,\H)/J^m(X) \ar [d]^{\otimes \QQ} \\ H^{2m}_{\AH}(\X,\QQ(m)) \ar [r]_{\Pi \mspace{20mu}} & \ANF(P,\H_{van})\otimes\QQ . } \\
\end{equation}

where $H_{\D}^{2m}(X,\ZZ(m))_{prim}$ is the subgroup of $H_{\D}^{2m}(X,\ZZ(m))$
whose elements project to primitive Hodge classes in $H^{2m}(X,\ZZ(m))$,
and $\Pi$ is induced by $\Pi_{00}$ together with projection onto
$\H_{van}$. Indeed, by the decomposition theorem \begin{eqnarray*}
H_{\AH}^{2m}(\mathcal{X},\QQ(m)) & = & H_{\AH}^{1-d_{\bar{P}}}(\mathcal{X},\QQ(m)[2m+d_{\bar{P}}-1])\\
 & = & \bigoplus\, H_{\AH}^{1-d_{\bar{P}}}(\bar{P},E_{ij}(m)[-i]).\end{eqnarray*}

Let $\tilde{\zeta}\in H_{\D}^{2m}(X,\ZZ(m))$ be a primitive Deligne
class and $\omega=\oplus_{ij}\,\omega_{ij}$ denote the component
of $\omega={\rm pr}^{*}(\tilde{\zeta})$ with respect to $E_{ij}(m)[-i]$
in accord with the previous equation. Then, in order to prove the
commutativity of (\ref{eqn PRVdiagram}) it is sufficient to show
that $(\omega)_{q}=(\omega_{00})_{q}$ for all $q\in P$. By Theorem
\ref{thm pervWL}, we know that $\omega_{ij}=0$ unless $ij=0$. Furthermore,
by \cite[Lemma 5.5]{BFNP}, $(\omega_{0j})_{q}=0$ for $j>1$. Likewise,
by Theorem \ref{thm vanishing-cycles}, $(\omega_{01})_{q}=0$ for
$q\in P$ since $E_{01}$ is supported on $\hat{X}$.

Thus, in order to prove the commutativity of (\ref{eqn PRVdiagram}),
it is sufficient to show that $(\omega_{i0})_{q}=0$ for $i>0$. However,
as a consequence of the second part of Theorem \ref{thm pervWL},
$E_{i0}(m)=K[d_{\bar{P}}]$ where $K$ is a constant variation of
Hodge structure on $\bar{P}$; and hence \begin{eqnarray*}
H^{1-d_{\bar{P}}}(\mathcal{X},E_{i0}(m)[-i]) & = & Ext_{\DbMHM(\bar{P})}^{1-d_{\bar{P}}}(\QQ_{\bar{P}},K[d_{\bar{P}}-i])\\
 & = & Ext_{\DbMHM(\bar{P})}^{1-i}(\QQ_{\bar{P}},K).\end{eqnarray*}
 Therefore, $(\omega_{i0})_{q}=0$ for $i>1$ while $(\omega_{10})_{q}$
corresponds to an element of $Hom(\QQ(0),K_{q})$ where $K$ is the
constant variation of Hodge structure with fiber $H^{2m}(X_{q},\QQ(m))$
over $q\in P$. It therefore follows from the fact that $\tilde{\zeta}$
is primitive that $(\omega_{10})_{q}=0$. Splicing diagram (\ref{eqn PRVdiagram})
together with (\ref{eqn BFNPdiagram}) (and replacing $f:X\to Y$
by $\pi:\mathcal{X}\to\bar{P}$, etc.) now gives the diagram (\ref{eqn dottedBETAp}).
\begin{rem}
The effect of passage from $\H$ to $\H^{{\rm van}}$ in the above
constructions is to annihilate $J^{m}(X)\subseteq H_{\D}^{2m}(X,\ZZ(m))_{prim}$.
Therefore, in (\ref{eqn PRVdiagram}) we can replace $H_{\D}^{2m}(X,\ZZ(m))_{prim}^{{\rm }}$
by $H_{{\rm prim}}^{m,m}(X,\ZZ(m))$.
\end{rem}
Finally, if all the vanishing cycles are non-vanishing, $E_{01}=0$.
Using this fact, we then get the injectivity of $\beta_{p}$ on the
image of $\alpha_{p}$.

Returning to the beginning of this section, we now see that although
extending normal functions along Lefschetz pencils is insufficient
to prove the Hodge conjecture for higher codimension cycles, the Hodge
conjecture is equivalent to a statement about the behavior of normal
functions on the complement of the dual variety of $X$ inside $|L|$
for $L\gg0$.

\section{Zeroes of Normal Functions}

\subsection{Algebraicity of the zero locus}

Some of the deepest evidence to date in support of the Hodge conjecture
is the following result of Cattani, Deligne and Kaplan on the algebraicity
of the Hodge locus:
\begin{thm}
\label{thm CDK}\cite{CDK} Let $\H$ be a variation of pure Hodge
structure of weight $0$ over a smooth complex algebraic variety $S$.
Let $\alpha_{s_{o}}$ be an integral Hodge class of type $(0,0)$
on the fiber of $\H$ at $s_{o}$. Let $U$ be a simply connected
open subset of $S$ containing $s_{o}$ and $\alpha$ be the section
of $\HH_{\ZZ}$ over $U$ defined by parallel translation of $\alpha_{s_{o}}$.
Let $T$ be the locus of points in $U$ such that $\alpha(s)$ is
of type $(0,0)$ on the fiber of $\H$ over $s$. Then, the analytic
germ of $T$ at $p$ is the restriction of a complex algebraic subvariety
of $S$.
\end{thm}
More precisely, as explained in the introduction of \cite{CDK}, in
the case where $\H$ arises from the cohomology of a family of smooth
projective varieties $f:X\to S$, the algebraicity of the germ of
$T$ follows from the Hodge conjecture. A natural analogue of this
result for normal functions is:
\begin{thm}
\label{conj CalgZL} Let $S$ be a smooth complex algebraic variety,
and $\nu:S\to J(\H)$ be an admissible normal function, where $\H$
is a variation of pure Hodge structure of weight $-1$. Then, the
zero locus \[
\mathcal{Z}(\nu)=\{\, s\in S\mid\nu(s)=0\,\}\]
is a complex algebraic subvariety of $S$.
\end{thm}
This theorem was still a conjecture when the present article was submitted,
and has just been proved by the second author in work with P. Brosnan
\cite{BP3}. It is of particular relevance to the Hodge conjecture,
due to the following relationship between the algebraicity of $\mathcal{Z}(\nu)$
and the existence of singularities of normal functions. Say $\dim(X)=2m$,
and let $(X,L,\zeta)$ be a triple consisting of a smooth complex
projective variety $X$, a very ample line bundle $L$ on $X$ and
a primitive integral Hodge class $\zeta$ of type $(m,m)$. Let $\nu_{\zeta}$
(assumed nonzero) be the associated normal function on the complement
of the dual variety $\hat{X}$ constructed in $\S3$, and $\mathcal{Z}$
be its zero locus. Then, assuming that $\mathcal{Z}$ is algebraic
and positive dimensional, the second author conjectured that $\nu$
should have singularities along the intersection of the closure of
$\mathcal{Z}$ with $\hat{X}$.
\begin{thm}
\cite{Sl1} Let $(X,L,\zeta)$ be a triple as above, and assume that
$L$ is sufficiently ample that, given any point $p\in\hat{X}$, the
restriction of $\beta_{p}$ to the image of $\alpha_{p}$ in diagram
\emph{(}\ref{eqn dottedBETAp}\emph{)} is injective. Suppose that
$\mathcal{Z}$ contains an algebraic curve. Then, $\nu_{\zeta}$ has
a non-torsion singularity at some point of the intersection of the
closure of this curve with $\hat{X}$.\end{thm}
\begin{proof}
(Sketch) Let $C$ be the normalization of the closure of the curve
in $\mathcal{Z}$. Let $\mathcal{X}\to\bar{P}$ be the universal family
of hyperplane sections of $X$ over $\bar{P}=|L|$ and $W$ be the
pullback of $X$ to $C$. Let $\pi:W\to C$ be the projection map,
and $U$ be set of points $c\in C$ such that $\pi^{-1}(c)$ is smooth,
and $W_{U}=\pi^{-1}(U)$. Then, via the Leray spectral sequence for
$\pi$, it follows that restriction of $\zeta$ to $W_{U}$ is zero
because $U\subseteq\mathcal{Z}$ and $\zeta$ is primitive. On the
other hand, since $W\twoheadrightarrow X$ is finite, $\zeta$ must
restrict (pull back) non-trivially to $W$, and hence $\zeta$ must
restrict non-trivially to the fiber $\pi^{-1}(c)$ for some point
$c\in C$ in the complement of $U$.
\end{proof}
Unfortunately, crude estimates for the expected dimension of the zero
locus $\mathcal{Z}$ arising in this context appear to be negative.
For instance, take $X$ to be an abelian surface in the following:
\begin{thm}
\label{thm dimESTIMATE}Let $X$ be a surface and $L=\mathcal{O}_{X}(D)$
be an ample line bundle on $X$. Then, for $n$ sufficiently large,
the expected dimension of the zero locus of the normal function $\nu_{\zeta}$
attached to the triple $(X,L^{\otimes n},\zeta)$ as above is \[
h^{2,0}-h^{1,0}-n(D.K)-1\]
 where $K$ is the canonical divisor of $X$.\end{thm}
\begin{proof}
(Sketch) Since Griffiths's horizontality is trivial in this setting,
computing the expected dimension boils down to computing the dimension
of $|L|$ and genus of a smooth hyperplane section of $X$ with respect
to $L$.\end{proof}
\begin{rem}
In Theorem \ref{thm dimESTIMATE}, we construct $\nu_{\zeta}$ from
a choice of lift to Deligne cohomology (or an algebraic cycle) to
get an element of $ANF(P,\H)$. But this is disingenuous, since we
are starting with a Hodge class. It is more consistent to work with
$\nu_{\zeta}\in ANF(P,\H)/J^{1}(X)$ (as in equation \eqref{eqn nf2}),
and then the dimension estimate improves by $\dim(J^{1}(X))=h^{1,0}$
to $h^{2,0}-n(D.K)-1$. Notice that this salvages at least the abelian
surface case (though it is still a crude estimate). For surfaces of
general type, one is still in trouble without more information, like
the constant $C$ in Remark \eqref{rem BOUND}.
\end{rem}
We will not attempt to describe the proof of Theorem \ref{conj CalgZL}
in general, but we will explain the following special case:
\begin{thm}
\label{thm BP2}\cite{BP2} Let $S$ be a smooth complex algebraic
variety which admits a projective completion $\bar{S}$ such that
$D=\bar{S}-S$ is a smooth divisor. Let $\H$ be a variation of pure
Hodge structure of weight $-1$ on $S$ and $\nu:S\to J(\H)$ be an
admissible normal function. Then, the zero locus $\mathcal{Z}$ of
$\nu$ is an complex algebraic subvariety of $S$.\end{thm}
\begin{rem}
This result was obtained contemporaneously by Morihiko Saito in \cite{S5}.

In analogy with the proof of Theorem \ref{thm CDK} on the algebraicity
of the Hodge locus, which depends heavily on the several variable
$SL_{2}$-orbit theorem for nilpotent orbits of pure Hodge structure
\cite{CKS1}, the proof of Theorem \ref{conj CalgZL} depends upon
the corresponding result for nilpotent orbits of mixed Hodge structure.
For simplicity of exposition, we will now review the $1$-variable
$SL_{2}$-orbit theorem in the pure case (which is due to Schmid \cite{Sc})
and a version of the $SL_{2}$-orbit theorem in the mixed case \cite{Pe2}
sufficient to prove Theorem \ref{thm BP2}. For the proof of Theorem
\ref{conj CalgZL}, we need the full strength of the several variable
$SL_{2}$-orbit theorem of Kato, Nakayama and Usui \cite{KNU}.
\end{rem}

\subsection{The classical nilpotent and $SL_{2}$ orbit theorems}

To outline the proof of Theorem \ref{thm BP2}, we now recall the
theory of degenerations of Hodge structure: Let $\H$ be a variation
of pure Hodge structure of weight $k$ over a simply connected complex
manifold $S$. Then, via parallel translation back to a fixed reference
fiber $H=H_{s_{o}}$ we obtain a period map \begin{equation}\label{eqn periodmap}\varphi : S \to \mathcal{D} \\
\end{equation} where $\D$ is Griffiths's classifying space of pure Hodge structures
on $H$ with fixed Hodge numbers $\{h^{p,k-p}\}$ which are polarized
by the bilinear form $Q$ of $H$. The set $\D$ is a complex manifold
upon which the Lie group \[
G_{\RR}={\rm Aut}_{\RR}(Q)\]
acts transitively by biholomorphisms, and hence $\D\cong G_{\RR}/G_{\RR}^{F_{o}}$
where $G_{\RR}^{F_{o}}$ is the isotropy group of $F_{o}\in\D$. The
compact dual of $\D$ is the complex manifold \[
\check{\D}\cong G_{\mathbb{C}}/G_{\mathbb{C}}^{F_{o}}\]
where $F_{o}$ is any point in $\D$. (In general, $F=F^{\bullet}$
denotes a Hodge filtration.) If $S$ is not simply connected, then
the period map (\ref{eqn periodmap}) is replaced by \begin{equation}\label{eqn pureperiodmap}\varphi : S \to \Gamma \backslash \D \\
\end{equation} where $\Gamma$ is the monodromy group of $\HH\to S$ acting on the
reference fiber $H$.

For variations of Hodge structure of geometric origin, $S$ will typically
be a Zariski open subset of a smooth projective variety $\bar{S}$.
By Hironaka's resolution of singularities theorem, we can assume $D=\bar{S}-S$
to be a divisor with normal crossings. The period map (\ref{eqn pureperiodmap})
will then have singularities at the points of $D$ about which $\HH$
has non-trivial local monodromy. A precise local description of the
singularities of the period map of a variation of Hodge structure
was obtained by Schmid \cite{Sc}: Let $\varphi:(\Delta^{*})^{r}\to\Gamma\backslash\D$
be the period map of variation of pure polarized Hodge structure over
the product of punctured disks. First, one knows that $\varphi$ is
locally liftable with quasi-unipotent monodromy. After passage to
a finite cover, we therefore obtain a commutative diagram \begin{equation}
\xymatrix{U^r \ar [r]_{F} \ar [d] & \D \ar [d] \\ (\Delta^*)^r \ar [r]_{\varphi} & \Gamma \backslash \D } \\
\end{equation} where $U^{r}$ is the $r$-fold product of upper half-planes and
$U^{r}\to(\Delta^{*})^{r}$ is the covering map \[
s_{j}=e^{2\pi iz_{j}},\qquad j=1,\dots,r\]
with respect to the standard Euclidean coordinates $(z_{1},\dots,z_{r})$
on $U^{r}\subset\CC^{r}$ and $(s_{1},\dots,s_{r})$ on $(\Delta^{*})^{r}\subset\CC^{r}$.

Let $T_{j}=e^{N_{j}}$ denote the monodromy of $\H$ about $s_{j}=0$.
Then, \[
\psi(z_{1},\dots,z_{r})=e^{-\sum_{j}\, z_{j}N_{j}}.F(z_{1},\dots,z_{r})\]
is a holomorphic map from $U^{r}$ into $\check{\D}$ which is invariant
under the transformation $z_{j}\mapsto z_{j}+1$ for each $j$, and
hence drops to a map $(\Delta^{*})^{r}\to\check{\D}$ which we continue
to denote by $\psi$.
\begin{defn}
Let $\D$ be a classifying space of pure Hodge structure with associated
Lie group $G_{\RR}$. Let $\mathfrak{g}_{\RR}$ be the Lie algebra
of $G_{\RR}$. Then, a holomorphic, horizontal map $\theta:\CC^{r}\to\check{\D}$
is a nilpotent orbit if 

(a) there is a constant $\alpha>0$ such that $\theta(z_{1},\dots,z_{r})\in\D$
if ${\rm Im}(z_{j})>\alpha$ $\forall\, j$; and

(b) there exist commuting nilpotent endomorphisms $N_{1},\dots,N_{r}\in\mathfrak{g}_{\RR}$
and a point $F\in\check{\D}$ such that $\theta(z_{1},\dots,z_{r})=e^{\sum_{j}\, z_{j}N_{j}}.F$. \end{defn}
\begin{thm}
\emph{(Nilpotent Orbit Theorem, \cite{Sc})} Let $\varphi:(\Delta^{*})^{r}\to\Gamma\backslash\D$
be the period map of a variation of pure Hodge structure of weight
$k$ with unipotent monodromy. Let $d_{\D}$ be a $G_{\RR}$-invariant
distance on $\D$. Then, 

(a) $F_{\infty}=\lim_{s\to0}\,\psi(s)$ exists, i.e. $\psi(s)$ extends
to a map $\Delta^{r}\to\check{\D}$; 

(b) $\theta(z_{1},\dots,z_{r})=e^{\sum_{j}\, z_{j}N_{j}}.F_{\infty}$
is a nilpotent orbit; and

(c) there exist constants $C$, $\alpha$ and $\beta_{1},\dots,\beta_{r}$
such that if ${\rm Im}(z_{j})>\alpha$ $\forall\, j$ then $\theta(z_{1},\dots,z_{r})\in\D$
and \[
d_{\D}(\theta(z_{1},\dots,z_{r}),F(z_{1},\dots,z_{r}))<C\sum_{j}\,{\rm Im}(z_{j})^{\beta_{j}}e^{-2\pi{\rm Im}(z_{j})}.\]
\end{thm}
\begin{rem}
Another way of stating part $(a)$ of the Nilpotent Orbit Theorem
is the Hodge bundles $\mathcal{F}^{p}$ of $\H_{\mathcal{O}}$ extend
to a system of holomorphic subbundles of the canonical extension of
$\H_{\mathcal{O}}$. Indeed, recall from $\S2.7$ that one way of
constructing a model of the canonical extension in the unipotent monodromy
case is to take a flat, multivalued frame $\{\sigma_{1},\dots,\sigma_{m}\}$
of $\HH_{\ZZ}$ and twist it to form a single valued holomorphic frame
$\{\tilde{\sigma}_{1},\dots,\tilde{\sigma}_{m}\}$ over $(\Delta^{*})^{r}$
where $\tilde{\sigma}_{j}=e^{-\frac{1}{2\pi i}\sum_{j}\,\log(s_{j})N_{j}}\sigma_{j}$,
and then declaring this twisted frame to define the canonical extension.
\end{rem}
Let $N$ be a nilpotent endomorphism of a finite dimensional vector
space over a field $k$. Then, $N$ can be put into Jordan canonical
form, and hence (by considering a Jordan block) it follows that there
is a unique, increasing filtration $\mathsf{W}(N)$ of $V$, such
that 

(a) $N(\mathsf{W}(N)_{j})\subseteq\mathsf{W}(N)_{j-2}$ and

(b) $N^{j}:Gr_{j}^{\mathsf{W}(N)}\to Gr_{-j}^{\mathsf{W}(N)}$ is
an isomorphism

for each index $j$. If $\ell$ is an integer then $(\mathsf{W}(N)[\ell])_{j}=\mathsf{W}(N)_{j+\ell}$.
\begin{thm}
\label{thm schmid}Let $\varphi:\Delta^{*}\to\Gamma\backslash\D$
be the period map of a variation of pure Hodge structure of weight
$k$ with unipotent monodromy $T=e^{N}$. Then, the limit Hodge filtration
$F_{\infty}$ of $\varphi$ pairs with the \emph{weight monodromy
filtration} $M(N):=\mathsf{W}(N)[-k]$ to define a mixed Hodge structure
relative to which $N$ is a $(-1,-1)$-morphism.\end{thm}
\begin{rem}
The limit Hodge filtration $F_{\infty}$ depends upon the choice of
local coordinate $s$, or more precisely on the value of $(ds)_{0}$.
Therefore, unless one has a preferred coordinate system (e.g. if the
field of definition matters), in order to extract geometric information
from the limit mixed Hodge structure $H_{\infty}=(F_{\infty},M(N))$
one usually has to pass to the mixed Hodge structure induced by $H_{\infty}$
on the kernel or cokernel of $N$. In particular, if $X\to\Delta$
is a semistable degeneration the the local invariant cycle theorem
asserts that we have an exact sequence \[
H^{k}(X_{0})\to H_{\infty}\stackrel{N}{\to}H_{\infty}\]
 where the map $H^{k}(X_{0})\to H_{\infty}$ is obtained by first
including the reference fiber $X_{s_{o}}$ into $X$ and then retracting
$X$ onto $X_{0}$. 
\end{rem}
The proof of Theorem \ref{thm schmid} depends upon Schmid's $SL_{2}$-orbit
theorem. Informally, this result asserts that any 1-parameter nilpotent
orbit is asymptotic to a nilpotent orbit arising from a representation
of $SL_{2}(\RR)$. In order to properly state Schmid's results we
need to discuss splittings of mixed Hodge structures.
\begin{thm}
\label{thm Ipq's} \emph{(Deligne, \cite{De1})} Let $(F,W)$ be a
mixed Hodge structure on $V$. Then, there exists a unique, functorial
bigrading \[
V_{\CC}=\bigoplus_{p,q}\, I^{p,q}\]
 such that 

(a) $F^{p}=\bigoplus_{a\geq p}\, I^{a,b}$; 

(b) $W_{k}=\bigoplus_{a+b\leq k}\, I^{a,b}$; 

(c) $\overline{I^{p,q}}=I^{q,p}\mod\bigoplus_{r<q,s<p}\, I^{r,s}$.
\end{thm}
In particular, if $(F,W)$ is a mixed Hodge structure on $V$ then
$(F,W)$ induces a mixed Hodge structure on $gl(V)\cong V\otimes V^{*}$
with bigrading \[
gl(V_{\CC})=\bigoplus_{r,s}\, gl(V)^{r,s}\]
 where $gl(V)^{r,s}$ is the subspace of $gl(V)$ which maps $I^{p,q}$
to $I^{p+r,q+s}$ for all $(p,q)$. In the case where $(F,W)$ is
graded-polarized, we have an analogous decomposition $\mathfrak{g}_{\CC}=\bigoplus_{r,s}\,\mathfrak{g}^{r,s}$
of the Lie algebra of $G_{\CC}(=Aut(V_{\CC},Q))$. For future use,
we define \begin{equation}\Lambda^{-1,-1}_{(F,W)} = \bigoplus_{r,s<0} gl(V)^{r,s} \\
\end{equation} and note that by properties $(a)$-$(c)$ of Theorem \ref{thm Ipq's}
\begin{equation}\lambda \in \Lambda_{(F,W)}^{-1,-1} \implies  I^{p,q}_{(e^{\lambda}.F,W)} = e^{\lambda}.I^{p,q}_{(F,W)}. \\
\end{equation} A mixed Hodge structure $(F,W)$ is split over $\RR$ if $\bar{I}^{p,q}=I^{q,p}$
for $(p,q)$. In general, a mixed Hodge structure $(F,W)$ is not
split over $\RR$. However, by a theorem of Deligne \cite{CKS1},
there is a functorial splitting operation \[
(F,W)\mapsto(\hat{F}_{\delta},W)=(e^{-i\delta}.F,W)\]
which assigns to any mixed Hodge structure $(F,W)$ a split mixed
Hodge structure $(\hat{F}_{\delta},W)$, such that 

(a) $\delta=\bar{\delta}$, 

(b) $\delta\in\Lambda_{(F,W)}^{-1,-1}$, and 

(c) $\delta$ commutes with all $(r,r)$-morphisms of $(F,W)$. 
\begin{rem}
$\Lambda_{(F,W)}^{-1,-1}=\Lambda_{(\hat{F}_{\delta},W)}^{-1,-1}$.
\end{rem}
A nilpotent orbit $\hat{\theta}(z)=e^{zN}.F$ is an $SL_{2}$-orbit
if there exists a group homomorphism $\rho:SL_{2}(\RR)\to G_{\RR}$
such that \[
\hat{\theta}(g.\sqrt{-1})=\rho(g).\hat{\theta}(\sqrt{-1})\]
for all $g\in SL_{2}(\RR)$. The representation $\rho$ is equivalent
to the data of an $sl_{2}$-triple $(N,H,N^{+})$ of elements in $G_{\RR}$
such that \[
[H,N]=-2N,\qquad[N^{+},N]=H,\qquad[H,N^{+}]=2N^{+}\]
We also note that, for nilpotent orbits of pure Hodge structure, the
statement that $e^{zN}.F$ is an $SL_{2}$-orbit is equivalent to
the statement that the limit mixed Hodge structure $(F,M(N))$ is
split over $\RR$ \cite{CKS1}.
\begin{thm}
\emph{($SL_{2}$-Orbit Theorem, \cite{Sc})} Let $\theta(z)=e^{zN}.F$
be a nilpotent orbit of pure Hodge structure. Then, there exists a
unique $SL_{2}$-orbit $\hat{\theta}(z)=e^{zN}.\hat{F}$ and a distinguished
real-analytic function \[
g(y):(a,\infty)\to G_{\RR}\]
(for some $a\in\RR$) such that: 

(a) $\theta(iy)=g(y).\hat{\theta}(iy)$ for $y>a$; and

(b) both $g(y)$ and $g^{-1}(y)$ have convergent series expansions
about $\infty$ of the form \[
g(y)=1+\sum_{k>0}\, g_{j}y^{-k},\qquad g^{-1}(y)=1+\sum_{k>0}\, f_{k}y^{-k}\]
 with $g_{k}$, $f_{k}\in\ker(\text{ad}\, N)^{k+1}$. 

Furthermore, the coefficients $g_{k}$ and $f_{k}$ can be expressed
in terms of universal Lie polynomials in the Hodge components of $\delta$
with respect to $(\hat{F},M(N))$ and $\text{ad}\, N^{+}$.\end{thm}
\begin{rem}
The precise meaning of the statement that $g(y)$ is a distinguished
real-analytic function, is that $g(y)$ arises in a specific way from
the solution of a system of differential equations attached to $\theta$.
\end{rem}
$\vspace{-10mm}$
\begin{rem}
If $\theta$ is a nilpotent orbit of pure Hodge structures of weight
$k$ and $\hat{\theta}=e^{zN}.\hat{F}$ is the associated $SL_{2}$-orbit
then $(\hat{F},M(N))$ is split over $\RR$. The map $(F,M(N))\mapsto(\hat{F},M(N))$
is called the $sl_{2}$-splitting of $(F,M(N))$. Furthermore, $\hat{F}=e^{-\xi}.F$
where $\xi$ is given by universal Lie polynomials in the Hodge components
of $\delta$. In this way, one obtains an $sl_{2}$-splitting $(F,W)\mapsto(\hat{F},W)$
for any mixed Hodge structure $(F,W)$.
\end{rem}

\subsection{Nilpotent and $SL_{2}$ orbit theorems in the mixed case }

In analogy to the theory of period domains for pure HS, one can form
a classifying space of graded-polarized mixed Hodge structure $\M$
with fixed Hodge numbers. Its points are the decreasing filtrations
$F$ of the reference fiber $V$ which pair with the weight filtration
$W$ to define a graded-polarized mixed Hodge structure (with the
given Hodge numbers). Given a variation of mixed Hodge structure $\mathcal{V}$
of this type over a complex manifold $S$, one obtains a period map
\[
\phi:S\to\Gamma\backslash\M.\]
$\M$ is a complex manifold upon which the Lie group $G$, consisting
of elements of $GL(V_{\CC})$ which preserve $W$ and act by real
isometries on $Gr^{W}$, acts transitively. Next, let $G_{\CC}$ denote
the Lie group consisting of elements of $GL(V_{\CC})$ which preserve
$W$ and act by \emph{complex} isometries on $Gr^{W}$. Then, in analogy
with the pure case, the {}``compact dual'' $\check{\M}$ of $\M$
is the complex manifold \[
\check{\M}\cong G_{\CC}/G_{\CC}^{F_{o}}\]
 for any base point $F_{o}\in\M$. The subgroup $G_{\RR}=G\cap GL(V_{\RR})$
acts transitively on the real-analytic submanifold $\M_{\RR}$ consisting
of points $F\in\M$ such that $(F,W)$ is split over $\RR$.
\begin{example}
Let $\M$ be the classifying space of mixed Hodge structures with
Hodge numbers $h^{1,1}=h^{0,0}=1$. Then, $\M\cong\CC$.
\end{example}
The proof of Schmid's nilpotent orbit theorem depends critically upon
the fact that the classifying space $\D$ has negative holomorphic
sectional curvature along horizontal directions \cite{GS}. Thus,
although one can formally carry out all of the constructions leading
up to the statement of the nilpotent orbit theorem in the mixed case,
in light of the previous example it follows that one can not have
negative holomorphic sectional curvature in the mixed case, and hence
there is no reason to expect an analog of Schmid's Nilpotent Orbit
Theorem in the mixed case. Indeed, for this classifying space $\M$,
the period map $\varphi(s)=\exp(s)$ gives an example of a period
map with trivial monodromy which has an essential singularity at $\infty$.
Some additional condition is clearly required, and this is where admissibility
comes in.

In the geometric case of a degeneration of pure Hodge structure, Steenbrink
\cite{St} gave an alternative construction of the limit Hodge filtration
that can be extended to variations of mixed Hodge structure of geometric
origin \cite{SZ}. More generally, given an \emph{admissible} variation
of mixed Hodge structure $\V$ over a smooth complex algebraic variety
$S\subseteq\bar{S}$ such that $D=\bar{S}-S$ is a normal crossing
divisor, and any point $p\in D$ about which $\VV$ has unipotent
local monodromy, one has an associated nilpotent orbit $(e^{\sum_{j}\, z_{j}N_{j}}.F_{\infty},W)$
with limit mixed Hodge structure $(F_{\infty},M)$ where $M$ is the
\emph{relative weight filtration} of $N=\sum_{j}\, N_{j}$ and $W$.%
\footnote{Recall \cite{SZ} that in general the relative weight filtration $M=M(N,W)$
is the unique filtration (if it exists) such that $N(M_{k})\subset M_{k-2}$
and $M$ induces the monodromy weight filtration of $N$ on each $Gr_{i}^{W}$
(centered about $i$).%
} Furthermore, one has the following {}``group theoretic'' version
of the nilpotent orbit theorem: As in the pure case, a variation of
mixed Hodge structure $\V\to(\Delta^{*})^{r}$ with unipotent monodromy
gives a holomorphic map \[
\psi:(\Delta^{*})^{r}\to\check{\M}\]
\[
\mspace{20mu}\underline{z}\,\longmapsto\, e^{-\sum z_{j}N_{j}}F(\underline{z})\,,\]
and this extends to $\Delta^{r}$ if $\V$ is admissible. Let \[
\mathfrak{q}_{\infty}=\bigoplus_{r<0}\,\mathfrak{g}^{r,s}\]
 where $\mathfrak{g}_{\CC}={\rm Lie}(G_{\CC})=\oplus_{r,s}\,\mathfrak{g}^{r,s}$
relative to the limit mixed Hodge structure $(F_{\infty},M)$. Then
$\mathfrak{q}_{\infty}$ is a nilpotent Lie subalgebra of $\mathfrak{g}_{\CC}$
which is a vector space complement to the isotropy algebra $\mathfrak{g}_{\CC}^{F_{\infty}}$
of $F_{\infty}$. Consequently, there exists an open neighborhood
$\mathcal{U}$ of zero in $\mathfrak{g}_{\CC}$ such that\[
\mathcal{U}\to\check{\M}\]
 \[
u\mapsto e^{u}.F_{\infty}\]
is a biholomorphism, and hence after shrinking $\Delta^{r}$ as necessary
we can write \[
\psi(s)=e^{\Gamma(s)}.F_{\infty}\]
relative to a unique $\mathfrak{q}_{\infty}$-valued holomorphic function
$\Gamma$ on $\Delta^{r}$ which vanishes at $0$. Recalling the construction
of $\psi$ from the lifted period map $F$, it follows that \[
F(z_{1},\dots,z_{r})=e^{\sum_{j}\, z_{j}N_{j}}e^{\Gamma(s)}.F_{\infty}.\]
This is called the \emph{local normal form} of $\V$ at $p$ and will
be used in the calculations of $\S$5.4-5.

There is also a version of Schmid's $SL_{2}$-orbit theorem for admissible
nilpotent orbits. In the case of 1-variable and weight filtrations
of short length, the is due to the second author in \cite{Pe2}. More
generally, Kato, Nakayama and Usui proved a several variable $SL_{2}$-orbit
theorem with arbitrary weight filtration in \cite{KNU}. Despite the
greater generality of \cite{KNU}, in this paper we are going to stick
with the version of the $SL_{2}$-orbit theorem from \cite{Pe2} as
it is sufficient for our needs and has the advantage that for normal
functions, mutatis mutandis, it is identical to Schmid's result.

\subsection{Outline of proof of Theorem (\ref{thm BP2})}

Let us now specialize to the case of an admissible normal function
$\nu:S\to J(\H)$ over a curve and outline the proof \cite{BP1} of
Theorem \ref{thm BP2}. Before proceeding, we do need to address one
aspect of the $SL_{2}$-orbit theorem in the mixed case. Let $\hat{\theta}=(e^{zN}.F,W)$
be an admissible nilpotent orbit with limit mixed Hodge structure
$(F,M)$ which is split over $\RR$. Then, $\hat{\theta}$ induces
an $SL_{2}$-orbit on each $Gr_{k}^{W}$, and hence a corresponding
$sl_{2}$-representation $\rho_{k}$.
\begin{defn}
Let $W$ be an increasing filtration, indexed by $\ZZ$, of a finite
dimensional vector space $V$. A \emph{grading} of W is a direct sum
decomposition $W_{k}=V_{k}\oplus W_{k-1}$ for each index $k$.
\end{defn}
In particular, a mixed Hodge structure $(F,W)$ on $V$ gives a grading
of $W$ by the rule $V_{k}=\oplus_{p+q=k}\, I^{p,q}$. Furthermore,
if the ground field has characteristic zero, a grading of $W$ is
the same thing as a semisimple endomorphism $Y$ of $V$ which acts
as multiplication by $k$ on $V_{k}$. If $(F,W)$ is a mixed Hodge
structure we let $Y_{(F,W)}$ denote the grading of $W$ which acts
on $I^{p,q}$ as multiplication by $p+q$, the \emph{Deligne grading}
of $(F,W)$.

Returning to the admissible nilpotent orbit $\hat{\theta}$ considered
above, we now have a system of representations $\rho_{k}$ on $Gr_{k}^{W}$.
To construct an $sl_{2}$-representation on the reference fiber $V$,
we need to pick a grading $Y$ of $W$. Clearly for each Hodge flag
$F(z)$ in the orbit we have the Deligne grading $Y_{(F(z),W)}$;
but we are after something more canonical. Now we also have the Deligne
grading $Y_{(\hat{F},M)}$ of $M$ associated to the $sl_{2}$-splitting
of the LMHS. In the unpublished letter \cite{De3}, Deligne observed
that:
\begin{thm}
\label{thm DELIGNEresult}There exists a unique grading $Y$ of $W$
which commutes with $Y_{(\hat{F},M)}$ and has the property that if
$(N_{0},H,N_{0}^{+})$ denote the liftings of the $sl_{2}$-triples
attached to the graded representations $\rho_{k}$ via $Y$ then $[N-N_{0},N_{0}^{+}]=0$.
\end{thm}
With this choice of $sl_{2}$-triple, and $\hat{\theta}$ an admissible
nilpotent orbit in 1-variable of the type arising from an admissible
normal function, the main theorem of \cite{Pe2} asserts that one
has a direct analog of Schmid's $SL_{2}$-orbit theorem as stated
above for $\hat{\theta}$.
\begin{rem}
More generally, given an admissible nilpotent orbit $(e^{zN}F,W)$
with relative weight filtration $M=M(N,W)$, Deligne shows that there
exists a grading $Y=Y(N,Y_{(F,M)})$ with similar properties (cf.
\cite{BP1} for details and further references).
\end{rem}
$\vspace{-5mm}$
\begin{rem}
In the case of a normal function, if we decompose $N$ according to
$\text{ad}\, Y$ we have $N=N_{0}+N_{-1}$ where $N_{-1}$ must be
either zero or a highest weight vector of weight $-1$ for the representation
of $sl_{2}(\RR)$ defined by $(N_{0},H,N_{0}^{+})$. Accordingly,
since there are no vectors of highest weight $-1$, we have $N=N_{0}$
and hence $[Y,N]=0$.
\end{rem}
The next thing that we need to recall is that if $\nu:S\to J(\H)$
is an admissible normal function which is represented by an extension
\[
0\to H\to\V\to\ZZ(0)\to0\]
in the category of admissible variations of mixed Hodge structure
on $S$ then the zero locus $\mathcal{Z}$ of $\nu$ is exactly the
set of points where the corresponding Deligne grading $Y_{(\F,\W)}$
is integral. In the case where $S\subset\bar{S}$ is a curve, in order
to prove the algebraicity of $\mathcal{Z}$, all we need to do is
show that $\mathcal{Z}$ cannot contain a sequence of points $s(m)$
which accumulate to a puncture $p\in\bar{S}-S$ unless $\nu$ is identically
zero. The first step towards the proof of Theorem \ref{thm BP2} is
the following result \cite{BP1}:
\begin{thm}
Let $\varphi:\Delta^{*}\to\Gamma\backslash\M$ denote the period map
of an admissible normal function $\nu:\Delta^{*}\to J(\mathcal{H})$
with unipotent monodromy, and $Y$ be the grading of $W$ attached
to the nilpotent orbit $\theta$ of $\varphi$ by Deligne's construction
\emph{(Theorem \ref{thm DELIGNEresult}}). Let $F:U\to\M$ denote
the lifting of $\varphi$ to the upper half-plane. Then, for ${\rm Re}(z)$
restricted to an interval of finite length, we have \[
\lim_{{\rm Im}(z)\to\infty}\, Y_{(F(z),W)}=Y\]
\end{thm}
\begin{proof}
(Sketch) Using \cite{De3}, one can prove this result in the case
where $\varphi$ is a nilpotent orbit with limit mixed Hodge structure
which is split over $\RR$. Let $z=x+iy$. In general, one writes
\[
F(z)=e^{zN}e^{\Gamma(s)}.F_{\infty}=e^{xN}e^{iyN}e^{\Gamma(s)}e^{-iyN}e^{iyN}.F_{\infty}\]
 where $e^{xN}$ is real, $e^{iyN}.F_{\infty}$ can be approximated
by an $SL_{2}$-orbit and $e^{iyN}e^{\Gamma(s)}e^{-iyN}$ decays to
$1$ very rapidly.
\end{proof}
In particular, if there exists a sequence $s(m)$ which converges
to $p$ along which $Y_{(\F,\W)}$ is integral it then follows from
the previous theorem that $Y$ is integral. An explicit computation
then shows that the equation of the zero locus near $p$ is given
by the equation \[
{\rm Ad}(e^{\Gamma(s)})Y=Y\]
 which is clearly holomorphic on a neighborhood of $p$ in $\bar{S}$.

That concludes the proof for $S$ a curve. In the case where $S$
has a compactification $\bar{S}$ such that $\bar{S}-S$ is a smooth
divisor, one can prove Theorem \ref{thm BP2} by the same techniques
by studying the dependence of the above constructions on holomorphic
parameters, i.e. at a point in $D$ we get a nilpotent orbit \[
\theta(z;s_{2},\dots,s_{r})=e^{zN}.F_{\infty}(s_{2},\dots,s_{r})\]
 where $F_{\infty}(s_{2},\dots,s_{r})$ depend holomorphically on
the parameters $(s_{2},\dots,s_{r})$.

\subsection{Zero loci and filtrations on Chow groups}

Returning now to the algebraicity of the Hodge locus discussed at
the beginning of this section, the Hodge Conjecture would further
imply that if $f:X\to S$ can be defined over an algebraically closed
subfield of $\CC$ then so can the germ of $T$. In \cite{Vo1}, C.
Voisin gave sufficient conditions for $T$ to be defined over $\bar{\QQ}$
if $f:X\to S$ is defined over $\QQ$. Very recently F. Charles \cite{Ch}
carried out an analogous investigation of the field of definition
of the zero locus $\Z$ of a normal function motivated over $\FF$.
We reprise this last notion (from $\S\S$1-2):
\begin{defn}
\label{Defn MOTnf}Let $S$ be a smooth quasiprojective variety defined
over a subfield $\FF_{0}\subset\CC$, and let $\FF\subset\CC$ be
a finitely generated extension of $\FF_{0}$. An admissible normal
function $\nu\in\ANF(S,\H)$ is \emph{motivated over $\FF$} if there
exists a smooth quasi-projective variety $\X$, a smooth projective
morphism $f:\X\to S$, and an algebraic cycle $\sZ\in Z^{m}(\X)_{prim}$,
all defined over $\FF$, such that $\H$ is a subVHS of $(R^{2m-1}f_{*}\ZZ)\otimes\mathcal{O}_{S}$
and $\nu=\nu_{\sZ}$.\end{defn}
\begin{rem}
Here $Z^{m}(\X)_{prim}$ denotes algebraic cycles with homologically
trivial restriction to fibres. One traditionally also assumes $\sZ$
is flat over $S$, but this can always be achieved by restricting
to $U\subset S$ sufficiently small (Zariski open); and then by \cite{S1}
\emph{(i) $\nu_{\sZ_{U}}$ is $\bar{S}$ admissible.} Next, for any
$s_{0}\in S$ one can move $\sZ$ by a rational equivalence to intersect
$X_{s_{0}}$ (hence the $\{X_{s}\}$ for $s$ in an analytic neighborhood
of $s_{0}$) properly, and then use the remarks at the beginning of
\cite{Ki} or \cite[sec. III.B]{GGK} to see that \emph{(ii) $\nu_{\sZ}$
is defined and holomorphic over all of $S$.} Putting \emph{(i)} and
\emph{(ii)} together with \cite[Lemma 7.1]{BFNP}, we see that $\nu_{\sZ}$
is itself admissible.
\end{rem}
Recall that the level of a VHS $\H$ is (for a generic fibre $H_{s}$)
the maximum difference $|p_{1}-p_{2}|$ for $H^{p_{1},q_{1}}$ and
$H^{p_{2},q_{2}}$ both nonzero. A fundamental open question about
motivic normal functions is then:
\begin{conjecture}
\label{Conj ZLfield}(i) \emph{{[}$\mathfrak{ZL}(D,E)${]}} For every
$\FF\subset\CC$ finitely generated over $\bar{\QQ}$, $S/\FF$ smooth
quasi-projective of dimension $D$, and $\H\to S$ VHS of weight $(-1)$
and level $\leq2E-1$, the following holds: $\nu$ motivated $/\FF$
$\implies$ $\Z(\nu)$ is an at most countable union of subvarieties
of $S$ defined over (possibly different) finite extensions of $\FF$.

(ii) \emph{{[}$\widetilde{\mathfrak{ZL}}(D,E)${]}} Under the same
hypotheses, $\Z(\nu)$ is an algebraic subvariety of $S$ defined
over an algebraic extension of $\FF$.
\end{conjecture}
Clearly Theorem \ref{conj CalgZL} and Conjecture $\mathfrak{ZL}(D,E)$
together imply $\widetilde{\mathfrak{ZL}}(D,E)$, but it is much more
natural to phrase some statements (especially Prop. \ref{PropZlL2}
below) in terms of $\mathfrak{ZL}(D,E)$. If true even for $D=1$
(but general $E$), Conj. \ref{Conj ZLfield}(i) would resolve a longstanding
question on the structure of Chow groups of complex projective varieties.
To wit, the issue is whether the second Bloch-Beilinson filtrand and
the kernel of the $AJ$ map must agree; we now wish to describe this
connection. We shall write $\overset{(\sim)}{\mathfrak{ZL}}(D,1)_{alg}$
for the case when $\nu$ is motivated by a family of cycles algebraically
equivalent to zero.

Let $X$ be smooth projective and $m\in\NN$. Denoting {}``$\otimes\QQ$''
by a subscript $\QQ$, we have the two {}``classical'' invariants
$cl_{X,\QQ}:CH^{m}(X)_{\QQ}\to Hg^{m}(X)_{\QQ}$ and $AJ_{X,\QQ}:\ker(cl_{X,\QQ})\to J^{m}(X)_{\QQ}$.
It is perfectly natural both to ask for further Hodge-theoretic invariants
for cycle-classes in $\ker(AJ_{X,\QQ})$, and inquire as to what sort
of filtration might arise from their successive kernels. The idea
of a (conjectural) \emph{system} of decreasing filtrations on the
rational Chow groups of \emph{all} smooth projective varieties $/\CC$,
compatible with the intersection product, morphisms induced by correspondences,
and the algebraic K\"unneth components of the diagonal $\Delta_{X}$,
was introduced by A. Beilinson in \cite{Be}, and independently by
S. Bloch. (One needs to assume something like the Hard Lefschetz Conjecture
so that these K\"unneth components exist; the compatibility roughly
says that $Gr^{i}CH^{m}(X)_{\QQ}$ is {}``controlled by $H^{2m-i}(X)$''.)
Such a filtration $F_{\text{BB}}^{\bullet}$ is unique if it exists
and is universally known as a \emph{Bloch-Beilinson filtration} (BBF);
there is a wide variety of constructions which yield a BBF under the
assumption of various more-or-less standard conjectures. The one which
is key for the filtration (due to Lewis \cite{Le2}) we shall consider
is the \emph{arithmetic Bloch-Beilinson Conjecture} (BBC):
\begin{conjecture}
Let $\X/\bar{\QQ}$ be a quasi-projective variety; then the absolute-Hodge
cycle-class map \begin{equation}
c_{\H}:CH^m(\X)_{\QQ} \to H^{2m}_{\H} (\X^{an}_{\CC},\QQ(m)) \\
\end{equation} is injective. (Here $CH^{m}(\X)[\neq CH^{m}(\X_{\CC})]$ denotes
$\equiv_{rat}$-classes of cycles $/\bar{\QQ}$.)
\end{conjecture}
Now for $X/\CC$, $c_{\H}$ on $CH^{m}(X)_{\QQ}$ is far from injective
(the kernel usually not even being parametrizable by an algebraic
variety); but any given cycle $Z\in Z^{m}(X)$ (a priori defined $/\CC$)
is in fact defined over a subfield $K\subset\CC$ finitely generated
$/\bar{\QQ}$, say of transcendence degree $t$. Considering $X,Z$
over $K$, the $\bar{\QQ}$-spread then provides
\begin{itemize}
\item a smooth projective variety $\bar{S}/\bar{\QQ}$ of dimension $t$,
with $\bar{\QQ}(\bar{S})\overset{\cong}{\to}K$ and $s_{0}:\text{Spec}(K)\to\bar{S}$
the corresponding generic point;
\item a smooth projective variety $\bar{\X}$ and projective morphism $\bar{\pi}:\bar{\X}\to\bar{S}$,
both defined $/\bar{\QQ}$, such that $X=X_{s_{0}}:=\X\times_{s_{0}}\text{Spec}(K)$;
and
\item an algebraic cycle $\bar{\sZ}\in Z^{m}(\X_{(\bar{\QQ})})$ with $Z=\bar{\sZ}\times_{s_{0}}\text{Spec}(K)$.
\end{itemize}
Writing $\bar{\pi}^{sm}=:\pi:\X\to S$ (and $\sZ:=\bar{\sZ}\cap\X$),
we denote by $U\subset S$ any affine Zariski open subvariety defined
$/\bar{\QQ}$, and put $\X_{U}:=\pi^{-1}(U)$, $\sZ_{U}:=\bar{\sZ}\cap\X_{U}$;
note that $s_{0}$ factors through all such $U$.

The point is that exchanging the field of definition for additional
geometry allows $c_{\H}$ to {}``see'' more; in fact, since we are
over $\bar{\QQ}$, it should now (by BBC) see everything. Now $c_{\H}(\sZ_{U})$
packages cycle-class and Abel-Jacobi invariants together, and the
idea behind Lewis's filtration (and filtrations of M. Saito and Green/Griffiths)
is to split the whole package up into Leray graded pieces with respect
to $\pi$. Miraculously, the $0^{\text{th}}$ such piece turns out
to agree with the fundametal class of $Z$, and the next piece is
the normal function generated by $\sZ_{U}$. The pieces folllowing
that define the so-called \emph{higher} cycle-class and $AJ$ maps. 

More precisely, we have \begin{equation}
\xymatrix{
& CH^m(X_{(K)})_{\QQ} \ar [d]^{\text{spread}}_{\cong} \ar @/_3pc/ [ldd]_{\Psi:=} \\
& \text{im}\{ CH^m(\bar{\X})_{\QQ} \to \underset{{}^U}{\underrightarrow{\lim}} \; CH^m(\X_U)_{\QQ}  \} \ar [d]^{c_{\H}} \\
\underline{H}_{\H}^{2m} : \ar @{=} [r] & \text{im} \left\{  H_{\D}^{2m}(\bar{\X}^{an}_{\CC},\QQ(m)) \to \underset{U}{\underrightarrow{\lim}} \; H^{2m}_{\H}((\X_U)^{an}_{\CC},\QQ(m)) \right\}
}\\
\end{equation} with $c_{\H}$ (hence $\Psi$) conjecturally injective. Lewis \cite{Le2}
defines a Leray filtration $\L^{\bullet}\underline{H}_{\H}^{2m}$
with graded pieces \begin{equation}\label{LewisGRs}
\xymatrix{ 0 \ar [d] \\  \frac{J^0\left( \underset{{}^U}{\underrightarrow{\lim}} \; W_{-1}H^{i-1}(U,R^{2m-i}\pi_*\QQ(m)) \right) }{ \text{im} \left\{ \underset{{}^U}{\underrightarrow{\lim}} \; Hg^0\left( Gr^W_0H^i(U,R^{2m-i}\pi_*\QQ(m))  \right) \right\} } \ar [d]^{\beta} \\  Gr^i_{\L}\underline{H}^{2m}_{\H} \ar [d]^{\alpha} \\ Hg^0 \left( \underset{{}^U}{\underrightarrow{\lim}} \; W_0H^i(U,R^{2m-i}\pi_* \QQ(m)) \right) \ar [d] \\ 0 } \\ 
\normalsize
\end{equation} and sets $\L^{i}CH^{m}(X_{K})_{\QQ}:=\Psi^{-1}(\L^{i}\underline{H}_{\H}^{2m})$.
For $Z\in\L^{i}CH^{m}(X_{K})_{\QQ}$, we put $cl_{X}^{i}(Z):=\alpha(Gr_{\L}^{i}\Psi(Z))$;
if this vanishes then $Gr_{\L}^{i}\Psi(Z)=:\beta(aj_{X}^{i-1}(Z))$,
and vanishing of $cl^{i}(Z)$ and $aj^{i-1}(Z)$ implies membership
in $\L^{i+1}$. One easily finds that $cl_{X}^{0}(Z)$ identifies
with $cl_{X,\QQ}(Z)\in Hg^{0}(X)_{\QQ}$.
\begin{rem}
The arguments of $Hg^{0}$ and $J^{0}$ in (\ref{LewisGRs}) have
canonical and functorial MHS by \cite{Ar}. One should think of the
top term as $Gr_{\L}^{i-1}$ of the lowest-weight part of $J^{m}(\X_{U})$
and the bottom as $Gr_{\L}^{i}$ of the lowest-weight part of $Hg^{m}(\X_{U})$
(both in the limit over $U$).
\end{rem}
Now to get a candidate BBF, Lewis takes \[
\L^{i}CH^{m}(X_{\CC})_{\QQ}:=\underset{\tiny{K\subset\CC\atop \text{f.g.}/\bar{\QQ}}}{\underrightarrow{\lim}}\L^{i}CH^{m}(X_{K})_{\QQ}.\]
Some consequences of the definition of a BBF mentioned above, specifically
the compatibility with the K\"unneth components of $\Delta_{X}$,
include

(a) $F_{\text{BB}}^{0}CH^{m}(X)_{\QQ}=CH^{m}(X)_{\QQ}$, $F_{\text{BB}}^{1}CH^{m}(X)_{\QQ}=CH_{hom}^{m}(X)_{\QQ}$,
\linebreak $F_{BB}^{2}CH^{m}(X)_{\QQ}\subseteq$$\ker(AJ_{X,\QQ})$,
and

(b) $F_{\text{BB}}^{m+1}CH^{m}(X)=\{0\}$;

these are sometimes stated as additional requirements for a BBF.
\begin{thm}
\emph{\cite{Le2}} $\L^{\bullet}$ is intersection- and correspondence-compatible,
and satisfies (a). Assuming BBC, $\L^{\bullet}$ satisfies (b); and
additionally assuming HLC, $\L^{\bullet}$ is a BBF.
\end{thm}
The limits in (\ref{LewisGRs}) inside $J^{0}$ and $Hg^{0}$ stabilize
for sufficiently small $U$; replacing $S$ by such a $U$, we may
consider the normal function $\nu_{\sZ}\in\ANF(S,\H_{\X/S}^{2m-1})$
attached to the $\bar{\QQ}$-spread of $Z$.
\begin{prop}
(i) For $i=1$, $(3)$ becomes \[
0\to J_{fix}^{m}(\X/S)_{\QQ}\to Gr_{\L}^{1}\underline{H}_{\H}^{2m}\to\left(H^{1}(S,R^{2m-1}\pi_{*}\QQ)\right)^{(0,0)}\to0.\]

(ii) For $Z\in CH_{hom}^{m}(X_{K})_{\QQ}$, $cl_{X}^{1}(Z)=[\nu_{\sZ}]_{\QQ}$;
if this vanishes, $aj_{X}^{0}(Z)=AJ_{X}(Z)_{\QQ}\in J_{fix}^{m}(\X/S)_{\QQ}\subset J^{m}(X)_{\QQ}$
$[\implies\L^{2}\subset\ker AJ_{\QQ}]$.
\end{prop}
So for $Z\in CH_{hom}^{m}(X_{K})$ with $\bar{\QQ}$-spread $\sZ$
over $S$, the information contained in $Gr_{\L}^{1}\Psi(Z)$ is (up
to torsion) precisely $\nu_{\sZ}$. Working over $\CC$, $\sZ\cdot X_{s_{0}}=Z$
is the fiber of the spread at a \emph{very general point} $s_{0}\in S(\CC)$:
$\text{trdeg}(\bar{\QQ}(s_{0})/\bar{\QQ})$ is maximal, i.e. equal
to the dimension of $S$. Since $AJ$ is a \emph{transcendental} (rather
than algebraic) invariant, there is no outright guarantee that vanishing
of $AJ_{X}(Z)\in J^{m}(X)$ --- or equivalently, of the normal function
at a very general point --- implies the identical vanishing of $\nu_{\sZ}$
or even $[\nu_{\sZ}]$. To display explicitly the depth of the question:
\begin{prop}
\label{PropZlL2}(i) $\mathfrak{ZL}(1,E)$ $(\forall E\in\NN)$ $\iff$
$\L^{2}CH^{m}(X)_{\QQ}=\ker(AJ_{X,\QQ})$ $(\forall$ sm. proj. $X/\CC)$.

(ii) $\mathfrak{ZL}(1,1)_{alg}$ $\iff$ $\L^{2}CH^{m}(X)_{\QQ}\cap CH_{alg}^{m}(X)_{\QQ}=\ker(AJ_{X,\QQ})\cap CH_{alg}^{m}(X)_{\QQ}$
$(\forall$ sm. proj. $X/\CC)$.
\end{prop}
Roughly speaking, these statements say that {}``sensitivity of the
zero locus (of a cycle-generated normal function) to field of definition''
is equivalent to {}``spreads of homologically- and $AJ$-trivial
cycles give trivial normal functions''. In (ii), the cycles in both
statements are assumed algebraically equivalent to zero.
\begin{proof}
We first remark that for any variety $S$ with field of definition
$\FF$ of minimal transcendence degree, no proper $\bar{\FF}$-subvariety
of $S$ contains (in its complex points) a very general point of $S$.

(i) $(\Longrightarrow):$ Let $\sZ$ be the $\bar{\QQ}$-spread of
$Z$ with $AJ(Z)_{\QQ}=0$, and suppose $Gr_{\L}^{1}\Psi(Z)=Gr_{\L}^{1}c_{\H}(\sZ)$
does not vanish. Taking a $1$-dimensional very general multiple hyperplane
section $S_{0}\subset S$ through $s_{0}$ ($S_{0}$ is {}``minimally''
defined over $k\overset{\Tiny\text{trdeg. }1}{\subseteq}K$), the
restriction $Gr_{\L}^{1}c_{\H}(\sZ_{0})\neq0$ by weak Lefschetz.
Since each $\Z(\nu_{N\sZ_{0}})\subseteq S_{0}$ is a union of subvarieties
defined $/\bar{k}$ and contains $s_{0}$ for some $N\in\NN$, one
of these is all of $S_{0}$ ($\implies Gr_{\L}^{1}\Psi(Z)=0$), a
contradiction. So $Z\in\L^{2}$.

$(\Longleftarrow):$ Let $\X_{0}\to S_{0}$, $\sZ_{0}\in Z^{m}(\X_{0})_{prim}$,
$\dim(S_{0})=1$, all be defined $/k$ and suppose $\Z(\nu_{\sZ_{0}})$
contains a point $s_{0}$ not defined $/\bar{k}$. Spreading this
out over $\bar{\QQ}$ to $\sZ,\X,S\supset S_{0}\ni s_{0}$, we have:
$s_{0}\in S$ is very general, $\sZ$ is the $\bar{\QQ}$-spread of
$Z=\sZ_{0}\cdot X_{s_{0}}$, and $AJ(Z)_{\QQ}=0$. So $Z\in\L^{2}$
$\implies$ $\nu_{\sZ}$ is torsion $\implies$ $\nu_{\sZ_{0}}$ is
torsion. But then $\nu_{\sZ_{0}}$ is zero since it is zero somewhere
(at $s_{0}$). So $\Z(\nu_{\sZ_{0}})$ is either $S_{0}$ or a (necessarily
countable) union of $\bar{k}$-points of $S_{0}$.

(ii) The spread $\sZ$ of $Z_{(s_{0})}\equiv_{alg}0$ has every fiber
$Z_{s}\equiv_{alg}0$, hence $\nu_{\sZ}$ is a section of $J(\H)$,
$\H\subset(R^{2m-1}\pi_{*}\QQ(m))\otimes\mathcal{O}_{S}$ subVHS of
level one (which can be taken to satisfy $H_{s}=(H^{2m-1}(X_{s}))_{h}$
for a.e. $s\in S$). The rest is as in (i).\end{proof}
\begin{rem}
A related candidate BBF which occurs in work of the first author with
J. Lewis \cite[sec. 4]{KL}, is defined via successive kernels of
\emph{generalized} normal functions (associated to the $\bar{\QQ}$-spread
$\sZ$ of a cycle). These take values on very general $(i-1)$-dimensional
subvarieties of $S$ (rather than at points), and have the above $cl^{i}(Z)$
as their topological invariants.
\end{rem}

\subsection{Field of definition of the zero locus}

We shall begin by showing that the equivalent conditions of Prop.
\ref{PropZlL2}(ii) are satisfied; the idea of the argument is due
in part to S. Saito \cite{Sa}. The first paragraph in the following
in fact gives more:
\begin{thm}
$\widetilde{\mathfrak{ZL}}(D,1)_{alg}$ holds for all $D\in\NN$.
That is, the zero locus of any normal function motivated by a family
of cycles $/\FF$ algebraically equivalent to zero, is defined over
an algebraic extension of $\FF$.

Consequently, cycles algebraically- and Abel-Jacobi-equivalent to
zero on a smooth projective variety $/\CC$, lie in the $2^{\text{nd}}$
Lewis filtrand.\end{thm}
\begin{proof}
Consider $\sZ\in Z^{m}(\X)_{prim}$, $f:\X\to S$ defined $/K$ ($K$
f.g.$/\bar{\QQ}$), with $Z_{s}\equiv_{alg}0$ $\forall s\in S$;
and let $s_{0}\in\Z(\nu_{\sZ})$. (Note: $s_{0}$ is just a complex
point of $S$.) We need to show: \begin{equation}\label{GalConjZL}\exists N \in \NN \text{ such that for any } \sigma \in Gal(\CC/K),\; \sigma(s_0)\in\Z(\nu_{N\sZ}). \\
\end{equation} 

Here is why (\ref{GalConjZL}) will suffice: the analytic closure
of the set of all conjugate points is simply the point's $K$-spread
$S_{0}(\subset S)$, a (possibly reducible) algebraic subvariety defined
$/K$. Clearly, on the $s_{0}$-connected component of $S_{0}$, $\nu_{\sZ}$
itself then vanishes; and this component is defined over an algebraic
extension of $K$. Trivially, $\Z(\nu_{\sZ})$ is the union of such
connected spreads of its points $s_{0}$; and since $K$ is finitely
generated $/\bar{\QQ}$, there are only countably many subvarieties
of $S_{0}$ defined $/K$ or algebraic extensions thereof. This proves
$\mathfrak{ZL}(D,1)_{alg}$, hence (by Theorem \ref{conj CalgZL})
$\widetilde{\mathfrak{ZL}}(D,1)_{alg}$.

To show (\ref{GalConjZL}), write $X=X_{s_{0}}$, $Z=Z_{s_{0}}$,
and $L(/K)$ for their field of definition. There exist $/L$

$\bullet$ a smooth projective curve $C$ and points $0,q\in C(L)$;

$\bullet$ an algebraic cycle $W\in Z^{m}(C\times X)$ such that $Z=W_{*}(q-0)$;
and

$\bullet$ another cycle $\Gamma\in Z^{1}(J(C)\times C)$ defining
Jacobi inversion.

Writing $\Theta:=W\circ\Gamma\in Z^{m}(J(C)\times X)$, the induced
map\[
[\Theta]_{*}:\, J(C)\to J^{m}(X)_{alg}\left(\subseteq J^{m}(X)_{h}\right)\]
is necessarily a morphism of abelian varieties $/L$; hence the identity
connected component of $\ker([\Theta]_{*})$ is a subabelian variety
of $J(C)$ defined over an algebraic extension $L'\supset L$. Define
$\theta:=\Theta|_{B}\in Z^{m}(B\times X)$, and observe that $[\theta]_{*}:B\to J^{m}(X)_{alg}$
is zero by construction, so that $cl(\theta)\in\L^{2}H^{2m}(B\times X)$.

Now, since $AJ_{X}(Z)=0$, a multiple $b:=N.AJ_{C}(q-0)$ belongs
to $B$, and then $N.Z=\theta_{*}b$. This {}``algebraizes'' the
$AJ$-triviality of $N.Z$: conjugating the $6$-tuple $(s_{0},Z,X,B,\theta,b)$
to $(\sigma(s_{0}),Z^{\sigma}[=Z_{\sigma(s_{0})}],X^{\sigma}[=X_{\sigma(s_{0})}],B^{\sigma},\theta^{\sigma},b^{\sigma}),$
we still have $N.Z^{\sigma}=\theta_{*}^{\sigma}b^{\sigma}$ and $cl(\theta^{\sigma})\in\L^{2}H^{2m}(B^{\sigma}\times X^{\sigma})$
by motivicity of the Leray filtration \cite{Ar}, and this implies
$N.AJ(Z^{\sigma})=[\theta^{\sigma}]_{*}b^{\sigma}=0$ as desired.
\end{proof}
We now turn to the result of \cite{Ch} indicated at the outset of
$\S4.5$. While interesting, it sheds no light on $\mathfrak{ZL}(1,E)$
or filtrations, since the hypothesis that the VHS $\H$ have no global
sections is untenable over a point.
\begin{thm}
\cite[Thm. 3]{Ch} Let $\mathcal{Z}$ be the zero locus of a $k$-motivated
normal function $\nu:S\to J(\H)$. Assume that $\mathcal{Z}$ is algebraic
and $\HH_{\CC}$ has no non-zero global sections over $\mathcal{Z}$.
Then $\mathcal{Z}$ is defined over a finite extension of $k$. \end{thm}
\begin{proof}
Charles's proof of this result uses the $\ell$-adic Abel--Jacobi
map. Alternatively, we can proceed as follows (using, with $\FF=k$,
the notation of Defn. \ref{Defn MOTnf}): take $\Z_{0}\subset\Z(\nu)$
to be an irreducible component (without loss of generality assumed
smooth), and $\sZ_{\zl}$ the restriction of $\sZ$ to $\zl$. Let
$[\sZ_{\zl}]$ and $[\sZ_{\zl}]_{dR}$ denote the Betti and de Rham
fundamental classes of $\sZ_{\zl}$, and $\mathcal{L}$ the Leray
filtration. Then, $Gr_{\mathcal{L}}^{1}[\sZ_{\zl}]$ is the topological
invariant of $[\sZ_{\zl}]$ in $H^{1}(\zl,R^{2m-1}f_{*}\ZZ)$, whereas
$Gr_{\mathcal{L}}^{1}[\sZ_{\zl}]_{dR}$ is the infinitesimal invariant
of $\nu_{\sZ}$ over $\Z_{0}$. In particular, since $\Z_{0}$ is
contained in the zero locus of $\nu_{\sZ}$, \begin{equation}\label{Eqn Charlespf}Gr^j_{\L}[\sZ_{\zl}]_{dR}=0,\; \; j=0,1. \\
\end{equation} Furthermore, by the algebraicity of the Gauss-Manin connection, (\ref{Eqn Charlespf})
is invariant under conjugation: \[
Gr_{\mathcal{L}}^{j}[\sZ_{\zl^{\sigma}}]_{dR}=(Gr_{\mathcal{L}}^{j}[\sZ_{\zl}]_{dR})^{\sigma}\]
 and hence $Gr_{\mathcal{L}}^{j}[\sZ_{\zl^{\sigma}}]_{dR}=0$ for
$j=0$, $1$. Therefore, $Gr_{\mathcal{L}}^{j}[\sZ_{\zl^{\sigma}}]=0$
for $j=0$, $1$, and hence $AJ(Z_{s})$ takes values in the fixed
part of $J(\H)$ for $s\in\mathcal{Z}_{0}^{\sigma}$. By assumption,
$\HH_{\CC}$ has no fixed part over $\mathcal{Z}_{0}$, and hence
no fixed part over $\mathcal{Z}_{0}^{\sigma}$ (since conjugation
maps $\nabla$-flat sections to $\nabla$-flat sections by virtue
of the algebraicity of the Gauss-Manin connection). As such, conjugation
must take us to another component of $\mathcal{Z}$, and hence (since
$\mathcal{Z}$ is algebraic over $\CC$ $\implies$ $\mathcal{Z}$
has only finitely many components), $\mathcal{Z}_{0}$ must be defined
over a finite extension of $k$.
\end{proof}
We conclude with a more direct analogue of Voisin's result \cite[Thm. 0.5(2)]{Vo1}
on the algebraicity of the Hodge locus. If $\V$ is a variation of
mixed Hodge structure over a complex manifold and \[
\alpha\in(\F^{p}\cap\W_{2p}\cap V_{\QQ})_{s_{o}}\]
for some $s_{o}\in S$, then the Hodge locus $T$ of $\alpha$ is
the set of points in $S$ where some parallel translate of $\alpha$
belongs to $\F^{p}$.
\begin{rem}
If $(F,W)$ is a mixed Hodge structure on $V$ and $v\in F^{p}\cap W_{2p}\cap V_{\QQ}$
then $v$ is of type $(p,p)$ with respect to Deligne's bigrading
of $(F,W)$. \end{rem}
\begin{thm}
\label{thm ChVOgeneralization}Let $S$ be a smooth complex algebraic
variety defined over a subfield $k$ of $\CC$, and $\V$ be an admissible
variation of mixed Hodge structure of geometric origin over $S$.
Suppose that $T$ is an irreducible subvariety of $S$ over $\CC$
such that: 

(a) $T$ is an irreducible component of the Hodge locus of some \[
\alpha\in(\F^{p}\cap\W_{2p}\cap\VV_{\QQ})_{t_{o}};\]

(b) $\pi_{1}(T,t_{o})$ fixes only the line generated by $\alpha$. 

Then, $T$ is defined over $\bar{k}$.\end{thm}
\begin{proof}
If $\V\cong\QQ(p)$ for some $p$ then $T=S$. Otherwise, $T$ cannot
be an isolated point without violating $(b)$. Assume therefore that
$\dim T>0$. Over $T$, we can extend $\alpha$ to a flat family of
de Rham classes. By the algebraicity of the Gauss--Manin connection,
the conjugate $\alpha^{\sigma}$ is flat over $T^{\sigma}$. Furthermore,
if $T^{\sigma}$ supports any additional flat families of de Rham
classes, conjugation by $\sigma^{-1}$ gives a contradiction to $(b)$.
Therefore, $\alpha^{\sigma}=\lambda\beta$ where $\beta$ is a $\pi_{1}(T^{\sigma})$-invariant
Betti class on $T^{\sigma}$ which is unique up to scaling. Moreover,
\[
Q(\alpha,\alpha)=Q(\alpha^{\sigma},\alpha^{\sigma})=\lambda^{2}Q(\beta,\beta)\]
 and hence there are countably many Hodge classes that one can conjugate
$\alpha$ to via ${\rm Gal}(\CC/k)$. Accordingly, $T$ must be defined
over $\bar{k}$. 
\end{proof}

\section{The N\'eron Model and Obstructions to Singularities}

The unifying theme of the previous sections is the study of algebraic
cycles via degenerations using the Abel--Jacobi map. In particular,
in the case of a semistable degeneration $\pi:X\to\Delta$ and a \emph{cohomologically
trivial} cycle $Z$ which properly intersects the fibers, we have
\[
\lim_{s\to0}\AJ_{X_{s}}(Z_{s})=\AJ_{X_{0}}(Z_{0})\]
as explained in detail in $\S2$. In general however, the existence
of the limit Abel--Jacobi map is obstructed by the existence of the
singularities of the associated normal function. Nonetheless, using
the description of the asymptotic behavior provided by the nilpotent
and $SL_{2}$-orbit theorems, we can define the limits of admissible
normal functions along curves and prove the algebraicity of the zero
locus.

\subsection{N\'eron models in 1 parameter}

In this section we consider the problem of geometrizing these constructions
(ANF's and their singularities, limits and zeroes) by constructing
a N?on model which graphs admissible normal functions. The quest
to construct such objects has a long history which traces back to
the work of N?on on minimal models for abelian varieties $A_{K}$
defined over the field of fractions $K$ of a discrete valuation ring
$R$. In \cite{Na}, Nakamura proved the existence of an analytic
N?on model for a family of abelian varieties $\mathcal{A}\to\Delta^{*}$
arising from a variation of Hodge structure $\H\to\Delta^{*}$ of
level 1 with unipotent monodromy. With various restrictions, this
work was then extended to normal functions arising from higher codimension
cycles in the work of Clemens \cite{Cl2}, El Zein and Zucker \cite{EZ},
and Saito \cite{S1}.
\begin{rem}
Unless otherwise noted, throughout this section we assume that the
local monodromy of the variation of Hodge structure $\mathcal{H}$
under consideration is unipotent, and the local system $\HH_{\ZZ}$
is torsion free.
\end{rem}
A common feature in all of these analytic constructions of N?on models
for variations of Hodge structure over $\Delta^{*}$ is that the fiber
over $0\in\Delta$ is a complex Lie group which has only finitely
many components. Furthermore, the component into which a given normal
function $\nu$ extends is determined by the value of $\sigma_{\ZZ,0}(\nu)$.
Using the methods of the previous section, one way to see this is
as follows: Let \[
0\to\H\to\V\to\ZZ(0)\to0\]
represent an admissible normal function $\nu:\Delta^{*}\to J(\H)$
and $F:U\to\M$ denote the lifting of the period map of $\V$ to the
upper half-plane, with monodromy $T=e^{N}$. Then, using the $SL_{2}$-orbit
theorem of the previous section, it follows (cf. Theorem $(4.15)$
of \cite{Pe2}) that \[
Y_{{\rm Hodge}}=\lim_{{\rm Im}(z)\to\infty}\, e^{-zN}.Y_{(F(z),W)}\]
exists, and is equal to the grading $Y(N,Y_{(F_{\infty},M)})$ constructed
in the previous section; moreover, recall that $Y(N,Y_{(F_{\infty},M)})\in\ker(\mbox{\text{ad}}\, N)$
due to the short length of the weight filtration. Suppose further
that there exists an integral grading $Y_{{\rm Betti}}\in\ker(\text{ad}\, N)$
of the weight filtration $W$. Let $j:\Delta^{*}\to\Delta$ and $i:\{0\}\to\Delta$
denote the inclusion maps. Then, $Y_{{\rm Hodge}}-Y_{{\rm Betti}}$
defines an element in \begin{equation}\label{eqn kerMHS}J(H_0) = Ext^1_{\text{MHS}}(\ZZ(0),H^0(i^*R_{j_*}\H)) \\
\end{equation}by simply applying $Y_{{\rm Hodge}}-Y_{{\rm Betti}}$ to any lift
of $1\in\ZZ(0)=Gr_{0}^{W}$. Reviewing \S2 and \S3, we see that
the obstruction to the existence of such a grading $Y_{{\rm Betti}}$
is exactly the class $\sigma_{\ZZ,0}(\nu)$.
\begin{rem}
More generally, if $\H$ is a variation of Hodge structure of weight
$-1$ over a smooth complex algebraic variety $S$ and $\bar{S}$
is a good compactification of $S$, given a point $s\in\bar{S}$ we
define \begin{equation}\label{eqn neronSEQ2}J(H_s)=Ext^1_{\text{MHS}}(\ZZ,H_s) \\
\end{equation} where $H_{s}=H^{0}(i_{s}^{*}R_{j*}\H)$ and $j:S\to\bar{S}$, $i_{s}:\{s\}\to\bar{S}$
are the inclusion maps. In case $\bar{S}\backslash S$ is a NCD in
a neighborhood of $S$, with $\{N_{i}\}$ the logarithms of the unipotent
parts of the local monodromies, then $H_{s}\cong\cap_{j}\ker(N_{j})$.
\end{rem}
In general, except in the classical case of degenerations of Hodge
structure of level 1, the dimension of $J(H_{0})$ is usually strictly
less than the dimension of the fibers of $J(\H)$ over $\Delta^{*}$.
Therefore, any generalized N?on model $J_{\Delta}(\H)$ of $J(\H)$
which graphs admissible normal functions cannot be a complex analytic
space. Rather, in the terminology of Kato and Usui \cite{KU}\cite{GGK},
we obtain a {}``slit analytic fiber space''. In the case where the
base is a curve, the above observations can be combined into the following
result:
\begin{thm}
Let $\H$ be a variation of pure Hodge structure of weight $-1$ over
a smooth algebraic curve $S$ with projective completion $\bar{S}$.
Let $j:S\to\bar{S}$ denote the inclusion map. Then, there exists
a N?on model for $J(\H)$, i.e. a topological group $J_{\bar{S}}(\H)$
over $\bar{S}$ such that: 

(i) $J_{\bar{S}}(\H)$ restricts to $J(\H)$ over $S$;

(ii) There is a 1-1 correspondence between the set of admissible normal
functions $\nu:S\to J(\H)$ and the set of continuous sections $\bar{\nu}:\bar{S}\to J_{\bar{S}}(\H)$
which restrict to holomorphic, horizontal sections of $J(\H)$ over
$S$; 

Furthermore, 

(iii) There is a short exact sequence of topological groups \[
0\to J_{\bar{S}}(\H)^{0}\to J_{\bar{S}}(\H)\to G\to0\]
 where $G_{s}$ is the torsion subgroup of $(R_{j*}^{1}\H_{\ZZ})_{s}$
for any $s\in\bar{S}$;

(iv) $J_{\bar{S}}(\H)^{0}$ is a slit analytic fiber space, with fiber
$J(H_{s})$ over $s\in\bar{S}$;

(v) If $\nu:S\to J(\H)$ is an admissible normal function with extension
$\bar{\nu}$ then the image of $\bar{\nu}$ in $G_{s}$ at the point
$s\in\bar{S}-S$ is equal to $\sigma_{\ZZ,s}(\nu)$. Furthermore,
if $\sigma_{\ZZ,s}(\nu)=0$ then the value of $\bar{\nu}$ at $s$
is given by the class of $Y_{{\rm Hodge}}-Y_{{\rm Betti}}$ as in
\emph{(\ref{eqn kerMHS})}. Equivalently, in the geometric setting,
if $\sigma_{\ZZ,s}(\nu)=0$ then the value of $\bar{\nu}$ at $s$
is given by the limit Abel--Jacobi map. 
\end{thm}
Regarding the topology of the N?on model, let us consider more generally
the case of smooth complex variety $S$ with good compactification
$\bar{S}$, and recall from \S2 that we have also have the Zucker
extension $J_{\bar{S}}^{Z}(\H)$ obtained by starting from the short
exact sequence of sheaves \[
0\to\HH_{\ZZ}\to\H_{\mathcal{O}}/F^{0}\to J(\H)\to0\]
 and replacing $\HH_{\ZZ}$ by $j_{*}\HH_{\ZZ}$ and $\H_{\mathcal{O}}/F^{0}$
by its canonical extension. Following \cite{S5}, let us suppose that
$D=\bar{S}-S$ is a smooth divisor and $J_{\bar{S}}^{Z}(\H)_{D}^{{\rm Inv}}$
be the subset of $J_{\bar{S}}^{Z}(\H)$ defined by the local monodromy
invariants.
\begin{thm}
\cite{S5} The Zucker extension $J_{\bar{S}}^{Z}(\H)$ has the structure
of a complex Lie group over $\bar{S}$, and it is a Hausdorff topological
space on a neighborhood of $J_{\bar{S}}^{Z}(\H)_{D}^{{\rm Inv}}$.
\end{thm}
Specializing this result to the case where $S$ is a curve, we then
recover the result of the first author together with Griffiths and
Green that $J_{\bar{S}}(\H)^{0}$ is Hausdorff, since in this case
we can identify $J_{\bar{S}}(\H)^{0}$ with $J_{\bar{S}}^{Z}(\H)_{D}^{{\rm Inv}}$.
\begin{rem}
Using this Hausdorff property, Saito was able to prove \cite{S5}
the algebraicity of the zero locus of an admissible normal function
in this setting (i.e., $D$ smooth). 
\end{rem}

\subsection{N\'eron models in many parameters}

To extend this construction further, we have to come to terms with
the fact that unless $S$ has a compactification $\bar{S}$ such that
$D=\bar{S}-S$ is a smooth divisor, the normal functions that we consider
may have non-torsion singularities along the boundary divisor. This
will be reflected in the fact that the fibers $G_{s}$ of $G$ need
no longer be finite groups. The first test case is when $\H$ is a
Hodge structure of level 1. In this case, a N?on model for $J(\H)$
was constructed in the thesis of Andrew Young \cite{Yo}. More generally,
in joint work with Patrick Brosnan and Morihiko Saito, the second
author proved the following result:
\begin{thm}
\label{thm neronBPS}\cite{BPS} Let $S$ be a smooth complex algebraic
variety and $\H$ be a variation of Hodge structure of weight $-1$
over $S$. Let $j:S\to\bar{S}$ be a good compactification of $\bar{S}$
and $\{S_{\alpha}\}$ be a Whitney stratification of $\bar{S}$ such
that: 

(a) $S$ is one of the strata of $\bar{S}$; 

(b) $R^{k}j_{*}\HH_{\ZZ}$ are locally constant on each stratum. 

Then, there exists a generalized N?on model for $J(\H)$, i.e. a
topological group $J_{\bar{S}}(\H)$ over $\bar{S}$ which extends
$J(\H)$ such that: 

(i) The restriction of $J_{\bar{S}}(\H)$ to $S$ is $J(\H)$;

(ii) Any admissible normal function $\nu:S\to J(\H)$ has a unique
extension to a continuous section $\bar{\nu}$ of $J_{\bar{S}}(\H)$;

Furthermore, 

(iii) There is a short exact sequence of topological groups \[
0\to J_{\bar{S}}(\H)^{0}\to J_{\bar{S}}(\H)\to G\to0\]
 over $\bar{S}$ such that $G_{s}$ is a discrete subgroup of $(R^{1}j_{*}\HH_{\ZZ})_{s}$
for any point $s\in\bar{S}$;

(iv) The restriction of $J_{\bar{S}}(\H)^{0}$ to any stratum $S_{\alpha}$
is a complex Lie group over $S_{\alpha}$ with fiber $J(H_{s})$ over
$s\in\bar{S}$.

(v) If $\nu:S\to J(\H)$ is an admissible normal function with extension
$\bar{\nu}$ then the image of $\bar{\nu}(s)$ in $G_{s}$ is equal
to $\sigma_{\ZZ,s}(\nu)$ for all $s\in\bar{S}$. If $\sigma_{\ZZ,s}(\nu)=0$
for all $s\in\bar{S}$ then $\bar{\nu}$ restricts to a holomorphic
section of $J_{\bar{S}}(\H)^{0}$ over each strata.\end{thm}
\begin{rem}
More generally, this is true under the following hypothesis: 

(1) $S$ is a complex manifold and $j:S\to\bar{S}$ is a partial compactification
of $S$ as an analytic space; 

(2) $\H$ is a variation of Hodge structure on $S$ of negative weight,
which need not have unipotent monodromy. 
\end{rem}
To construct the identity component $J_{\bar{S}}(\H)^{0}$, let $\nu:S\to J(\H)$
be an admissible normal function which is represented by an extension
\begin{equation}\label{eqn neronSEQ1}0 \to \H \to \V \to \ZZ(0) \to 0 \\
\end{equation} and $j:S\to\bar{S}$ denote the inclusion map. Also, given $s\in\bar{S}$
let $i_{s}:\{s\}\to\bar{S}$ denote the inclusion map. Then, the short
exact sequence \eqref{eqn neronSEQ1} induces an exact sequence of
mixed Hodge structures \begin{equation}\label{eqn neronSEQ3}0 \to H_s \to H^0(i_s^* R j_* \V) \to \ZZ(0) \to H^1(i_s^* R j_* \H) \\
\end{equation}where the arrow $\ZZ(0)\to H^{1}(i_{s}^{*}Rj_{*}\H)$ is given by
$1\mapsto\sigma_{\ZZ,s}(\nu)$. Accordingly, if $\sigma_{\ZZ,s}(\nu)=0$
then \eqref{eqn neronSEQ3} determines a point $\bar{\nu}(s)\in J(H_{s})$.
Therefore, as a set, we define \[
J_{\bar{S}}(\H)^{0}=\coprod_{s\in\bar{S}}\, J(H_{s})\]
 and topologize by identifying it with a subspace of the Zucker extension
$J_{\bar{S}}^{Z}(\H)$.

Now, by condition $(b)$ of Theorem \eqref{thm neronBPS} and the
theory of mixed Hodge modules\cite{S4}, it follows that if $i_{\alpha}:S_{\alpha}\to\bar{S}$
are the inclusion maps then $H^{k}(i_{\alpha}^{*}Rj_{*}\H)$ are admissible
variations of mixed Hodge structure over each stratum $S_{\alpha}$.
In particular, the restriction of $J_{\bar{S}}(\H)^{0}$ to $S_{\alpha}$
is a complex Lie group.

Suppose now that $\nu:S\to J(\H)$ is an admissible normal function
with extension $\bar{\nu}:\bar{S}\to J_{\bar{S}}(\H)$ such that $\sigma_{\ZZ,s}(\nu)=0$
for each $s\in\bar{S}$. Then, in order to prove that $\bar{\nu}$
is a continuous section of $J_{\bar{S}}(\H)^{0}$ which restricts
to a holomorphic section over each stratum, is is sufficient to prove
that $\bar{\nu}$ coincides with the section of the Zucker extension
(cf. \cite[Prop. 2.3]{S1}). For this, it is in turn sufficient to
consider the curve case by restriction to the diagonal curve $\Delta\to\Delta^{r}$
by $t\mapsto(t,\dots,t)$; see \cite[sec. 1.4]{BPS}.

It remains now to construct $J_{\bar{S}}(\H)$ via the following gluing
procedure: Let $U$ be an open subset of $\bar{S}$ and $\nu:U\to J(\H)$
be an admissible normal function with cohomological invariant \[
\sigma_{\ZZ,U}(\nu)=\partial(1)\in H^{1}(U,\HH_{\ZZ})\]
 defined by the map $\partial:H^{0}(U,\ZZ(0))\to H^{1}(U,\HH_{\ZZ})$
induced by the short exact sequence \eqref{eqn neronSEQ1} over $U$.
Then, we declare $J_{U}(\H_{U\cap S})^{\nu}$ to be the component
of $J_{\bar{S}}(\H)$ over $U$, and equip $J_{U}(\H_{U\cap S})^{\nu}$
with a canonical morphism \[
J_{U}(\H_{U\cap S})^{\nu}\to J_{U}(\H_{U\cap S})^{0}\]
 which sends $\nu$ to the zero section. If $\mu$ is another admissible
normal function over $U$ with $\sigma_{\ZZ,U}(\nu)=\sigma_{\ZZ,U}(\mu)$
then there is a canonical isomorphism \[
J_{U}(\H_{U\cap S})^{\nu}\cong J_{U}(\H_{U\cap S})^{\mu}\]
 which corresponds to the section $\nu-\mu$ of $J_{U}(\H_{U\cap S})^{0}$
over $U$.

\subsection*{Addendum to 5.2}

Since the submission of this article, there have been several important
developments in the theory of N\'eron models for admissible normal
functions on which we would like to report here. To this end, let
us suppose that $\mathcal{H}$ is a variation of Hodge structure of
level 1 over a smooth curve $S\subset\bar{S}$. Let $\mathcal{A}_{S}$
denote the corresponding abelian scheme with N\'eron model $\mathcal{A}_{\bar{S}}$
over $\bar{S}$. Then, we have a canonical morphism \[
\mathcal{A}_{\bar{S}}\to J_{\bar{S}}(\mathcal{H})\]
 which is an isomorphism over $S$. However, unless $\mathcal{H}$
has unipotent local monodromy about each point $s\in\bar{S}-S$, this
morphism is not an isomorphism \cite{BPS}. Recently however, building
upon his work on local duality and mixed Hodge modules \cite{Sl2},
Christian Schnell has found an alternative construction of the identity
component of a N\'eron model which contains the construction of \cite{BPS}
in the case of unipotent local monodromy and agrees \cite{SS} with
the classical N\'eron model for VHS of level 1 in the case of non-unipotent
monodromy. In the paragraphs below, we reproduce a summary of this
construction which has been generously provided by Schnell for inclusion
in this article.

{}``The genesis of the new construction is in unpublished work of
Clemens on normal functions associated to primitive Hodge classes.
When $Y$ is a smooth hyperplane section of a smooth projective variety
$X$ of dimension $2n$, and $H_{\ZZ}=H^{2n-1}(Y,\ZZ)_{van}$ its
vanishing cohomology modulo torsion, the intermediate Jacobian $J(Y)$
can be embedded into a bigger object, $K(Y)$ in Clemens's notation,
defined as \[
K(Y)=\frac{\bigl(H^{0}\bigl(X,\Omega_{X}^{2n}(nY)\bigr)^{\vee}}{H^{2n-1}(Y,\ZZ)_{van}}.\]
 The point is that the vanishing cohomology of $Y$ is generated by
residues of meromorphic $2n$-forms on $X$, with the Hodge filtration
determined by the order of the pole (provided that $\mathcal{O}_{X}(Y)$
is sufficiently ample). Clemens introduced $K(Y)$ with the hope of
obtaining a weak, topological form of Jacobi inversion for its points,
and because of the observation that the numerator in its definition
makes sense \emph{even when $Y$ becomes singular}. In his Ph.D.~thesis
\cite{Sl3}, Schnell proved that residues and the pole order filtration
actually give a filtered holonomic $\mathcal{D}$-module on the projective
space parametrizing hyperplane sections of $X$; and that this $\mathcal{D}$-module
underlies the polarized Hodge module corresponding to the vanishing
cohomology by Saito's theory. At least in the geometric case, therefore,
there is a close connection between the question of extending intermediate
Jacobians, and filtered $\mathcal{D}$-modules (with the residue calculus
providing the link).

{}``The basic idea behind Schnell's construction is to generalize
from the geometric setting above to arbitrary bundles of intermediate
Jacobians. As before, let $\H$ be a variation of polarized Hodge
structure of weight $-1$ on a complex manifold $S$, and $M$ its
extension to a polarized Hodge module on $\Sbar$. Let $(\Mmod,F)$
be its underlying filtered left $\Dmod$-module: $\Mmod$ is a regular
holonomic $\Dmod$-module, and $F=F_{\bullet}\Mmod$ a good filtration
by coherent subsheaves. In particular, $F_{0}\Mmod$ is a coherent
sheaf on $\Sbar$ that naturally extends the Hodge bundle $F^{0}\shHO$.
Now consider the analytic space over $\Sbar$, given by \[
T=T(F_{0}\Mmod)=\Spec_{\Sbar}\bigl(\Sym_{\shO_{\bar{S}}}(F_{0}\Mmod)\bigr),\]
 whose sheaf of sections is $(F_{0}\Mmod)^{\vee}$. (Over $S$, it
is nothing but the vector bundle corresponding to $(F^{0}\shHO)^{\vee}$.)
It naturally contains a copy $\TZ$ of the \'etal\'e space of the
sheaf $j_{\ast}\HH_{\ZZ}$; indeed, every point of that space corresponds
to a local section of $\HH_{\ZZ}$, and it can be shown that every
such section defines a map of $\Dmod$-modules $\Mmod\to\shO_{\bar{S}}$
via the polarization.

{}``Schnell proves that $\TZ\subseteq T$ is a closed analytic subset,
discrete on fibers of $T\to\Sbar$. This makes the fiberwise quotient
space $\Jbar=T/\TZ$ into an analytic space, naturally extending the
bundle of intermediate Jacobians for $H$. He also shows that admissible
normal functions with no singularities extend uniquely to holomorphic
sections of $\Jbar\to\Sbar$. To motivate the extension process, note
that the intermediate Jacobian of a polarized Hodge structure of weight
$-1$ has two models, \[
\frac{H_{\CC}}{F^{0}H_{\CC}+H_{\ZZ}}\simeq\frac{(F^{0}H_{\CC})^{\vee}}{H_{\ZZ}},\]
 with the isomorphism coming from the polarization. An extension of
mixed Hodge structure of the form \begin{equation}\label{eqn SCHNELL}0 \to H \to V \to \ZZ(0) \to 0 \\
\end{equation} \\ gives a point in the second model in the following manner. Let $H^{\ast}=\Hom\bigl(H,\ZZ(0)\bigr)$
be the dual Hodge structure, isomorphic to $H(-1)$ via the polarization.
After dualizing, we have \[
0\to\ZZ(0)\to V^{\ast}\to H^{\ast}\to0,\]
 and thus an isomorphism $F^{1}V_{\CC}^{\ast}\simeq F^{1}H_{\CC}^{\ast}\simeq F^{0}H_{\CC}$.
Therefore, any $v\in V_{\ZZ}$ lifting $1\in\ZZ$ gives a linear map
$F^{0}H_{\CC}\to\CC$, well-defined up to elements of $H_{\ZZ}$;
this is the point in the second model of $J(H)$ that corresponds
to the extension in \eqref{eqn SCHNELL}.

{}``It so happens that this second construction is the one that extends
to all of $\Sbar$. Given a normal function $\nu$ on $S$, let \[
0\to\HH_{\ZZ}\to\VV_{\ZZ}\to\ZZ_{S}\to0\]
 be the corresponding extension of local systems. By applying $j_{\ast}$,
it gives an exact sequence \[
0\to j_{\ast}\HH_{\ZZ}\to j_{\ast}\VV_{\ZZ}\to\ZZ_{\Sbar}\to R^{1}j_{\ast}\HH_{\ZZ},\]
 and when $\nu$ has no singularities, an extension of sheaves \[
0\to j_{\ast}\HH_{\ZZ}\to j_{\ast}\VV_{\ZZ}\to\ZZ_{\Sbar}\to0.\]
 Using duality for filtered $\Dmod$-modules, one obtains local sections
of $(F_{0}\Mmod)^{\vee}$ from local sections of $j_{\ast}\VV_{\ZZ}$,
just as above, and thus a well-defined holomorphic section of $\Jbar\to\Sbar$
that extends $\nu$.

{}``As in the one-variable case, where the observation is due to
Green-Griffiths-Kerr, horizontality constrains such extended normal
functions to a certain subset of $\Jbar$; Schnell proves that this
subset is precisely the identity component of the N\'eron model constructed
by Brosnan-Pearlstein-Saito. With the induced topology, the latter
is therefore a Hausdorff space, as expected. This provides an additional
proof for the algebraicity of the zero locus of an admissible normal
function, similar in spirit to the one-variable result in Saito's
paper, in the case when the normal function has no singularities.''

The other advance, is the recent construction \cite{KNU2} of log
intermediate Jacobians by Kato, Nakayama and Usui. Although a proper
exposition of this topic would take us deep into logarthmic Hodge
theory \cite{KU}, the basic idea is as follows: Let $\mathcal{H}\to\Delta^{*}$
be a variation of Hodge structure of weight $-1$ with unipotent monodromy.
Then, we have a commutative diagram\begin{equation}
\xymatrix{J(\H) \ar [r]^{\tilde{\varphi}} \ar [d] & {\tilde{\Gamma}} \setminus \M \ar [d]^{Gr^W_{-1}} \\ {\Delta}^* \ar [r]^{\varphi} & {\Gamma} \setminus \mathcal{D} } \\
\end{equation} \\where $\tilde{\varphi}$ and $\varphi$ are the respective period
maps. In \cite{KU}, Kato and Usui explained how to translate the
bottom row of this diagram into logarithmic Hodge theory. More generally,
building on the ideas of \cite{KU} and the several variable $SL_{2}$-orbit
theorem \cite{KNU}, Kato, Nakayama and Usui claim to be able to be
able construct a theory of logarthmic mixed Hodge structures which
they can then apply to the top row of the previous diagram. In this
way, they obtain a log intermediate Jacobian which serves the role
of a N\'eron model and allows them to give an alternate proof of
Theorem \ref{conj CalgZL} \cite{KNU3}.

\subsection{Singularities of normal functions overlying nilpotent orbits}

We now consider the group of components $G_{s}$ of $J_{\bar{S}}(\H)$
at $s\in\bar{S}$. For simplicity, we first consider the case where
$\H$ is a nilpotent orbit $\H_{nilp}$ over $(\Delta^{*})^{r}$.
To this end, we recall that in the case of a variation of Hodge structure
$\H$ over $(\Delta^{*})^{r}$ with unipotent monodromy, the intersection
cohomology of $\HH_{\QQ}$ is computed by the cohomology of a complex
$(B^{\bullet}(N_{1},\dots,N_{r}),d)$ (cf. $\S3.4$). Furthermore,
the short exact sequence of sheaves \[
0\to\HH_{\QQ}\to\VV_{\QQ}\to\QQ(0)\to0\]
 associated to an admissible normal function $\nu:(\Delta^{*})^{r}\to J(\H)$
with unipotent monodromy gives a connecting homomorphism \[
\partial:\IH^{0}(\QQ(0))\to\IH^{1}(\HH_{\QQ})\]
 such that \[
\partial(1)=[(N_{1}(e_{o}^{\QQ}),\dots,N_{r}(e_{o}^{\QQ})]=\sing_{0}(\nu)\]
 where $e_{o}^{\QQ}$ is an element in the reference fiber $V_{\QQ}$
of $\VV_{\QQ}$ over $s_{o}\in(\Delta^{*})^{r}$ which maps to $1\in\QQ(0)$.
After passage to complex coefficients, the admissibility of $\V$
allows us to to pick an alternate lift $e_{o}\in V_{\CC}$ to be of
type $(0,0)$ with respect to the limit MHS of $\V$. It also forces
$h_{j}=N_{j}(e_{o})$ to equal $N_{j}(f_{j})$ for some element $f_{j}\in H_{\CC}$
of type $(0,0)$ with respect to the limit MHS of $\H$. Moreover,
$e_{0}^{\QQ}-e_{0}=:h$ maps to $0\in Gr_{0}^{W}$ hence lies in $H_{\CC}$,
so that $(N_{1}(e_{0}^{\QQ}),\ldots,N_{r}(e_{0}^{\QQ}))\equiv(N_{1}(e_{0}),\ldots,N_{r}(e_{0}))$
modulo $d(B^{0})=\text{im}(\oplus_{j=1}^{r}N_{j})$ (i.e. up to $(N_{1}(h),\ldots,N_{r}(h))$).
\begin{cor}
$\sing_{0}(\nu)$ is a rational class of type $(0,0)$ in $\IH^{1}(\HH_{\QQ})$. \end{cor}
\begin{proof}
(Sketch) This follows from the previous paragraph together with the
explicit description of the mixed Hodge structure on the cohomology
of $B^{\bullet}(N_{1},\dots,N_{r})$ given in \cite{CKS2}.
\end{proof}
Conversely, we have the following result:
\begin{lem}
Let $\H_{nilp}=e^{\sum_{j}\, z_{j}N_{j}}.F_{\infty}$ be a nilpotent
orbit of weight $-1$ over $\Delta^{*r}$ with rational structure
$\HH_{\QQ}$. Then, any class $\beta$ of type $(0,0)$ in $\IH^{1}(\HH_{\QQ})$
is representable by a $\QQ$-normal function $\nu$ which is an extension
of $\QQ(0)$ by $\H_{nilp}$ such that $\sing_{0}(\nu)=\beta$.\end{lem}
\begin{proof}
By the above remarks, $\beta$ corresponds to a collection of elements
$h_{j}\in N_{j}(H_{\CC})$ such that 

(a) $h_{1},\dots,h_{r}$ are of type $(-1,-1)$ with respect to the
limit mixed Hodge structure of $\H_{nilp}$; 

(b) $d(h_{1},\dots,h_{r})=0$, i.e. $N_{j}(h_{k})-N_{k}(h_{j})=0$; 

(c) There exists $h\in H_{\CC}$ such that $N_{j}(h)+h_{j}\in H_{\QQ}$
for each $j$, i.e. the class of $(h_{1},\dots,h_{r})$ in $\IH^{1}(\HH_{\CC})$
belongs to the image $\IH^{1}(\HH_{\QQ})\to\IH^{1}(\HH_{\CC})$. 

We now define the desired nilpotent orbit by formally setting $V_{\CC}=\CC e_{o}\oplus H_{\CC}$
where $e_{o}$ is of type $(0,0)$ with respect to the limit mixed
Hodge structure and letting $V_{\QQ}=\QQ(e_{o}+h)\oplus H_{\QQ}$.
We define $N_{j}(e_{o})=h_{j}$. Then, following Kashiwara \cite{Ka}:

(a) The resulting nilpotent orbit $\V_{nilp}$ is pre-admissible; 

(b) The relative weight filtration of \[
W_{-2}=0,\qquad W_{-1}=H_{\QQ},\qquad W_{0}=V_{\QQ}\]
 with respect to each $N_{j}$ exists. 

Consequently $\V_{nilp}$ is admissible, and the associated normal
function $\nu$ has singularity $\beta$ at $0$. 

\end{proof}

\subsection{Obstructions to the existence of normal functions with prescribed
singularity class}

Thus, in the case of a nilpotent orbit, we have a complete description
of the group of components of the Neron model $\otimes\QQ$. In analogy
with nilpotent orbits, one might expect that given a variation of
Hodge structure $\H$ of weight $-1$ over $(\Delta^{*})^{r}$ with
unipotent monodromy, the group of components of the Neron model $\otimes\QQ$
to equal the classes of type $(0,0)$ in $\IH^{1}(\HH_{\QQ})$. However,
Saito \cite{S6} has managed to construct examples of variations of
Hodge structure over $(\Delta^{*})^{r}$ which do not admit any admissible
normal functions with non-torsion singularities. We now want to describe
Saito's class of examples. We begin with a discussion of the deformations
of an admissible nilpotent orbit into an admissible variation of mixed
Hodge structure over $(\Delta^{*})^{r}$.

Let $\varphi:(\Delta^{*})^{r}\to\Gamma\backslash\D$ be the period
map of a variation of pure Hodge structure with unipotent monodromy.
Then, after lifting the period map of $\H$ to the product of upper
half-planes $U^{r}$, the work of Cattani, Kaplan and Schmid on degenerations
of Hodge structure gives us a local normal form of the period map
\[
F(z_{1},\dots,z_{r})=e^{\sum_{j}\, z_{j}N_{j}}e^{\Gamma(s)}.F_{\infty}.\]
 Here, $(s_{1},\dots,s_{r})$ are the coordinates on $\Delta^{r}$,
$(z_{1},\dots,z_{r})$ are the coordinates on $U^{r}$ relative to
which the covering map $U^{r}\to(\Delta^{*})^{r}$ is given by $s_{j}=e^{2\pi iz_{j}}$;
\[
\Gamma:\Delta^{r}\to\mathfrak{g}_{\CC}\]
 is a holomorphic function which vanishes at the origin and takes
values in the subalgebra \[
\mathfrak{q}=\bigoplus_{p<0}\,\mathfrak{g}^{p,q};\]
 and $\oplus_{p,q}\,\mathfrak{g}^{p,q}$ denotes the bigrading of
the MHS induced on $\mathfrak{g}_{\CC}$ (cf. $\S4.2$) by the limit
MHS $(F_{\infty},W(N_{1}+\cdots N_{r})[1])$ of $\H$. The subalgebra
$\mathfrak{q}$ is graded nilpotent \[
\mathfrak{q}=\oplus_{a<0}\,\mathfrak{q}_{a},\qquad\mathfrak{q}_{a}=\oplus_{b}\,\mathfrak{g}^{a,b}\]
 with $N_{1},\dots,N_{r}\in\mathfrak{q}_{-1}$. Therefore, \[
e^{\sum_{j}\, z_{j}N_{j}}e^{\Gamma(s)}=e^{X(z_{1},\dots,z_{r})}\]
 where $X$ takes values in $\mathfrak{q}$, and hence the horizontality
of the period map becomes \[
e^{-X}\pd e^{X}=\pd X_{-1}\]
 where $X=X_{-1}+X_{-2}+\cdots$ relative to the grading of $\mathfrak{q}$.
Equality of mixed partial derivatives then forces \[
\pd X_{-1}\wedge\pd X_{-1}=0\]
 Equivalently, \begin{equation}\label{eqn obs1}\left[ N_j + 2 \pi i s_j \frac{\pd\Gamma_{-1}}{\pd s_j}, N_k + 2\pi i s_k \frac{\pd\Gamma_{-1}}{\pd s_k}\right] = 0 \\
\end{equation}
\begin{rem}
The function $\Gamma$ and the local normal form of the period map
appear in \cite{CK}. 
\end{rem}
In his letter to Morrison \cite{De4}, Deligne showed that for VHS
over $(\Delta^{*})^{r}$ with maximal unipotent boundary points, one
could reconstruct the VHS from data equivalent to the nilpotent orbit
and the function $\Gamma_{-1}$. More generally, one can reconstruct
the function $\Gamma$ starting from $\pd X_{-1}$ using the equation
\[
\pd e^{X}=e^{X}\pd X_{-1}\]
 subject to the integrability condition $\pd X_{-1}\wedge\pd X_{-1}=0$.
This is shown by Cattani and Javier Fernandez in \cite{CF}.

The above analysis applies to VMHS over $(\Delta^{*})^{r}$ as well:
As discussed in the previous section, a VMHS is given by a period
map from the parameter space into the quotient of an appropriate classifying
space of graded-polarized mixed Hodge structure $\M$. As in the pure
case, we have a Lie group $G$ which acts on $\M$ by biholomorphisms
and a complex Lie group $G_{\CC}$ which acts on the {}``compact
dual'' $\check{\M}$.

As in the pure case (and also discussed in $\S4$), an admissible
VMHS with nilpotent orbit $(e^{\sum_{j}\, z_{j}N_{j}}.F_{\infty},W)$
will have a local normal form \[
F(z_{1},\dots,z_{r})=e^{\sum_{j}\, z_{j}N_{j}}e^{\Gamma(s)}.F_{\infty}\]
 where $\Gamma:\Delta^{r}\to\mathfrak{g}_{\CC}$ takes values in the
subalgebra \[
\mathfrak{q}=\bigoplus_{p<0}\,\mathfrak{g}^{p,q}\]
Conversely (given an admissible nilpotent orbit), subject to the integrability
condition \eqref{eqn obs1} above, any function $\Gamma_{-1}$ determines
a corresponding admissible VMHS (cf. \cite[Thm. 6.16]{Pe1}). 

Returning to Saito's examples (which for simplicity we only consider
in the two dimensional case), let $\H$ be a variation of Hodge structure
of weight $-1$ over $\Delta^{*}$ with local normal form $F(z)=e^{zN}e^{\Gamma(s)}.F_{\infty}$.
Let $\pi:\Delta^{2}\to\Delta$ by $\pi(s_{1},s_{2})=s_{1}s_{2}$.
Then for $\pi^{*}(\H)$, we have \[
\Gamma_{-1}(s_{1},s_{2})=\Gamma_{-1}(s_{1}s_{2})\]

In order to construct a normal function, we need to extend $\Gamma_{-1}(s_{1},s_{2})$
and $N_{1}=N_{2}=N$ on the reference fiber $H_{\CC}$ of $\H$ to
include a new class $u_{0}$ of type $(0,0)$ which projects to $1$
in $\ZZ(0)$. Let \[
N_{1}(u_{0})=h_{1},\qquad N_{2}(u_{0})=h_{2},\qquad\Gamma_{-1}(s_{1},s_{2})u_{0}=\alpha(s_{1},s_{2})\]
Note that $(h_{1},h_{2})$ determines the cohomology class of the
normal function so constructed, and that $h_{2}-h_{1}$ depends only
on the cohomology class, and not the particular choice of representative
$(h_{1},h_{2})$.

In order to construct a normal function in this way, we need to check
horizontality. This amounts to checking the equation \begin{eqnarray*}
N\left(s_{2}\frac{\partial\alpha}{\partial s_{2}}-s_{1}\frac{\partial\alpha}{\partial s_{1}}\right) & + & s_{1}s_{2}\Gamma_{-1}'(s_{1}s_{2})(h_{2}-h_{1})\\
 & \hphantom{} & +2\pi is_{1}s_{2}\Gamma_{-1}'(s_{1}s_{2})\left(s_{2}\frac{\partial\alpha}{\partial s_{2}}-s_{1}\frac{\partial\alpha}{\partial s_{1}}\right)=0\end{eqnarray*}

Computation shows that the coefficient of $(s_{1}s_{2})^{m}$ of the
left hand side is \begin{equation}\label{eqn obs4}\frac{1}{(m-1)!} \Gamma_{-1}^{(m)}(0)(h_2-h_1) \\
\end{equation}Therefore, a necessary condition for the cohomology class represented
by $(h_{1},h_{2})$ to arise from an admissible normal function is
for $h_{2}-h_{1}$ to belong to the kernel of $\Gamma_{-1}(t)$. This
condition is also sufficient since under this hypothesis, one can
simply set $\alpha=0$.
\begin{example}
Let $\X\overset{\rho}{\to}\Delta$ be a family of Calabi-Yau $3$-folds
(smooth over $\Delta^{*}$, smooth total space) with Hodge numbers
$h^{3,0}=h^{2,1}=h^{1,2}=h^{0,3}=1$ and central singular fiber having
an ODP. Setting $\H:=\H_{\X^{*}/\Delta^{*}}^{3}(2)$, the LMHS has
as its nonzero $I^{p,q}$'s $I^{-2,1},\, I^{-1,-1},\, I^{0,0},$ and
$I^{1,-2}$. Assume that the Yukawa coupling $(\nabla_{\delta_{s}})^{3}\in Hom_{\mathcal{O}_{\Delta}}(\H_{e}^{3,0},\H_{e}^{0,3})$
is nonzero ($\delta_{s}=s\frac{d}{ds}$), and thus the restriction
of $\Gamma_{-1}(s)$ to \linebreak  $Hom_{\mathcal{O}_{\Delta}}(I^{-1,-1},I^{-2,1})$,
does not vanish identically. Then for any putative singularity class
$0\neq h_{2}-h_{1}\in(I^{-1.-1})_{\QQ}\cong\ker(N)_{\QQ}^{(-1,-1)}$($\cong$(\ref{eqn singtarget codim2})
in this case, which is just one dimensional) for admissible normal
functions overlying $\pi^{*}\H$, non-vanishing of $\Gamma_{-1}(s)(h_{2}-h_{1})$
on $\Delta$ $\implies$ (\ref{eqn obs4}) cannot be zero for every
$m$.
\end{example}

\subsection{Implications for the Griffiths-Green conjecture}

Returning now to the work of Griffiths and Green on the Hodge conjecture
via singularities of normal functions, it follows using the work of
Richard Thomas that for a sufficiently high power of $L$, the Hodge
conjecture implies that one can force $\nu_{\zeta}$ to have a singularity
at a point $p\in\hat{X}$ such that $\pi^{-1}(p)$ has only ODP singularities.
In general, on a neighborhood of such a point $\hat{X}$ need not
be a normal crossing divisor. However, the image of the monodromy
representation is nevertheless abelian. Using a result of Steenbrink
and Nemethi \cite{NS}, it then follows from the properties of the
monodromy cone of a nilpotent orbit of pure Hodge structure that $\text{sing}_{p}(\nu_{\zeta})$
persists under blowup. Therefore, it is sufficient to study ODP degenerations
in the normal crossing case (cf. \cite[sec. 7]{BFNP}). What we will
find below is that the {}``infinitely many'' conditions above (vanishing
of \eqref{eqn obs4} for all $m$) are replaced by surjectivity of
a single logarithmic Kodaira-Spencer map at each boundary component.
Consequently, as suggested in the introduction, it appears that M.
Saito's examples are not a complete show-stopper for existence of
singularities for Griffiths-Green normal functions.

The resulting limit mixed Hodge structure is of the form \begin{eqnarray*}
 & I^{0,0}\\
\cdots\quad I^{-2,1}\qquad I^{-1,0} & \qquad & I^{0,-1}\qquad I^{1,-2}\quad\cdots\\
 & I^{-1,-1}\end{eqnarray*}
 and $N^{2}=0$ for every element of the monodromy cone $\mathcal{C}$.
The weight filtration is given by \[
M_{-2}(N)=\sum_{j}\, N_{j}(H_{\CC}),\qquad M_{-1}(N)=\cap_{j}\,\ker(N_{j}),\qquad M_{0}(N)=H_{\CC}\]

For simplicity of notation, let us restrict to a two parameter version
of such a degeneration, and consider the obstruction to constructing
an admissible normal function with cohomology class represented by
$(h_{1},h_{2})$ as above. As in Saito's example, we need to add a
class $u_{o}$ of type $(0,0)$ such that $N_{j}(u_{o})=h_{j}$ and
construct $\alpha=\Gamma_{-1}(u_{o})$. Then, the integrability condition
$\pd X_{-1}\wedge\pd X_{-1}=0$ becomes \begin{equation} \label{eqn obs3}{
\begin{matrix}
-(2\pi is_2)\frac{\pd\Gamma_{-1}}{\pd s_2}(h_1)       &+&(2\pi is_1)\frac{\pd\Gamma_{-1}}{\pd s_1}(h_2)
\\
&+&(2\pi is_1)(2\pi is_2)\left(        \frac{\pd\Gamma_{-1}}{\pd s_1}\frac{\pd\alpha}{\pd s_2} -         \frac{\pd\Gamma_{-1}}{\pd s_2}\frac{\pd\alpha}{\pd s_1}\right) =0
\end{matrix}
}
\\
\end{equation} since $\alpha=\Gamma_{-1}(u_{o})$ takes values in $M_{-1}(N)$.

Write $\alpha=\sum_{j,k}\, s_{1}^{j}s_{2}^{k}\alpha_{jk}$ and $\Gamma_{-1}=\sum_{p,q}\, s_{1}^{p}s_{2}^{q}\gamma_{pq}$
on $H_{\CC}$. Then, for $ab\neq0$, the coefficient of $s_{1}^{a}s_{2}^{b}$
on the left hand side of equation (\eqref{eqn obs3}) is \[
-2\pi ib\gamma_{ab}(h_{2})+2\pi ia\gamma_{ab}(h_{1})+(2\pi i)^{2}\sum_{p+j=a,q+k=b}\,(pk-qj)\gamma_{pq}(\alpha_{jk})\]
 Define \[
\zeta_{ab}=2\pi ib\gamma_{ab}(h_{2})-2\pi ia\gamma_{ab}(h_{1})-(2\pi i)^{2}\sum_{p+j=a,q+k=b,pq\neq0}\,(pk-qj)\gamma_{pq}(\alpha_{jk})\]
 Then, equation \eqref{eqn obs3} is equivalent to \[
(2\pi i)^{2}b\gamma_{10}(\alpha_{(a-1)b})-(2\pi i)^{2}a\gamma_{01}(\alpha_{a(b-1)})=\zeta_{ab}\]
 where $\alpha_{jk}$ occurs in $\zeta_{ab}$ only in total degree
$j+k<a+b-1$. Therefore, \emph{provided} that \[
\gamma_{10},\gamma_{01}:F_{\infty}^{-1}/F_{\infty}^{0}\to F_{\infty}^{-2}/F_{\infty}^{-1}\]
 are surjective, we can always solve (non-uniquely!) for the coefficients
$\alpha_{jk}$, and hence formally (i.e. modulo checking convergence
of the resulting series) construct the required admissible normal
function with given cohomology class.
\begin{rem}
(i) Of course, it is not necessary to have only ODP singularities
for the above analysis to apply. It is sufficient merely that the
limit mixed Hodge structure have the stated form. In particular, this
is always true for degenerations of level 1. Furthermore, in this
case $Gr_{F_{\infty}}^{-2}=0$, and hence ($\otimes\QQ$) the group
of components of the N?on model surjects onto the Tate-classes of
type $(0,0)$ in $\IH^{1}(\HH_{\QQ})$. 

(ii) In Saito's examples from $\S5.4$, even if $\Gamma_{-1}'(0)\neq0$,
we will have $\gamma_{01}=0=\gamma_{10}$, since the condition of
being a pullback via $(s_{1},s_{2})\mapsto s_{1}s_{2}$ means $\Gamma_{-1}(s_{1},s_{2})=\sum_{p,q}s_{1}^{p}s_{2}^{q}\gamma_{pq}=\sum_{r}s_{1}^{r}s_{2}^{r}\gamma_{rr}$.\end{rem}
\begin{example}
In the case of a degeneration of Calabi--Yau threefolds with limit
mixed Hodge structure on the middle cohomology (shifted to weight
$-1$) \begin{eqnarray*}
 & I^{0,0}\\
I^{-2,1}\qquad I^{-1,0} & \qquad & I^{0,-1}\qquad I^{1,-2}\\
 & I^{-1,-1}\end{eqnarray*}
 the surjectivity of the partial derivatives of $\Gamma_{-1}$ are
related to the Yukawa coupling as follows: Let \[
F(z)=e^{\sum_{j}\, z_{j}N_{j}}e^{\Gamma(s)}.F_{\infty}\]
 be the local normal form of the period map as above. Then, a global
non-vanishing holomorphic section of the canonical extension of $\F^{1}$
(i.e. of $\F^{3}$ before we shift to weight $-1$) is of the form
\[
\Omega=e^{\sum_{j}\, z_{j}N_{j}}e^{\Gamma(s)}\sigma_{\infty}(s)\]
 where $\sigma_{\infty}:\Delta^{r}\to I^{1,-2}$ is holomorphic and
non-vanishing. Then, the Yukawa coupling of $\Omega$ is given by
\[
Q(\Omega,D_{j}D_{k}D_{\ell}\,\Omega),\qquad D_{a}=\frac{\pd}{\pd z_{a}}.\]
In keeping with the above notation, let $e^{X}=e^{\sum_{j}\, z_{j}N_{j}}e^{\Gamma(s)}$
and $A_{j}=D_{j}\, X_{-1}$. Then, using the 1st Hodge--Riemann bilinear
relation and the fact that $e^{X}$ is an automorphism of $Q$, it
follows that \[
Q(\Omega,D_{j}D_{k}D_{\ell}\,\Omega)=Q(\sigma_{\infty}(s),A_{j}A_{k}A_{\ell}\,\sigma_{\infty}(s))\]
 Moreover (cf. \cite{CK},\cite{Pe1}), the horizontality of the period
map implies that \[
\left[\left.\Gamma_{-1}\right|_{s_{k}=0},N_{k}\right]=0\]
 Using this relation, it then follows that \[
\lim_{s\to0}\,\frac{Q(\Omega,D_{j}D_{k}D_{\ell}\,\Omega)}{(2\pi is_{j})(2\pi is_{k})(2\pi is_{\ell})}=Q(\sigma_{\infty}(0),G_{j}G_{k}G_{\ell}\sigma_{\infty}(0))\]
 for $j\neq k$, where $G_{a}=\frac{\pd\Gamma_{-1}}{\pd s_{a}}(0)$.
In particular, if for each index $j$ there exist indices $k$ and
$\ell$ with $k\neq\ell$ such that the left-hand side of the previous
equation is non-zero then $G_{j}:(F_{\infty}^{-1}/F_{\infty}^{0})\to(F_{\infty}^{-2}/F_{\infty}^{-1})$
is surjective.
\end{example}

\section{Global Considerations: Monodromy of Normal Functions}

Returning to a normal function $\V\in NF^{1}(S,\H)_{\bar{S}}^{ad}$
over a $complete$ base, we want to speculate a bit about how one
might {}``force'' singularities to exist. The (inconclusive) line
of reasoning we shall pursue rests on two basic principles:

(i) maximality of the geometric (global) monodromy group of $\VV$
may be deduced from hypotheses on the torsion locus of $\V$; and

(ii) singularities of $\V$ can be interpreted in terms of the local
monodromy of $\VV$ being sufficiently large.

While it is unclear what hypotheses (if any) would allow one to pass
from global to local monodromy-largeness, the proof of the first principle
is itself of interest as a first application of algebraic groups (the
algebraic variety analogue of Lie groups, originally introduced by
Picard) to normal functions.

\subsection{Background}

Mumford-Tate groups of Hodge structures were introduced by Mumford
\cite{Mu} for pure HS and by Andr\'e \cite{An} in the mixed setting.
Their power and breadth of applicability is not well-known, so we
will first attempt a brief summary. They were first brought to bear
on $H^{1}(A)$ for $A$ an abelian variety, which has led to spectacular
results:
\begin{itemize}
\item Deligne's theorem \cite{De2} that $\QQ$-Bettiness of a class in
$F^{p}H_{dR}^{2p}(A_{k})$ ($k=\bar{k}$) is independent of the embedding
of $k$ into $\CC$ ({}``Hodge $\implies$ absolute Hodge'');
\item the proofs by Hazama \cite{Ha} and Murty \cite{Mr} of the HC for
$A$ {}``nondegenerate'' (MT of $H^{1}(A)$ is maximal in a sense
to be defined below); and
\item the density of special (Shimura) subvarieties in Shimura varieties
and the partial resolution of the Andr\'e-Oort Conjecture by Klingler-Yafaev
\cite{KY}.
\end{itemize}
More recently, MT groups have been studied for higher weight HS's;
one can still use them to define special $\bar{\QQ}$-subvarieties
of (non-Hermitian-symmetric) period domains $D$, which classify polarized
HS's with fixed Hodge numbers (and polarization). In particular, the
$0$-dimensional subdomains --- still dense in $D$ --- correspond
to HS with CM (complex multiplication); that is, with abelian MT group.
One understands these HS well: their irreducible subHS may be constructed
directly from primitive CM types (and have endomorphism algebra equal
to the underlying CM field), which leads to a complete classification;
and their Weil and Griffiths intermediate Jacobians are CM abelian
varieties \cite{Bo}. Some further applications of MT groups include:
\begin{itemize}
\item Polarizable CM-HS are motivic \cite{Ab}; when they come from a CY
variety, the latter often has good modularity properties;
\item Given $H^{*}$ of a smooth projective variety, the level of the MT
Lie algebra furnishes an obstruction to the variety being dominated
by a product of curves \cite{Sc};
\item Transcendence degree of the space of periods of a VHS (over a base
$S$), viewed as a field extension of $\CC(S)$ \cite{An};
\end{itemize}
and specifically in the mixed case:
\begin{itemize}
\item the recent proof \cite{AK} of a key case of the Beilinson-Hodge Conjecture
for semiabelian varieties and products of smooth curves.
\end{itemize}
The latter paper, together with \cite{An} and \cite{De2}, are the
best references for the definitions and properties we now summarize.

To this end, recall that an algebraic group $G$ over a field $k$
is an algebraic variety $/k$ together with $k$-morphisms of varieties
$1_{G}:\, Spec(k)\to G$, {}``multiplication'' $\mu_{G}:G\times G\to G$,
and {}``inversion'' $\imath_{G}:G\to G$ satisfying obvious compatibility
conditions. The latter ensure that for any extension $K/k$, the $K$-points
$G(K)$ form a group.
\begin{defn}
(i) A ($k$-)closed algebraic subgroup $M\leq G$ is one whose underlying
variety is ($k$-)Zariski closed.

(ii) Given a subgroup $\M\leq G(K)$, the $k$-closure of $\M$ is
the smallest $k$-closed algebraic subgroup $M$ of $G$ with $K$-points
$M(K)\geq\M$. 

If $\M:=M(K)$ for an algebraic $k$-subgroup $M\leq G$, then the
$k$-closure of $\M$ is just the $k$-Zariski closure of $\M$ (i.e.
the algebraic variety closure).

But in general, this is not true: instead, $M$ may be obtained as
the $k$-Zariski (algebraic variety) closure of the group generated
by the $k$-spread of $\M$.
\end{defn}
We refer the reader to \cite{Sp} (esp. Chap. 6) for the definitions
of reductive, semisimple, unipotent, etc. in this context (which are
less crucial for the sequel). We will write $DG:=[G,G]\,(\trianglelefteq G)$
for the derived group.

\subsection{Mumford-Tate and Hodge groups}

Let $V$ be a (graded-polarizable) mixed Hodge structure with dual
$V^{\vee}$ and tensor spaces $T^{m,n}V:=V^{\otimes m}\otimes(V^{\vee})^{\otimes n}$
($n,m\in\ZZ_{\geq0}$). These carry natural MHS, and any $g\in GL(V)$
acts naturally on $T^{m,n}V$.
\begin{defn}
(i) A \emph{Hodge $(p,p)$-tensor }is any $\tau\in(T^{m,n}V)_{\QQ}^{(p,p)}$.

(ii) The \emph{MT group} $M_{V}$ (resp. \emph{Hodge group}%
\footnote{\emph{In an unfortunate coincidence of terminology, these are completely
different objects from (though not unrelated to) the finitely generated
abelian groups $Hg^{m}(H)$ discussed in $\S1$.}%
} $M_{V}^{\circ}$) of $V$ is the (largest) $\QQ$-algebraic subgroup
of $GL(V)$ fixing%
\footnote{{}``fixing'' means fixing pointwise; the term for {}``fixing as
a set'' is {}``stabilizing''%
} the Hodge $(0,0)$-tensors $\forall m,n$ (resp. Hodge $(p,p)$-tensors
$\forall m,n,p$). $M_{V}$ respects the weight filtration $W_{\bullet}$
on $V$.

(iii) The weight filtration on $V$ induces one on MT/Hodge: \[
W_{-i}M_{V}^{(\circ)}:=\left\{ \left.g\in M_{V}^{(\circ)}\right|(g-\text{id.})W_{\bullet}V\subset W_{\bullet-i}V\right\} \trianglelefteq M_{V}^{(\circ)}.\]
One has: $W_{0}M_{V}^{(\circ)}=M_{V}^{(\circ)}$; $W_{-1}M_{V}^{(\circ)}$
is unipotent; and $Gr_{0}^{W}M_{V}^{(\circ)}\cong M_{V^{\text{split}}}^{(\circ)}$
($V^{\text{split}}:=\oplus_{\ell\in\ZZ}Gr_{\ell}^{W}V$), cf. \cite{An}.
\end{defn}
Clearly $M_{V}^{\circ}\trianglelefteq M_{V}$; and unless $V$ is
pure of weight $0$, we have $M_{V}/M_{V}^{\circ}\cong\GG_{m}$. If
$V$ has polarization $Q\in Hom_{\text{MHS}}\left(V\otimes V,\QQ(-k)\right)$
for $k\in\ZZ\backslash\{0\}$, then $M_{V}^{\circ}$ is of finite
index in $M_{V}\cap GL(V,Q)$ (where $g\in GL(V,Q)$ means $Q(gv,gw)=Q(v,w)$),
and if in addition $V(=H)$ is pure (or at least split) then both
are reductive. One has in general that $W_{-1}M_{V}\subseteq DM_{V}\subseteq M_{V}^{\circ}\subseteq M_{V}$.
\begin{defn}
(i) If $M_{V}$ is abelian ($\Longleftrightarrow\, M_{V}(\CC)\cong(\CC^{*})^{\times a}$),
$V$ is called a \emph{CM-MHS}. (A subMHS of a CM-MHS is obviously
CM.)

(ii) The endomorphisms $End_{\text{MHS}}(V)$ can be interpreted as
the $\QQ$-points of the algebra $(End(V))^{M_{V}}=:E_{V}$. One always
has $M_{V}\subset GL(V,E_{V})$$\linebreak$($=$centralizer of $E_{V}$);
if this is an equality, then $V$ is said to be \emph{nondegenerate}.
\end{defn}
Neither notion implies the other; however: any CM or nondegenerate
MHS is ($\QQ$-)\emph{split}, i.e. $V\,(=V^{\text{split}})$ is a
direct sum of pure HS in different weights.
\begin{rem}
(a) We point out why CM-MHS are split. If $M_{V}$ is abelian, then
$M_{V}\subset E_{V}$ and so $M_{V}(\QQ)$ consists of morphisms of
MHS. But then any $g\in W_{-1}M_{V}(\QQ)$, hence $g-\text{id.}$,
is a morphism of MHS with $(g-\text{id.})W_{\bullet}\subset W_{\bullet-1}$;
so $g=\text{id.}$, and $M_{V}=M_{V^{\text{split}}}$, which implies
$V=V^{\text{split}}$.

(b) For an arbitrary MHS $V$, the subquotient tensor representations
of $M_{V}$ killing $DM_{V}$ (i.e., factoring through the abelianization)
are CM-MHS. By (a), they are split, so that $W_{-1}M_{V}$ acts trivially;
this gives the inclusion $W_{-1}M_{V}\subseteq DM_{V}$.
\end{rem}
Now we turn to the representation-theoretic point of view on MHS.
Define the algebraic $\QQ$-subgroups $U\subset S\subset GL_{2}$
via their complex points \small
\begin{equation} 
\xymatrix{S(\CC) : \ar @{=} [r] & \left\{ \left. 
{\begin{pmatrix} \alpha & \beta \\ -\beta & \alpha \end{pmatrix} }
\right| {\begin{matrix} \alpha, \beta \in \CC \\ (\alpha,\beta)\neq (0,0) \end{matrix}} \right\}  \ar [r]^{\mspace{120mu} \cong}_{\mspace{120mu} \tiny \text{eigenvalues}} & \CC^* \times \CC^* & \left( z, \frac{1}{z} \right) \\ U(\CC) : \ar @{^(->} [u] \ar @{=} [r] & \left\{ \left. 
{\begin{pmatrix} \alpha & \beta \\ -\beta & \alpha \end{pmatrix} }
\right| \alpha, \beta \in \CC ; \, \alpha^2 + \beta^2 = 1 \right\} \ar [r]^{\mspace{180mu} \cong}  & \CC^* \ar @{^(->} [u] & z \ar @{|->} [u] }
\end{equation} \normalsize where the top map sends $\left(\begin{array}{cc}
\alpha & \beta\\
-\beta & \alpha\end{array}\right)\mapsto(\alpha+i\beta,\alpha-i\beta)=:(z,w)$. (Points in $S(\CC)$ will be represented by the {}``eigenvalues''
$(z,w)$.) Let \[
\varphi:\, S(\CC)\to GL(V_{\CC})\]
 be given by \[
\varphi(z,w)|_{I^{p,q}(H)}:=\text{multiplication by }z^{p}w^{q}\,\,\,(\forall p,q).\]
Note that this map is in general only defined $/\CC$, though in the
pure case it is defined $/\RR$ (and as $S(\RR)\subset S(\CC)$ consists
of tuples $(z,\bar{z})$, one tends not to see precisely the above
approach in the literature). The following useful result%
\footnote{Proof of this, and of Prop. \ref{prop SD} below, will appear in a
work of the first author with P. Griffiths and M. Green.%
} allows one to compute MT groups in some cases.
\begin{prop}
$M_{V}$ is the $\QQ$-closure of $\varphi(S(\CC))$ in $GL(V)$.\end{prop}
\begin{rem}
In the pure ($V=H$) case, this condition can be replaced by $M_{H}(\RR)\supset\varphi(S(\RR))$,
and $M_{H}^{\circ}$ defined similarly as the $\QQ$-closure of $\varphi(U(\RR))$;
unfortunately, for $V$ a non-$\QQ$-split MHS the $\QQ$-closure
of $\varphi(U(\CC))$ is smaller than $M_{H}^{\circ}$.
\end{rem}
Now let $H$ be a pure polarizable HS with Hodge numbers $h^{p,q}$,
and take $D$ (with compact dual $\check{D}$) to be the classifying
space for such. We may view $\check{D}$ as a quasi-projective variety
$/\QQ$ in a suitable flag variety. Consider the subgroup $M_{H,\varphi}^{\circ}\subset M_{H}^{\circ}$
with real points $M_{H,\varphi}^{\circ}(\RR):=(M_{H}^{\circ}(\RR))^{\varphi(S(\RR))}$.
If we view $M_{H}^{\circ}$ as acting on a Hodge flag of $H_{\CC}$
with respect to a (fixed) basis of $H_{\QQ},$then $M_{H,\varphi}^{\circ}$
is the stabilizer of the Hodge flag. This leads to a Noether-Lefschetz-type
substratum in $D$:
\begin{prop}
\label{prop SD}The MT domain \[
D_{H}:=\frac{M_{H}^{\circ}(\RR)}{M_{H,\varphi}^{\circ}(\RR)}\,\left(\subset\frac{M_{H}^{\circ}(\CC)}{M_{H,\varphi}^{\circ}(\CC)}=:\check{D}_{H}\right)\]
classifies HS with Hodge group contained in $M_{H}$, or equivalently
with Hodge-tensor set containing that of $H$. The action of $M_{H}^{\circ}$
upon $H$ embeds $\check{D}_{H}\hookrightarrow\check{D}$ as a quasi-projective
subvariety, defined over an algebraic extension of $\QQ$. The $GL(H_{\QQ},Q)$-translates
of $\check{D}_{H}$ give isomorphic subdomains (with conjugate MT
groups) dense in $\check{D}$.
\end{prop}
A similar definition works for certain kinds of MHS. The trouble with
applying this in the variational setting (which is our main concern
here), is that the {}``tautological VHS'' (or VMHS) over such domains
(outside of a few classical cases in low weight or level) violate
Griffiths transversality hence are not actually VHS. Still, it can
happen that MT domains in non-Hermitian symmetric period domains are
themselves Hermitian symmetric. For instance, taking $Sym^{3}$ of
HS's embeds the classifying space ($\cong\mathfrak{H}$) of (polarized)
weight $1$ HS with Hodge numbers (1,1) into that for weight 3 HS
with Hodge numbers (1,1,1,1).

\subsection{MT groups in the variational setting}

Let $S$ be a smooth quasi-projective variety with good compactification
$\bar{S}$, and $\V\in VMHS(S)_{\bar{S}}^{\text{ad}}$; assume $\V$
is graded-polarized, which means we have $Q\in$\linebreak  $\oplus_{i}Hom_{\text{VMHS}(S)}\left((Gr_{i}^{W}\V)^{\otimes2},\QQ(-i)\right)$
satisfying the usual positivity conditions. The Hodge flag embeds
the universal cover $\hat{S}(\twoheadrightarrow S)$ in a flag variety;
let the \emph{image-point} of $\hat{s}_{0}(\mapsto s_{0})$ be of
maximal transcendence degree. (One might say $s_{0}\in S(\CC)$ is
a {}``very general point in the sense of Hodge''; we are $not$
saying $s_{0}$ is of maximal transcendence degree.) Parallel translation
along the local system $\VV$ gives rise to the monodromy representation
$\rho:\,\pi_{1}(S,s_{0})\to GL(V_{s_{0},\QQ},W_{\bullet},Q)$. Moreover,
taking as basis for $V_{s,\QQ}$ the parallel translate of one for
$V_{s_{0},\QQ}$, $M_{V_{s}}$ is constant on paths (from $s_{0})$
avoiding a countable union $T$ of proper analytic subvarieties of
$S$, where in fact $S^{\circ}:=S\backslash T$ is pathwise connected.
(At points $t\in T$, $M_{V_{t}}\subset M_{V_{s}}$; and even the
MT group of the LMHS $\psi_{\underline{s}}\V$ at $x\in\bar{S}\backslash S$
naturally includes in $M_{V_{s}}$.)
\begin{defn}
(i) $M_{V_{s_{0}}}^{(\circ)}=:M_{\V}^{(\circ)}$ is called the \emph{MT
(Hodge) group of} $\V$. One has $End_{\text{MHS}}(V_{s_{0}})\cong End_{\text{VMHS}(S)}(\V)$,
cf. \cite{PS2}.

(ii) The identity connected component $\Pi_{\V}$ of the $\QQ$-closure
of $\rho(\pi_{1}(S,s_{0}))$ is the geometric monodromy group of $\V$;
it is invariant under finite covers $\tilde{S}\twoheadrightarrow S$
(and semisimple in the split case). \end{defn}
\begin{prop}
(Andr\'e) $\Pi_{\V}\trianglelefteq DM_{\V}$.\end{prop}
\begin{proof}
(Sketch) By a theorem of Chevalley, any closed $\QQ$-algebraic subgroup
of $GL(V_{s_{0}})$ is the stabilizer, for some multitensor $\mathfrak{t}\in\oplus_{i}T^{m_{i},n_{i}}(V_{s_{0},\QQ})$
of $\QQ\left\langle \mathfrak{t}\right\rangle $. For $M_{\V}$, we
can arrange this $\mathfrak{t}_{\V}$ to be $itself$ fixed and lie
in $\oplus_{i}\left(T^{m_{i},n_{i}}(V_{s_{0}})\right)_{\QQ}^{(0,0)}$.
By genericity of $s_{0}$, $\QQ\left\langle \mathfrak{t}_{\V}\right\rangle $
extends to a subVMHS with (again by $\exists$ of $Q$) finite monodromy
group, and so $\mathfrak{t}_{\V}$ is fixed by $\Pi_{\V}$. This proves
$\Pi_{\V}\subset M_{\V}$ (in fact, $\subset M_{\V}^{\circ}$ since
monodromy preserves $Q$). Normality of this inclusion then follows
from the {}``Theorem of the Fixed Part'': the largest constant sublocal
system of any $T^{m,n}(\VV)$ (stuff fixed by $\Pi_{\V}$) is a subVMHS,
hence subMHS at $s_{0}$ and stable under $M_{\V}$.

Now let $M_{\V}^{\text{ab}}:=\frac{M_{\V}}{DM_{\V}}$, $\Pi_{\V}^{\text{ab}}:=\frac{\Pi_{\V}}{\Pi_{\V}\cap DM_{\V}}\subset M_{\V}^{\text{ab}}$
(which is a connected component of the $\QQ$-closure of some $\pi^{\text{ab}}\subset M_{\V}^{\text{ab},\circ}(\ZZ)$),
and (taking a more exotic route than Andr\'e) $V^{\text{ab}}$ be
the (CM)MHS corresponding to a faithful representation of $M_{\V}^{\text{ab}}$.
For each irreducible $H\subset V^{\text{ab}}$, the image $\overline{M_{\V}^{\text{ab}}}$
has integer points $\cong\mathcal{O}_{L}^{*}$ for some CM field $L$,
and $\overline{M_{\V}^{\text{ab},\circ}}(\QQ)\subset L$ consists
of elements of norm $1$ under any embedding. The latter generate
$L$ (a well-known fact for CM fields) but, by a theorem of Kronecker,
have finite intersection with $\mathcal{O}_{L}^{*}$: the roots of
unity. It easily follows from this that $\overline{\Pi_{\V}^{\text{ab}}}$,
hence $\Pi_{\V}^{\text{ab}}$, is trivial.\end{proof}
\begin{defn}
\label{def MT1}Let $x\in\bar{S}$ with neighborhood $(\Delta^{*})^{k}\times\Delta^{n-k}$
in $S$ and local (commuting) monodromy logarithms $\{N_{i}\}$;%
\footnote{Though this has been suppressed so far throughout this paper, one
has $\{N_{i}\}$ and LMHS even in the general case where the local
monodromies $T_{i}$ are only quasi-unipotent, by writing $T_{i}=:(T_{i})_{ss}(T_{i})_{u}$
uniquely as a product of semisimple and unipotent parts (Jordan decomposition)
and setting $N_{i}:=\log((T_{i})_{u})$.%
} define the weight monodromy filtration $M_{\bullet}^{x}:=M(N,W)_{\bullet}$
where $N:=\sum_{i=1}^{k}N_{i}$. In the following we assume a choice
of path from $s_{0}$ to $x$:

(a) Write $\pi_{\V}^{x}$ for the \emph{local monodromy group} in
$GL(V_{s_{0},\ZZ},W_{\bullet},Q)$ generated by the $T_{i}=(T_{i})_{ss}e^{N_{i}}$,
and $\rho^{x}$ for the corresponding representation.

(b) We say that $\V$ is \emph{nonsingular at $x$} if $V_{s_{0}}\cong\oplus_{j}Gr_{j}^{W}V_{s_{0}}$
as $\rho^{x}$-modules. In this case, the condition that $\psi_{\underline{s}}\V\cong\oplus_{j}\psi_{\underline{s}}Gr_{j}^{W}\V$
is independent of the choice of local coordinates $(s_{1},\ldots,s_{n})$
at $x$, and $\V$ is called \emph{semi-split (nonsingular)} at $x$
when this is satisfied.

(c) The $Gr_{i}^{M^{x}}\psi_{\underline{s}}\V$ are always independent
of $\underline{s}$. We say that $\V$ is \emph{totally degenerate
(TD)} at $x$ if these $Gr_{i}^{M}$ are (pure) Tate and \emph{strongly
degenerate (SD)} at $x$ if they are CM-HS. Note that the SD condition
is interesting already for the non-boundary points ($x\in S,\, k=0$).
\end{defn}
We can now generalize results of Andr\'e \cite{An} and Mustafin
\cite{Ms}.
\begin{thm}
\label{thm MT1}If $\V$ is semi-split TD (resp. SD) at a point $x\in\bar{S}$,
then $\Pi_{\V}=M_{\V}^{\circ}$ (resp. $DM_{V}^{\circ}$). \end{thm}
\begin{rem}
Note that semi-split SD at $x\in S$ simply means that $V_{x}$ is
a CM-MHS (this case is done in \cite{An}). Also, if $\Pi_{\V}=M_{\V}^{\circ}$
then in fact $\Pi_{\V}=DM_{\V}^{\circ}=M_{\V}^{\circ}$.\end{rem}
\begin{proof}
Passing to a finite cover to identify $\Pi_{\V}$ and $\overline{\rho(\pi_{1})}$,
if we can show that any invariant tensor $\mathfrak{t}\in\left(T^{m,n}V_{s_{0},\QQ}\right)^{\Pi_{\V}}$
is also fixed by $M_{\V}^{\circ}$ (resp. $DM_{\V}^{\circ}$), we
are done by Chevalley. Now the span of $M_{\V}^{\circ}\mathfrak{t}$
is (since $\Pi_{\V}\trianglelefteq M_{\V}^{\circ}$) fixed by $\rho(\pi_{1})$,
and (using the Theorem of the Fixed Part) extends to a constant subVMHS
$\mathcal{U}\subset T^{m,n}\V=:\mathcal{T}$. Now the hypotheses on
$\V$ carry over to $\T$ and taking LMHS at $x$, $\mathcal{U}=\psi_{\underline{s}}\mathcal{U}=\oplus_{i}\psi_{\underline{s}}Gr_{i}^{W}\mathcal{U}=\oplus_{i}Gr_{i}^{W}\mathcal{U}$,
we see that $\mathcal{U}$ splits (as VMHS). As $\mathcal{T}$ is
TD (resp. SD) at $x$, $\mathcal{U}$ is split Hodge-Tate (resp. CM-MHS).

If $\mathcal{U}$ is H-T then it consists of Hodge tensors; so $M_{\V}^{\circ}$
acts trivially on $\mathcal{U}$ hence on $\mathfrak{t}$.

If $\mathcal{U}$ is CM then $M_{\V}^{\circ}|_{\mathcal{U}}=M_{\mathcal{U}}^{\circ}$
is abelian; and so the action of $M_{\V}^{\circ}$ on $\mathcal{U}$
factors through $M_{\V}^{\circ}/DM_{\V}^{\circ}$, so that $DM_{\V}^{\circ}$
fixes $\mathfrak{t}$.
\end{proof}
A reason why one would want this {}``maximality'' result $\Pi_{\V}=M_{\V}^{\circ}$
is to satisfy the hypothesis of the following interpretation of Theorem
\eqref{thm ChVOgeneralization} (which was a partial generalization
of results of \cite{Vo1} and \cite{Ch}). Recall that a VMHS $\V/S$
is $k$-motivated if there is a family $\mathcal{X}\to S$ of quasiprojective
varieties defined $/k$ with $V_{s}=$ the canonical (Deligne) MHS
on $H^{r}(X_{s})$ for each $s\in S$.
\begin{prop}
Suppose $\V$ is motivated over $k$ with trivial fixed part, and
let $T_{0}\subset S$ be a connected component of the locus where
$M_{V_{s}}^{\circ}$ fixes some vector (in $V_{s}$). If $T_{0}$
is algebraic (over $\CC$), $M_{\V_{T_{0}}}^{\circ}$ has \emph{only}
one fixed line, and $\Pi_{\V_{T_{0}}}=M_{\V_{T_{0}}}^{\circ}$, then
$T_{0}$ is defined over $\bar{k}$.
\end{prop}
Of course, to be able to use this one also needs a result on algebraicity
of $T_{0}$, i.e. a generalization of the theorems of \cite{CDK}
and \cite{BP3} to arbitrary VMHS.

\subsection{MT groups of (higher) normal functions}

We now specialize to the case where $\V\in NF^{r}(S,\H)_{\bar{S}}^{\text{ad}}$,
with $\H\to S$ the underlying VHS of weight $-r$. $M_{\V}^{\circ}$
is then an extension of $M_{\H}^{\circ}\cong M_{\V^{\text{split}}(=\H\oplus\QQ_{S}(0))}^{\circ}$
by (in fact, semi-direct product with) an additive (unipotent) group
$U:=W_{-r}M_{\V}^{\circ}\cong\GG_{a}^{\times\mu}$, with $\mu\leq\text{rank}\HH$.
Since $M_{\V}^{\circ}$ respects weights, there is a natural map $M_{\V}^{\circ}\overset{\eta}{\twoheadrightarrow}M_{\H}^{\circ}$
and one might ask when this is an isomorphism.
\begin{prop}
\label{prop MT1}$\mu=0$ $\Longleftrightarrow$ $\V$ is torsion.\end{prop}
\begin{proof}
First we note that $\V$ torsion $\Longleftrightarrow$ after a finite
cover $\tilde{S}\twoheadrightarrow S$, \[
\{0\}\neq Hom_{\text{VMHS}(\tilde{S})}(\QQ_{S}(0),\V)=End_{\text{VMHS}(\tilde{S})}(\V)\cap\text{ann}(\H)=\]
\[
End_{\text{MHS}}(V_{s_{0}})\cap\text{ann}(H_{s_{0}})=\left(Hom_{\QQ}\left((V_{s_{0}}/H_{s_{0}}),V_{s_{0}}\right)\right)^{M_{\V}^{\circ}}.\]
 The last expression can be interpreted as vectors $\underline{w}\in H_{s_{0},\QQ}$
satisfying $(\text{id.}-M)\underline{w}=\underline{u}$ $\forall\left(\begin{array}{cc}
1 & 0\\
\underline{u} & M\end{array}\right)\in M_{\V}^{\circ}$. This is possible only if there is one $\underline{u}$ for each
$M$, i.e. if $M_{\V}^{\circ}\overset{\eta}{\to}M_{\H}^{\circ}$ is
an isomorphism. Conversely, assuming this, write $\underline{u}=\eta^{-1}(M)$
{[}noting $\tilde{\eta}^{-1}(M_{1}M_{2})\overset{(*)}{=}\tilde{\eta}^{-1}(M_{1})+M_{1}\tilde{\eta}^{-1}(M_{2})${]}
and set $\underline{w}:=\tilde{\eta}^{-1}(0)$. Taking $M_{2}=0$,
$M_{1}=M$ in $(*)$, we get $(\text{id.}-M)\underline{w}=(\text{id.}-M)\eta^{-1}(0)\overset{(*)}{=}\tilde{\eta}^{-1}(M)\,(=\underline{u})$
$\forall M\in M_{\H}^{\circ}$.
\end{proof}
We can now address the problem which lies at the heart of this section:
what can one say about the monodromy of the normal function above
and beyond that of the underlying VHS --- for example, about the kernel
of the natural map $\Pi_{\V}\overset{\Theta}{\twoheadrightarrow}\Pi_{\H}$?
One can make some headway simply by translating Definition \ref{def MT1}
and Theorem \ref{thm MT1} into the language of normal functions;
all vanishing conditions are $\otimes\QQ$.
\begin{prop}
\label{prop MT2}Let $\V$ be an admissible higher normal function
over $S$, and let $x\in\bar{S}$ with local coordinate system $\underline{s}$.

(i) $\V$ is non-singular (as AVMHS) at $x$ $\iff$ $sing_{x}(\V)=0$.
Assuming this, $\V$ is semi-simple at $x$ $\iff$ $lim_{x}(\V)=0$.
(In case $x\in S$, $sing_{x}(\V)=0$ is automatic and $lim_{x}(\V)=0$
$\iff$ $x\in$torsion locus of $\V$.)

(ii) $\V$ is TD (resp. SD) at $x$ $\iff$ the underlying VHS $\H$
is. (For $x\in S$, this just means that $H_{x}$ is CM.)

(iii) If $sing_{x}(\V)$, $lim_{x}(\V)$ vanish and $\psi_{\underline{s}}\H$
is graded CM, then $\Pi_{\V}=DM_{\V}$. (For $x\in S$, we are just
hypothesizing that the torsion locus of $\V$ contains a CM point
of $\H$.)

(iv) Let $x\in\bar{S}\backslash S$. If $sing_{x}(\V)$, $lim_{x}(\V)$
vanish and $\psi_{\underline{s}}\H$ is Hodge-Tate, then $\Pi_{\V}=M_{\V}^{\circ}$.

(v) Under the hypotheses of (iii) and (iv), $\dim(\ker(\Theta))=\mu$.
(In general one has $\leq$.)\end{prop}
\begin{proof}
self-evident except for (v), which follows from observing (in both
cases (iii) and (iv)) via the diagram\begin{equation} \xymatrix{\GG_a^{\times \mu} \cong W_{-1}M_{\V}^{(\circ)} = \ker(\eta) \subseteq DM_{\V} \ar @{=} [r] & \Pi_\V \ar @{^(->} [r] \ar @{->>} [d]^{\Theta} & M_{\V}^{(\circ)} \ar @{->>} [d]^{\eta} \\ & \Pi_{\H} \ar @{^(->} [r] & M_{\H}^{(\circ)} \, , } 
\end{equation} that $\ker(\eta)=\ker(\Theta)$.\end{proof}
\begin{example}
The Morrison-Walcher normal function from $\S1.7$ (Ex. \ref{ex MW})
lives {}``over'' the VHS $\H$ arising from $R^{3}\pi_{*}\ZZ(2)$
for a family of {}``mirror quintic'' CY 3-folds, and vanishes at
$z=\infty$. (One should take a suitable, e.g. order 2 or 10 pullback
so that $\V$ is well-defined.) The underlying HS $H$ at this point
is of CM type (the fiber is the usual $(\ZZ/5\ZZ)^{3}$ quotient of
$\{\sum_{i=0}^{4}Z_{i}^{5}=0\}\subset\PP^{4}$), with $M_{H}(\QQ)\cong\QQ(\zeta_{5})$.
So $\V$ would satisfy the conditions of Prop. \ref{prop MT2}(iii).
It should be interesting to work out the consequences of the resulting
equality $\Pi_{\V}=DM_{\V}$.
\end{example}
There is a different aspect to the relationship between local and
global behavior of $\V$. Assuming for simplicity that the local monodromies
at $x$ are unipotent, let $\kappa_{x}:=\ker(\pi_{\V}^{x}\twoheadrightarrow\pi_{\H}^{x})$
denote the local monodromy kernel, and $\mu_{x}$ the dimensions of
its $\QQ$-closure $\overline{\kappa_{x}}$. This is an additive (torsion-free)
subgroup of $\ker(\Theta)$, and so $\dim(\ker(\Theta))\geq\mu_{x}$
($\forall x\in\bar{S}\backslash S$). Writing $\{N_{i}\}$ for the
local monodromy logarithms at $x$, we have the
\begin{prop}
\label{prop MT3}(i) $\mu_{x}>0$ $\implies$ $sing_{x}(\V)\neq0$
(nontorsion singularity)

(ii) The converse holds assuming $r=1$ and $\text{rank}(N_{i})=1$
$(\forall i)$.\end{prop}
\begin{proof}
Let $g\in\pi_{\V}^{x}$, and define $\underline{m}\in\QQ^{\oplus k}$
by $\log(g)=:\sum_{i=1}^{k}m_{i}N_{i}$. Writing $\bar{g}$, $\bar{N}_{i}$
for $g|_{\HH}$, $N_{i}|_{H}$, consider the (commuting) diagram of
morphisms of MHS \begin{equation}
\xymatrix{& \psi_{\underline{s}}\H \ar [ld]_{\oplus\bar{N}_i}  \ar @{^(->} [rd] \ar @/^10pc/ [dd]^{\log(\bar{g})} \\ \oplus_i \psi_{\underline{s}}\H(-1) \ar [rd]_{\chi} & & \psi_{\underline{s}}\V \ar [ll]_{\oplus N_i} \ar [ld]^{\log(g)} \\ & \psi_{\underline{s}}\H(-1) }
\end{equation} where $\chi(\underline{w_{1}},\ldots,\underline{w_{k}})=\sum_{i=1}^{k}m_{i}\underline{w_{i}}$,
$\log(g)=\sum_{i=1}^{k}m_{i}\bar{N}_{i}$. We have that $0\neq sing_{x}(\V)\,\iff\,(\oplus N_{i})\nu_{\QQ}\notin\text{im}(\oplus\bar{N}_{i})$
where $\nu_{\QQ}$ (cf. Definition 2(b)) generates $\psi_{\underline{s}}\V/\psi_{\underline{s}}\H$.

(i) If $g\in\kappa_{x}\backslash\{1\}$ then $0=\log(\bar{g})$ $\implies$$0=\chi(\text{im}(\oplus\bar{N}_{i}))$
while $0\neq\log g$ $\implies$$0\neq(\log(g))\nu_{\QQ}=\chi((\oplus N_{i})\nu_{\QQ}).$
So $\chi$ {}``detects'' a singularity.

(ii) If $r=1$ we may replace $\oplus_{i=1}^{k}\psi_{\underline{s}}\H(-1)$
in the diagram by the subspace $\oplus_{i=1}^{k}(N_{i}(\psi_{\underline{s}}\H)).$
Since each summand is of dimension 1, and (by assumption) $(\oplus N_{i})\nu_{\QQ}\notin\text{im}(\oplus\bar{N}_{i}),$
we can choose $\underline{m}=\{m_{i}\}$ in order that $\chi$kill
$\text{im}(\oplus\bar{N}_{i})$ but not $(\oplus N_{i})\nu_{\QQ}$.
Using the diagram, $\log(\bar{g})=0\neq\log(g)$ $\implies$ $g\in\kappa_{x}\backslash\{1\}$.\end{proof}
\begin{rem}
(a) The existence of a singularity $always$ implies that $\V$ is
nontorsion, hence $\mu>0$.

(b) In the \cite{GG} situation, we have $r=1$ and rank 1 local monodromy
logarithms; hence (by Prop. \ref{prop MT3}(ii)) the existence of
a singularity $\implies$ $\dim(\ker(\Theta))>0$ (consistent with
(a)).

(c) By Prop. \ref{prop MT3}(i), in the normal function case ($r=1$),
$\mu_{x}=0$ along codim. 1 boundary components.

(d) In the {}``maximal geometric monodromy'' situation of Prop.
\ref{prop MT2}(v), $\mu\geq\mu_{x}$ $\forall x\in\bar{S}\backslash S$.
\end{rem}
Obviously, for the purpose of forcing singularities to exist, the
inequality in (d) points in the wrong direction. One wonders if some
sort of cone or spread on a VMHS might be used to translate global
into local monodromy kernel, but this seems unlikely to be helpful.

We conclude with an amusing application of differential Galois theory
related to a result of Andr\'e \cite{An}:
\begin{prop}
Consider a normal function $\V$ of geometric origin together with
an $\mathcal{O}_{S}$-basis $\{\omega_{i}\}$ of holomorphic sections
of $\F^{0}\H$. (That is, $V_{s}$ is the extension of MHS corresponding
to $AJ(Z_{s})\in J^{p}(X_{s})$ for some flat family of cycles on
a family of smooth projective varieties over $S$.) Let $K$ denote
the extension of $\CC(S)$ by the (multivalued) periods of the $\{\omega_{i}\}$;
and $L$ denote the further extension of $K$ via the (multivalued)
Poincar\'e normal functions given by pairing the $\omega_{i}$ with
an integral lift of $1\in\QQ_{S}(0)$ (i.e. the membrane integrals
$\int_{\Gamma_{s}}\omega_{i}(s)$ where $\partial\Gamma_{s}=Z_{s}$).
Then $\text{trdeg}(L/K)=\dim(\ker(\Theta))$.
\end{prop}
The proof rests on a result of N. Katz \cite[Cor. 2.3.1.1]{Ka} relating
transcendence degrees and dimensions of differential Galois groups,
together with the fact that the $\{\int_{\Gamma_{s}}\omega_{i}\}$
(for each $i$) satisfy a homogeneous linear ODE with regular singular
points \cite{Gr1}. (This fact implies equality of differential Galois
and geometric monodromy groups, since monodromy invariant solutions
of such an ODE belong to $\CC(S)$ which is the fixed field of the
Galois group.) In the event that $\H$ has no fixed part (so that
$L$ can introduce no new constants and one has a {}``Picard-Vessiot
field extension'') and the normal function is motivated over $k=\bar{k}$,
one can probably replace $\CC$ by $k$ in the statement.

\end{document}